\documentclass{svproc}
\usepackage{url}

\usepackage{graphicx}
\usepackage{multicol}
\usepackage{footmisc}

\usepackage{amssymb, amsmath, hyperref}
\usepackage{synttree, bussproofs,scrextend, stmaryrd, rotating, xcolor}
\usepackage{esvect, comment}
\usepackage{pgf, color}
\usepackage[all]{xy}
\usepackage{changepage}
\usepackage{enumerate}
\usepackage{proof,upgreek,hyperref}
\usepackage{marvosym,amssymb}

\usepackage[makeroom]{cancel}

\usepackage{tikz}
\usetikzlibrary{positioning, arrows ,calc}
\tikzset{
modal/.style={>=stealth',shorten >=1pt,shorten <=1pt,auto,node distance=1.5cm,
semithick},
world/.style={circle,draw,minimum size=0.5cm,fill=gray!15},
point/.style={circle,draw,inner sep=0.5mm,fill=black},
reflexive above/.style={->,loop,looseness=7,in=120,out=60},
reflexive below/.style={->,loop,looseness=7,in=240,out=300},
reflexive left/.style={->,loop,looseness=7,in=150,out=210},
reflexive right/.style={->,loop,looseness=7,in=30,out=330}
}

\def\imp{\supset}

\def\far{\rightarrow}

\def\langfo{\mathcal{Q}}

\def\ctrrel{(ctr_{\R})}
\def\wk{(wk)}
\def\ctrr{(ctr_{r})}

\def\ctrl{(ctr_{l})}

\def\cut{(cut)}
\def\id{(id)}
\def\idfo{(id_{q})}
\def\idfond{(id_{q}^{n})}
\def\idfocd{(id_{q}^{c})}

\def\botr{(\bot_{l})}
\def\botl{(\bot_{l})}

\newcommand{\disrr}{(\vee_{r})}
\newcommand{\conrr}{(\wedge_{r})}
\newcommand{\disrl}{(\vee_{l})}
\newcommand{\conrl}{(\wedge_{l})}
\newcommand{\imprr}{(\supset_{r})}
\newcommand{\impr}{(\supset_{r})}
\newcommand{\imprl}{(\supset_{l})}
\newcommand{\trans}{(tra)}
\newcommand{\refl}{(ref)}

\def\lift{(lift)}

\def\allr{(\forall_{r})}
\def\alll{(\forall_{l})}
\def\existsl{(\exists_{l})}
\def\existsr{(\exists_{r})}
\def\nd{(nd)}
\def\cd{(cd)}

\def\negl{(\neg_{l})}
\def\negr{(\neg_{r})}
\def\implnew{(\imp_{l}^{*})}
\def\idnew{(id^{*})}

\def\landl{(\land_{l})}

\def\lorr{(\lor_{r})}
\def\impl{(\supset_{l})}

\def\alll{(\forall_{l})}

\def\existsl{(\exists_{l})}
\def\existsr{(\exists_{r})}

\def\lift{(lift)}
\def\idcd{(id^{c}_{q})}
\def\idnd{(id^{n}_{q})}
\def\allcdl{(\forall^{c}_{l})}
\def\allndl{(\forall^{n}_{l})}
\def\allcdr{(\forall^{c}_{r})}
\def\allndr{(\forall^{n}_{r})}
\def\existscdr{(\exists^{c}_{r})}
\def\existsndr{(\exists^{n}_{r})}
\def\existsnedr{(\exists^{in}_{r})}
\def\allnedl{(\forall^{in}_{l})}
\def\existscedr{(\exists^{ic}_{r})}
\def\allcedl{(\forall^{ic}_{l})}
\def\ned{(ihd)}

\def\uc{\forall \vv{x}}
\def\sar{\Rightarrow}

\def\intfo{\mathsf{IntQ}}
\def\intfond{\mathsf{IntQ}}
\def\intfocd{\mathsf{IntQC}}

\def\int{\mathsf{Int}}

\def\lint{\mathsf{G3Int}}
\def\intl{\mathsf{IntL}}

\def\intfondl{\mathsf{IntQL}}
\def\intfocdl{\mathsf{IntQCL}}
\def\nint{\mathsf{NInt}}

\def\lintfo{\mathsf{G3IntQ}}
\def\lintfocd{\mathsf{G3IntQC}}
\def\lintfond{\mathsf{G3IntQ}}
\def\nintfo{\mathsf{NIntQ}}
\def\nintfond{\mathsf{NIntQ}}
\def\nintfocd{\mathsf{NIntQC}}

\newcommand{\lang}{\mathcal{L}}
\def\lang{\mathcal{L}}

\def\lcut{\{}
\def\rcut{\}}

\def\switch{\mathfrak{N}}

\def\prop{\mathrm{Prop}}
\def\pred{\mathrm{Pred}}
\def\unda{\underline{a}}
\def\undb{\underline{b}}
\def\undc{\underline{c}}

\def\switchtwo{\mathfrak{L}}

\def\strucset{\mathrm{LR}} 
\def\nstrucset{\mathrm{NR}} 

\providecommand{\acknowledgments}[1]{\textbf{Acknowledgments. } #1}

\newcommand{\R}{\mathcal{R}}

\newtheorem{innercustomlem}{Lemma}
\newenvironment{customlem}[1]
  {\renewcommand\theinnercustomlem{#1.}\innercustomlem}
  {\endinnercustomlem}
  
\newtheorem{innercustomthm}{Theorem}
\newenvironment{customthm}[1]
  {\renewcommand\theinnercustomthm{#1.}\innercustomthm}
  {\endinnercustomthm}
  
\newtheorem{innercustomcor}{Corollary}
\newenvironment{customcor}[1]
  {\renewcommand\theinnercustomcor{#1.}\innercustomcor}
  {\endinnercustomcor}

\begin{document}
\mainmatter              
\title{On the Correspondence between Nested Calculi and Semantic Systems for Intuitionistic Logics} 
\titlerunning{On the Correspondence between Nested Calculi and Semantic Systems}  
%
\author{Tim Lyon} 
\authorrunning{Tim Lyon} 
%
%
\institute{Institut f\"ur Logic and Computation, Technische Universit\"at Wien, 1040 Wien, Austria  \\ \email{lyon@logic.at}}

\maketitle              

\begin{abstract}

This paper 
studies the relationship between labelled and nested calculi for propositional intuitionistic logic, first-order intuitionistic logic with non-constant domains, and first-order intuitionistic logic with constant domains. 
It is shown that Fitting's nested calculi naturally arise from their corresponding labelled calculi---for each of the aforementioned logics---via the elimination of structural rules in labelled derivations. 
 The translational correspondence between the two types of systems is leveraged to show that the nested calculi inherit proof-theoretic properties from their associated labelled calculi, such as completeness, invertibility of rules, and cut-admissibility. Since labelled calculi are easily obtained via a logic's semantics, the method presented in this paper can be seen as one whereby refined versions of labelled calculi (containing nested calculi as fragments) with favorable properties are derived directly from a logic's semantics.

\end{abstract}

\section{Introduction}\label{Intro}

In his seminal work~\cite{Gen35}, Gentzen introduced the \emph{sequent calculus} framework for classical and intuitionistic logic, and proved the celebrated \emph{Hauptsatz}, i.e. \emph{cut-elimination theorem}, for the systems. As a corollary of his theorem, 
it can be observed that any formula provable in one of Gentzen's systems, is provable with an \emph{analytic derivation}, that is, a derivation where all formulae used to \emph{reach} the conclusion of the derivation, \emph{occur} in the conclusion of the derivation. This method of proof happens to be of practical consequence, and as such, many variations and extensions of Gentzen's sequent calculi have been assembled and proposed---examples include tableaux calculi~\cite{Fit72,Fit14}, display calculi~\cite{Bel82,LyoIttEckGra17,Wan94}, hypersequent calculi~\cite{Avr96,Pog08}, labelled calculi~\cite{Gab96,Sim94,Vig00}, and nested calculi~\cite{Bru09,Bul92,Kas94}. Such calculi have been applied to provide decision algorithms~\cite{Gen35,Pog08}, to automate the extraction of interpolants~\cite{LyoTiuGorClo20}, and to automated counter-model extraction~\cite{LyoBer19,TiuIanGor12}.

In this paper, we focus entirely on the labelled and nested proof-theoretic paradigms. The labelled paradigm may be qualified as \emph{semantic} as calculi are built 
by 
transforming the semantic clauses and Kripke-frame properties of a logic into inference rules~\cite{Sim94,Vig00}. Despite some drawbacks and criticisms of this approach~\cite{Avr96}, the labelled paradigm offers many advantages. First, it is relatively straightforward to transform the semantics of a logic into a calculus; in fact, this process has been shown to be \emph{automatable}~\cite{CiaMafSpe13}. Second, the approach is exceptionally modular---allowing for the addition or deletion of rules to immediately obtain calculi for weaker or stronger logics---and is applicable to a wide variety of logics~\cite{DycNeg12,LyoBer19,Sim94,Vig00}. Last, labelled calculi consistently possess fundamental proof-theoretic properties such as invertibility of rules, admissibility of structural rules, and cut-admissibility---with fairly general results provided for large classes of modal, intuitionistic, and related logics~\cite{DycNeg12,LyoBer19,Sim94,Vig00}. 
Although these characteristics are certainly favorable, a drawback of labelled calculi is that they typically involve a complicated syntax (which incorporates a large amount of semantic information), the sequents utilized in proofs encode general graphs, and inference rules often violate the subformula property (i.e. labelled calculi are not usually analytic). Such properties cause an unnecessary increase in the size of sequents/proofs, and a decrease in the efficiency of associated automated reasoning algorithms.

In contrast to the data structures (called, \emph{labelled sequents}) employed in labelled calculi---which can be viewed as general graphs---the nested paradigm employs treelike data structures (called, \emph{nested sequents}) which manipulate logical information and are used to derive theorems. The inception of the paradigm is often attributed to Bull~\cite{Bul92} and Kashima~\cite{Kas94}, though it should be noted that nested calculi can be considered `upside down' versions of prefixed tableaux calculi, which were introduced much earlier in 1972 by Fitting~\cite{Fit72}. (NB. See~\cite{Fit14} for a discussion on the relationship between nested and tableaux systems). A strength of the nested paradigm is that the nested sequents employed reduce the bureaucracy and syntactic structures appearing in proofs, 
showing the nested formalism to be more parsimonious than the labelled formalism. Also, the nested formalism has proven itself useful in applications such as constructing analytic calculi~\cite{Bru09,Pog09Trends}, developing automated reasoning methods~\cite{TiuIanGor12}, and verifying interpolation~\cite{LyoTiuGorClo20}. 
 Still, in spite of these advantages, the construction of nested calculi and the confirmation of their proof-theoretic properties is often done on a case-by-case basis. That is to say, the nested paradigm lacks general results---like those of the labelled paradigm---regarding the immediate construction of calculi in possession of fundamental properties. 

Due to the fact that the labelled formalism is well-suited for constructing calculi possessing essential proof-theoretic properties, and the nested formalism is more refined and better suited for a variety of applications, a method of extracting nested calculi from labelled calculi---with the properties of the latter preserved---is highly desirable. This \emph{refinement} process allows us to capture the best of both worlds: we invoke the general results of the labelled setting to obtain satisfactory labelled calculi for a class of logics, and via refinement, transform the systems into nested calculi better suited for applications. Similar ideas and relationships have been discussed in the literature~\cite{GorRam12AIML,LyoBer19,Lyo20,Mina13,Pim18}, where refined calculi (which can be considered nested calculi) were derived from labelled calculi for modal, intuitionistic, and related logics. (NB. The paper~\cite{Pim18} mentions results strongly related to Sect.~\ref{section-4}. Although the results presented here were discovered independently,
the work of Sect.~\ref{section-4} can be seen as a detailed explication and expansion of the work presented in~\cite{Pim18}. Moreover,~\cite{Pim18} does not consider the non-trivial and interesting first-order cases considered here.) 

In this paper we advance our understanding of the aforementioned method, and show how to extract \emph{slight} variants of Fitting's nested calculi for propositional intuitionistic logic, first-order intuitionistic logic with non-constant domains, and first-order intuitionistic logic with constant domains from the labelled calculi for these logics. Additionally, we demonstrate the converse translation---showing how to transform each considered nested calculus into its associated labelled calculus. These translations are worthwhile in that they show how each nested calculus inherits properties from its corresponding labelled calculus, and also shed light on how the semantics of each logic affects the shape of rules and syntactic structures inherent in nested derivations (via the labelled calculi).

The organization of this paper is as follows: Sect.~\ref{section-1} introduces the semantics and axiomatizations for the intuitionistic logics we will consider. Sect.~\ref{section-2} introduces the labelled and nested calculi for these logics. Sect.~\ref{section-3} introduces preliminary definitions and lemmata sufficient to translate intuitionistic labelled calculi into nested calculi, and vice-versa. Sect.~\ref{section-4}, Sect.~\ref{section-5}, and Sect.~\ref{section-6} show how to refine each labelled calculus through structural rule elimination, allowing for the extraction of the nested calculus from the labelled calculus for propositional intuitionistic logic, first-order intuitionistic logic with non-constant domains, and first-order intuitionisitic logic with constant domains, respectively. Sect.~\ref{section-7} shows how to translate each nested calculus into its associated labelled calculus, and discusses corollaries of the translational correspondence between the calculi---primarily focusing on the properties inherited by each nested calculus from its corresponding labelled calculus. The last section, Sect.~\ref{conclusion}, concludes. 

This paper serves as an enhanced and revised version of the conference paper~\cite{Lyo20}. Most significantly, the nested-to-labelled translation of Sect.~\ref{section-3}, the content of Sect.~\ref{section-5}, and the majority of content in Sect.~\ref{section-7} are entirely new. Beyond this, the paper was written with an increased focus and more detailed exposition on \emph{how} the labelled calculi are refined in order to extract each nested calculus---this provides the paper with more explanatory force than~\cite{Lyo20}. Also, it should be noted that this paper corrects an error that occurs in the conference version. In~\cite{Lyo20}, the labelled calculus $\lintfocd$ is lacking a structural rule (called $\ned$ in this paper) corresponding to the condition that domains in first-order intuitionistic models are \emph{non-empty}, or \emph{inhabited}. Without the inclusion of this rule, $\lintfocd$ is incomplete as modus ponens cannot be simulated in the calculus (see Thm.~\ref{thm:lint-properties} and Appendix~\ref{app:A} for details). We have included this rule here and adjusted the content of~\cite{Lyo20} regarding $\lintfocd$ accordingly. Furthermore, to increase the flow and readability of the paper, some results outside the scope of, or auxiliary to, our main focus (i.e. the translational correspondence between labelled and nested calculi for intuitionistic logics) have been moved to a technical appendix (Appendix~\ref{app:A}) for the interested reader.

Last, Fig.~\ref{fig:relationships} shows the transformations (indicated by a solid arrow) and translations (indicated by a dotted arrow) between the various intuitionistic calculi considered. Transformations indicate that derivations in one system (or in a fragment of the system) are effectively (i.e. algorithmically) transformable to derivations in another system, and translations indicate a transformation along with a change in notation. The symbols $\switch$ and $\switchtwo$ represent a change from labelled to nested notation, and nested to labelled notation, respectively. The inclusion sign $\subset$ is taken to mean that one calculus consists of a proper subset of the rules in another calculus.

\begin{figure}[t]
\noindent\hrule
\begin{center}
\begin{tabular}{c}
\xymatrix{
	 \lint\ar@/^.8pc/@{->}[rr]^{\text{Thm.~\ref{thm:admiss-all-rules-propositional}}}\ar@{}[d]|{\bigcap}	& &   \intl\ar@/^.8pc/@{.>}[rr]|\switch\ar@/^.8pc/@{->}[ll]^{\text{Lem.~\ref{lem:admiss-rules-lint}}}\ar@{}[d]|{\bigcap} & & \nint\ar@/^.8pc/@{.>}[ll]|\switchtwo\ar@{}[d]|{\bigcap}\ar@{}[ll]|{\text{Thm.~\ref{thm:refinement-prop}}} \\
	 \lintfond\ar@/^.8pc/@{->}[rr]^{\text{Thm.~\ref{thm:admiss-all-rules-fond}}}\ar@{}[d]|{\bigcap}	& &   \intfondl\ar@/^.8pc/@{.>}[rr]|\switch^{\text{Thm.~\ref{thm:Intfondl-to-Nintfond}}}\ar@/^.8pc/@{->}[ll]^{\text{Lem.~\ref{lem:admiss-rules-lintfond}}}\ar@{->}[d]^{\text{Cor.~\ref{cor:embed-Q-to-QC}}} & & \nintfond\ar@/^.8pc/@{.>}[ll]|\switchtwo^{\text{Thm.~\ref{thm:Nintfond-to-Intfondl}}}\ar@{->}[d]^{\text{Cor.~\ref{cor:embed-Q-to-QC}}} \\
		 \lintfocd\ar@/^.8pc/@{->}[rr]^{\text{Thm.~\ref{thm:admiss-all-rules-focd}}}	&  &  \intfocdl\ar@/^.8pc/@{.>}[rr]|\switch^{\text{Thm.~\ref{thm:Intfocdl-to-Nintfocd}}}\ar@/^.8pc/@{->}[ll]^{\text{Lem.~\ref{lem:admiss-rules-lintfocd}}} & & \nintfocd\ar@/^.8pc/@{.>}[ll]|\switchtwo^{\text{Thm.~\ref{thm:Nintfocd-to-Intfocdl}}} \\
}
\end{tabular}
\end{center}
\hrule
\caption{Transformations and translations between the intuitionistic calculi considered.}\label{fig:relationships}
\end{figure}
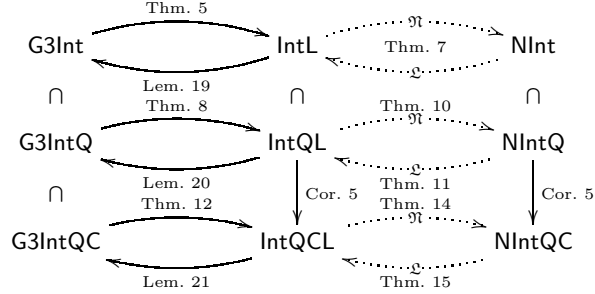

\section{Logical Preliminaries}\label{section-1}

In this section, we introduce the language, semantics, and axiomatizations for propositional intuitionistic logic $\int$, first-order intuitionistic logic with non-constant domains $\intfond$, and first-order intuitionistic logic with constant domains $\intfocd$. The first subsection will discuss the propositional setting, whereas the second subsection will discuss the first-order setting.

\subsection{Propositional Intuitionistic Logic}

The propositional language $\lang$ is defined via the BNF grammar shown below:
$$
A ::= p \ | \ \bot \ | \ (A \vee A) \ | \ (A \wedge A) \ | \ (A \imp A)
$$
where $p$ is among a denumerable set of \emph{propositional variables} $\prop = \{p, q, r, \ldots\}$. As usual, we define intuitionistic negation as $\neg A := A \imp \bot$. Moreover, the language admits a relational (or, Kripke-style) semantics as defined below (cf.~\cite{Kri65}).

\begin{definition}[$\int$-Frame, $\int$-Model] An \emph{$\int$-frame} is an ordered pair $F = (W, \leq)$ such that:
\begin{itemize}

\item[$\blacktriangleright$] $W$ is a non-empty set of worlds $\{w, u, v, \ldots\}$; 

\item[$\blacktriangleright$] $\leq \ \subseteq W \times W$ is a reflexive and transitive binary relation on $W$.\footnote{The properties imposed on $\leq$ are defined as follows: (reflexivity) for all $w \in W$, $w \leq w$, and (transitivity) for all $w, u, v \in W$, if $w \leq v$ and $v \leq u$, then $w \leq u$.}
\end{itemize}

An \emph{$\int$-model} is an ordered pair $M = (F,V)$ where $F$ is an $\int$-frame and $V : \prop \mapsto 2^{W}$ is a \emph{valuation function} mapping propositional variables to subsets of $W$ satisfying the following \emph{monotonicity condition}:
\begin{flushleft}
$\mathbf{(M)}$ \ \ If $w \in V(p)$ and $w \leq v$, then $v \in V(p)$.
\end{flushleft}
\end{definition}


\begin{definition}[Propositional Semantic Clauses]
\label{def:prop-semantic-clauses} Let $M = (W,\leq,V)$ be an $\int$-model with $w \in W$. The \emph{satisfaction relation} $M,w \Vdash A$ between $w \in W$ and a formula $A$ from $\lang$ is inductively defined as follows:

\begin{itemize}

\item[$\blacktriangleright$] $M, w \not\Vdash \bot$;

\item[$\blacktriangleright$] $M,w \Vdash p$ iff $w \in V(p)$;

\item[$\blacktriangleright$] $M,w \Vdash A \vee B$ iff $M,w \Vdash A$ or $M,w \Vdash B$;

\item[$\blacktriangleright$] $M,w \Vdash A \land B$ iff $M,w \Vdash A$ and $M,w \Vdash B$;

\item[$\blacktriangleright$] $M,w \Vdash A \imp B$ iff for all $u \in W$, if $w \leq u$ and $M,u \Vdash A$, then $M,u \Vdash B$.

\end{itemize}


We say that a formula $A$ is \emph{globally true on $M$}, written $M \Vdash A$, iff $M,u \Vdash A$ for all worlds $u \in W$. A formula $A$ is \emph{$\int$-valid}, written $ \Vdash_{\int} A$, iff it is globally true on all $\int$-models. Last, we say that a set $\Phi$ of formulae \emph{semantically implies} a formula $A$, written $\Phi \Vdash A$, iff for all intuitionistic models $M$ with $w \in W$, if $M,w \Vdash B$ for all $B \in \Phi$, then $M,w \Vdash A$.\footnote{We note that we could define a \emph{global} version of semantic implication as follows: A set of formulae $\Phi$ \emph{globally implies} a formula $A$ iff for all intuitionistic models $M$, if $M \Vdash B$ for all $B \in \Phi$, then $M \Vdash A$. We make use of the local version in Def.~\ref{def:prop-semantic-clauses} however, because the axiomatization we use for $\int$ (shown in Fig.~\ref{fig:intuitionistic-axioms}) is sound and complete relative to the local version of semantic implication (Thm.~\ref{thm:sound-complete-Int})~\cite{GabSheSkv09}.}
\end{definition} 

The monotonicity condition $\mathbf{(M)}$, together with the intuitionistic semantics defined above, necessitates a general form of monotonicity as detailed below:

\begin{lemma}[General Monotonicity]\label{lm:prop-gen-monotonicity} 
Let $M$ be a model with $w,v \in W$ of $M$. If $M,w \Vdash A$ and $w \leq v$, then $M,v \Vdash A$.
\end{lemma}

\begin{proof} See~{\cite[Lem.~3.2.16]{GabSheSkv09}} for details.
\qed
\end{proof}

Additionally, propositional intuitionistic logic $\int$ is finitely axiomatizable. The axioms and inference rule syntactically characterizing $\int$ are given in Fig.~\ref{fig:intuitionistic-axioms}. We define an \emph{$\int$-derivation} (relative to the axiomatization of $\int$) from a set of formulae $\Phi$ (written $\Phi \vdash A$) in the usual way (cf.~\cite{GabSheSkv09,TroDal88}). It is well-known that the notion of semantic consequence for $\int$ is equivalent to the syntactic notion of an $\int$-derivation from a set of formulae:

\begin{theorem}[Soundness and Completeness~\cite{GabSheSkv09}]\label{thm:sound-complete-Int} For any $A \in \mathcal{L}$, $\Phi \Vdash A$ iff $\Phi \vdash A$.
\end{theorem}

\begin{figure}[t]
\noindent\hrule
$$
A \supset (B \supset A)
\qquad
(A \imp (B \imp C)) \imp ((A \imp B) \imp (A \imp C))
\qquad
A \supset (B \supset (A \land B))
$$

$$
(A \land B) \supset A
\qquad
(A \land B) \supset B
\qquad
A \supset (A \lor B)
\qquad
\AxiomC{$A$}
\AxiomC{$A \imp B$}
\RightLabel{$(mp)$}
\BinaryInfC{$B$}
\DisplayProof
$$

$$
B \supset (A \lor B)
\qquad
\bot \imp A
\qquad
(A \supset C) \supset ((B \supset C) \supset ((A \lor B) \supset C))
$$

\hrule
\caption{An axiomatization of propositional intuitionistic logic $\int$~\cite{GabSheSkv09}.}\label{fig:intuitionistic-axioms}
\end{figure}



\subsection{First-Order Intuitionistic Logics}

The language $\langfo$ for our first-order logics is defined via the BNF grammar below:
$$
A ::= p(x_{1}, \ldots, x_{n}) \ | \ \bot \ | \ (A \vee A) \ | \ (A \land A) \ | \ (A \imp A) \ | \ (\forall x) A \ | \ (\exists x) A
$$
where $p$ is among a denumerable set of $n$-ary \emph{predicate symbols} $\pred = \{p, q, r, \ldots\}$ and $x_{1}, \ldots, x_{n}, x$ are \emph{variables} (with $n \in \mathbb{N}$). We refer to formulae of the form $p(x_{1}, \ldots, x_{n})$ as \emph{atomic formulae} when $n > 0$, and refer to formulae of the form $p$ as \emph{propositional variables} when $n=0$ (i.e. a $0$-ary predicate $p$ is a propositional variable). Furthermore, we define a variable $x$ to be a \emph{free variable} in $A$ iff it is not within the scope of a quantifier $\forall x$ or $\exists x$, and to be a \emph{bound variable} iff it is within the scope of a quantifier. Last, as in the propositional case, our language admits a relational semantics, defined below. 

\begin{definition}[$\intfo$-Frames, $\intfo$-Models~\cite{GabSheSkv09}]\label{def:FO-frame-model} We define an \emph{$\intfo$-frame} to be a tuple $F = (W,\leq,D)$ such that:
\begin{itemize}

\item[$\blacktriangleright$] $W$ is a non-empty set of worlds $\{w, u, v, \ldots\}$; 

\item[$\blacktriangleright$] $\leq \ \subseteq W \times W$ is a reflexive and transitive binary relation on $W$;

\item[$\blacktriangleright$] $D$ is a \emph{domain function} mapping a world $w \in W$ to a non-empty set $D_{w}$ of \emph{objects} $\{a, b, c, \dots\}$ satisfying the \emph{nested domain condition} shown below:
\end{itemize}
\begin{flushleft}
$\mathbf{(ND)}$ \ \ If $a \in D_{w}$ and $w \leq v$, then $a \in D_{v}$.
\end{flushleft}

A \emph{$\intfocd$-frame} is an \emph{$\intfo$-frame} that additionally satisfies the following \emph{constant domain condition} shown below:
\begin{flushleft}
$\mathbf{(CD)}$ \ \ If $a \in D_{v}$ and $w \leq v$, then $a \in D_{w}$.
\end{flushleft}

An \emph{$\intfo$-model} (\emph{$\intfocd$-model}) $M$ is an ordered pair $(F,V)$ where $F$ is an $\intfo$-frame ($\intfocd$-frame) and $V$ is a \emph{valuation function} such that $V(p,w) \subseteq (D_{w})^{n}$ (with $n \in \mathbb{N}$) satisfying the following \emph{monotonicity condition}:
\begin{flushleft}
$\mathbf{(M)}$ \ \ If $w \in V(p,w)$ and $w \leq v$, then $v \in V(p,v)$ (if $p$ is of arity $0$);\\
\hspace{3em} If $w \leq v$, then $V(p,w) \subseteq V(p,v)$ (if $p$ is of arity $n > 0$).
\end{flushleft}
We uphold the convention in~\cite{GabSheSkv09} and assume that for each world $w \in W$, $(D_{w})^{0} = \{w\}$, so $V(p,w) = \{w\}$ or $V(p,w) = \emptyset$, for a propositional variable $p$.
\end{definition}

As in~\cite{GabSheSkv09}, we forgo the direct interpretation of formulae from $\langfo$ on relational models, and instead, introduce $D_{w}$-sentences (Def.~\ref{def:d-sentence}) to be interpreted on such models. Defining satisfaction relative to $D_{w}$-sentences gives rise to a notion of validity for formulae in $\langfo$ (Def.~\ref{def:semantic-clauses-fo}). However, this notion of validity also depends on the \emph{universal closure} of a formula: given that $A \in \langfo$ contains only $x_{1}, \ldots, x_{n}$ as free variables, the universal closure $\uc A$ is taken to be the formula $\forall x_{1} \ldots \forall x_{n} A$.

\begin{definition}[$D_{w}$-Sentence]\label{def:d-sentence} Let $M = (W, \leq,D,V)$ be an $\intfo$-model with $w \in W$. We define $\langfo({D_{w}})$ to be the language $\langfo$ expanded with parameters $\unda, \undb, \undc, \ldots$ corresponding to the objects in the set $D_{w} = \{a, b, c, \ldots \}$. We define a \emph{$D_{w}$-formula} to be a formula in $\langfo(D_{w})$, and we define a \emph{$D_{w}$-sentence} to be a $D_{w}$-formula that does not contain any free variables. 
\end{definition}

\begin{definition}[First-Order Semantic Clauses]
\label{def:semantic-clauses-fo} Let $M = (W, \leq,D,V)$ be an $\intfo$- or $\intfocd$-model with $w \in W$. The \emph{satisfaction relation} $M,w \Vdash A$ between $w$ and a $D_{w}$-sentence $A$ is inductively defined as follows:

\begin{itemize}

\item[$\blacktriangleright$] If $p$ is a propositional variable, then $M,w \Vdash p$ iff $w \in V(p,w)$;

\item[$\blacktriangleright$] If $p$ is an $n$-ary predicate symbol $($with $n > 0)$, then $M,w \Vdash p(\unda_{1}, \cdots, \unda_{n})$ iff $(a_{1}, \cdots, a_{n}) \in V(p,w)$;




\item[$\blacktriangleright$] $M,w \Vdash \forall x A$ iff for all $v \in W$ and all $a \in D_{v}$, if $w \leq v$, then $M,v \Vdash A(\unda / x)$;

\item[$\blacktriangleright$] $M,w \Vdash \exists x A$ iff there exists an $a \in D_{w}$ such that $M,w \Vdash A(\unda / x)$.

\end{itemize}
The clauses for the $\lor$, $\land$, and $\imp$ connectives are the same as in Def.~\ref{def:prop-semantic-clauses}. 
We say that a formula $A$ is \emph{globally true on $M$}, written $M \Vdash A$, iff $M,u \Vdash \uc A$ for all worlds $u \in W$. A formula $A$ is \emph{$\intfo$-valid} (\emph{$\intfocd$-valid}), written $ \Vdash_{\intfo} A$ ($ \Vdash_{\intfocd} A$, resp.), iff it is globally true on all $\intfo$-models ($\intfocd$-models).
\end{definition}

Similar to the propositional case, the monotonicity condition imposed on atomic formulae in models generalizes:

\begin{lemma}[General Monotonicity] 
\label{lm:fo-gen-monotonicity} Let $M$ be an $\intfo$- or $\intfocd$-model with $w,v \in W$ of $M$. For any $D_{w}$-sentence $A$, if $M,w \Vdash A$ and $w \leq v$, then $M,v \Vdash A$.
\end{lemma}

\begin{proof} See~{\cite[Lem.~3.2.16]{GabSheSkv09}} for details.
\qed
\end{proof}

Sound and complete axiomatizations for our first-order intuitionistic logics (viz. $\intfo$ and $\intfocd$) are provided in Fig.~\ref{fig:axioms-FO}. We define the \emph{substitution} $(y/x)$ of the variable $y$ for the free variable $x$ on a formula $A$ in the standard way as the replacement of all free occurrences of $x$ in $A$ with $y$. 
Last, the side condition \emph{$y$ is free for $x$} (see~Fig.~\ref{fig:axioms-FO}) is taken to mean that $y$ does not become bound by a quantifier if substituted for $x$.


\begin{figure}[t]
\noindent\hrule
\[
\forall x (B \imp A) \imp (B \imp \forall x A)~\textit{ with $x \not\in B$}
\qquad
\forall x (A \imp B) \imp (\exists x A \imp B)~\textit{ with $x \not\in B$}
\]
\quad
\[
\forall x A \supset A(y/x)~\textit{y free for x}
\qquad
A(y/x) \supset \exists x A~\textit{y free for x}
\qquad
\infer[gen]
{\forall x A}
{A}
\]
\quad
\[
\forall x (A \vee B) \imp \forall x A \vee B~\textit{ with $x \not\in B$}
\]
\hrule
\caption{The axiomatization for the logic $\intfo$ is given by extending the axiomatization of $\int$ with the first two rows. The axiomatization for the logic $\intfocd$ is given by extending the axiomatization of $\int$ with all three rows. Both axiomatizations can be found in~\cite{GabSheSkv09}.}
\label{fig:axioms-FO}
\end{figure}

The logics $\intfo$ and $\intfocd$ are defined to be the smallest set of formulae from $\langfo$ closed under substitutions of the axioms and applications of the inference rules in their axiomatizations. The sole difference between the axiomatizations for $\intfond$ and $\intfocd$ is that the former omits the \emph{constant domain axiom} $\forall x (A \vee B) \imp \forall x A \vee B$ (with $x \not\in B$) whereas the latter includes it. We write $\vdash_{\intfo} A$ ($\vdash_{\intfocd} A$) to denote that $A$ is an element, or \emph{theorem}, of $\intfo$ ($\intfocd$, resp.).

\begin{theorem}[Soundness and Completeness~\cite{GabSheSkv09}] For any $A \in \langfo$, $\Vdash_{\intfo} A$ ($\Vdash_{\intfocd} A$) iff $\vdash_{\intfo} A$ ($\vdash_{\intfocd} A$, resp.).
\end{theorem}


\section{Proof Calculi for Intuitionistic Logics}\label{section-2}


In this section, we introduce the labelled and nested proof systems for the intuitionistic logics $\int$, $\intfond$, and $\intfocd$. The first subsection presents the labelled system $\lint$ for propositional intuitionistic logic from~\cite{DycNeg12} as well as the first-order extensions of this calculus, that makes use of quantifier and structural rules motivated by those (viz. $\forall L$, $\forall R$, $id$, and $dd$) given in~\cite[Ch.~6]{Vig00}. The second subsection introduces (slight variants of) Fitting's nested calculi for intuitionistic logics from~\cite{Fit14}.

\subsection{The Labelled Calculi $\lint$, $\lintfo$, and $\lintfocd$}

We define \emph{propositional $($first-order$)$ labelled sequents} to be syntactic objects of the form $L_{1} \Rightarrow L_{2}$ ($L_{1}' \Rightarrow L_{2}'$, resp.), where $L_{1}$ and $L_{2}$ ($L_{1}'$ and $L_{2}'$, resp.) are formulae defined via the BNF grammar below top (below bottom, resp.). 
\begin{center}
\begin{tabular}{c @{\hskip 2em} c}
$L_{1} ::= \varepsilon \ |\ w : A \ | \ w \leq v \ | \ L_{1},L_{1}$

&

$L_{2} ::= \varepsilon \ |\ w : A \ | \ L_{2},L_{2}$
\end{tabular}
\end{center}
\begin{center}
\begin{tabular}{c @{\hskip 2em} c}
$L_{1}' ::= \varepsilon \ |\ w : A \ | \ \unda \in D_{w} \ | \ w \leq v \ | \ L_{1}',L_{1}'$

&

$L_{2}' ::= \varepsilon \ |\ w : A \ | \ L_{2}',L_{2}'$
\end{tabular}
\end{center}
In the propositional case, $A$ is in the language $\lang$ and $w$ is among a denumerable set of labels $\{w, v, u, \ldots \}$. In the first-order case, $A$ is in the language $\langfo$, $\unda$ is among a denumerable set of \emph{parameters} $\{\unda, \undb, \undc, \ldots\}$, and $w$ is among a denumerable set of labels $\{w, v, u, \ldots \}$. We refer to formulae of the forms $w \leq v$ and $\unda \in D_{w}$ as \emph{relational atoms} (with formulae of the form $\unda \in D_{w}$ sometimes referred to as \emph{domain atoms}, more specifically) and refer to formulae of the form $w : A$ as \emph{labelled formulae}. Due to the two types of formulae occurring in a labelled sequent, we often use $\R$ to denote relational atoms, and $\Gamma$ and $\Delta$ to denote labelled formulae, thus distinguishing between the two. Labelled sequents are therefore written in a general form as $\R, \Gamma \Rightarrow \Delta$. 

\begin{figure}[t]
\noindent\hrule

\begin{center}
\begin{tabular}{c @{\hskip 1em} c} 

\AxiomC{}
\RightLabel{$\id$}
\UnaryInfC{$\R,w \leq v,w :p,\Gamma \Rightarrow \Delta, v :p$}
\DisplayProof

&

\AxiomC{$\R, w \leq v, v :A, \Gamma \Rightarrow \Delta, v :B$}
\RightLabel{$\imprr^{\dag_{1}}$}
\UnaryInfC{$\R, \Gamma \Rightarrow \Delta, w :A \imp B$}
\DisplayProof

\end{tabular}
\end{center}

\begin{center}
\begin{tabular}{c @{\hskip 1em} c @{\hskip 1em} c}

\AxiomC{$\R,w :A, w :B, \Gamma \Rightarrow \Delta$}
\RightLabel{$\conrl$}
\UnaryInfC{$\R, w :A \wedge B, \Gamma \Rightarrow \Delta$}
\DisplayProof

&

\AxiomC{$\R, \Gamma \Rightarrow \Delta, w :A$}
\AxiomC{$\R, \Gamma \Rightarrow \Delta, w :B$}
\RightLabel{$\conrr$}
\BinaryInfC{$\R, \Gamma \Rightarrow \Delta, w :A \wedge B$}
\DisplayProof

\end{tabular}
\end{center}

\begin{center}
\resizebox{\columnwidth}{!}{
\begin{tabular}{c @{\hskip 1em} c}

\AxiomC{$\R, w :A, \Gamma \Rightarrow \Delta$}
\AxiomC{$\R, w :B, \Gamma \Rightarrow \Delta$}
\RightLabel{$\disrl$}
\BinaryInfC{$\R, w :A \vee B, \Gamma \Rightarrow \Delta$}
\DisplayProof

&

\AxiomC{$\R,w \leq v, v \leq u, w \leq u, \Gamma \Rightarrow \Delta$}
\RightLabel{$\trans$}
\UnaryInfC{$\R,w \leq v, v \leq u, \Gamma \Rightarrow \Delta$}
\DisplayProof

\end{tabular}
}
\end{center}




\begin{center}
\begin{tabular}{c}
\AxiomC{$\R,w \leq v, w :A \imp B, \Gamma \Rightarrow \Delta, v :A$}
\AxiomC{$\R,w \leq v, w :A \imp B, v :B, \Gamma \Rightarrow \Delta$}
\RightLabel{$\imprl$}
\BinaryInfC{$\R,w \leq v, w :A \imp B, \Gamma \Rightarrow \Delta$}
\DisplayProof 
\end{tabular}
\end{center}

\begin{center}
\begin{tabular}{c c c}

\AxiomC{$\R,w \leq w, \Gamma \Rightarrow \Delta$}
\RightLabel{$\refl$}
\UnaryInfC{$\R,\Gamma \Rightarrow \Delta$}
\DisplayProof

&

\AxiomC{$\R, \Gamma \Rightarrow \Delta, w :A, w :B$}
\RightLabel{$\disrr$}
\UnaryInfC{$\R, \Gamma \Rightarrow \Delta, w :A \vee B$}
\DisplayProof

&

\AxiomC{}
\RightLabel{$\botr$}
\UnaryInfC{$\R,w :\bot, \Gamma \Rightarrow \Delta$}
\DisplayProof
\end{tabular}
\end{center}

\begin{center}
\begin{tabular}{c @{\hskip 1em} c}
\AxiomC{}
\RightLabel{$\idfo$}
\UnaryInfC{$\R,w \leq v, \vv{\unda} \in D_{w}, w :p(\vv{\unda}),\Gamma \Rightarrow \Delta, v :p(\vv{\unda})$}
\DisplayProof

&

\AxiomC{$\R, \unda \in D_{w}, \Gamma \Rightarrow \Delta$}
\RightLabel{$\ned^{\dag_{2}}$}
\UnaryInfC{$\R,\Gamma \Rightarrow \Delta$}
\DisplayProof
\end{tabular}
\end{center}

\begin{center}
\resizebox{\columnwidth}{!}{
\begin{tabular}{c c} 
\AxiomC{$\R, w \leq v, \unda \in D_{v}, \Gamma \Rightarrow \Delta, v : A(\unda/x)$}
\RightLabel{$\allr^{\dag_{3}}$}
\UnaryInfC{$\R, \Gamma \Rightarrow \Delta, w : \forall x A$}
\DisplayProof

&

\AxiomC{$\R, \unda \in D_{w}, \Gamma \Rightarrow \Delta, w: A(\unda/x), w: \exists x A$}
\RightLabel{$\existsr$}
\UnaryInfC{$\R, \unda \in D_{w}, \Gamma \Rightarrow \Delta, w: \exists x A$}
\DisplayProof

\end{tabular}
}
\end{center}

\begin{center}
\resizebox{\columnwidth}{!}{
\begin{tabular}{c c} 

\AxiomC{$\R, \unda \in D_{w}, w: A(\unda/x), \Gamma \Rightarrow \Delta$}
\RightLabel{$\existsl^{\dag_{2}}$}
\UnaryInfC{$\R, w : \exists x A, \Gamma \Rightarrow \Delta$}
\DisplayProof

&

\AxiomC{$\R, w \leq v, \unda \in D_{v}, v : A(\unda/x), w : \forall x A, \Gamma \Rightarrow \Delta$}
\RightLabel{$\alll$}
\UnaryInfC{$\R, w \leq v, \unda \in D_{v}, w : \forall x A, \Gamma \Rightarrow \Delta$}
\DisplayProof

\end{tabular}
}
\end{center}

\begin{center}
\begin{tabular}{c @{\hskip 1em} c}

\AxiomC{$\R, w \leq v, \unda \in D_{w}, \unda \in D_{v}, \Gamma \Rightarrow \Delta$}
\RightLabel{$\nd$}
\UnaryInfC{$\R, w \leq v, \unda \in D_{w}, \Gamma \Rightarrow \Delta$}
\DisplayProof

&

\AxiomC{$\R, w \leq v, \unda \in D_{v}, \unda \in D_{w}, \Gamma \Rightarrow \Delta$}
\RightLabel{$\cd$}
\UnaryInfC{$\R, w \leq v, \unda \in D_{v}, \Gamma \Rightarrow \Delta$}
\DisplayProof

\end{tabular}
\end{center}

\hrule
\caption{The labelled calculus $\lint$ for $\int$ consists of $\id$, $\botr$, $\conrl$, $\conrr$, $\disrl$, $\disrr$, $\imprl$, $\imprr$, $\refl$, and $\trans$ (see~\cite{DycNeg12}). The labelled calculus $\lintfo$ for $\intfo$ consists of all rules minus the $\cd$ rule, and all rules give the calculus $\lintfocd$ for $\intfocd$. The side condition $\dag_{1}$ states that the variable $v$ does not occur in the conclusion, $\dag_{2}$ states that $\unda$ does not occur in the conclusion, and $\dag_{3}$ states that neither $\unda$ nor $v$ occurs in the conclusion. Labels and parameters restricted from occurring in the conclusion of an inference are called \emph{eigenvariables}.
}
\label{fig:labelled-calculi}
\end{figure}

Moreover, we take the comma operator to be commutative and associative; for example, we identify the labelled sequent $w \leq u, w : A, \unda \in D_{w} \sar v : C, u :B$ with $\unda \in D_{w}, w \leq u, w : A \sar u :B, v : C$. This interpretation of comma lets us view $\R, \Gamma$ (the \emph{antecedent}) and $\Delta$ (the \emph{succedent}) of a labelled sequent $\R, \Gamma \sar \Delta$ as multisets. Also, we use $\varepsilon$ to denote the \emph{empty string} which acts as the identity element for the comma operator (e.g. we identify $w \leq v, \varepsilon, v : B$ with $w \leq v, v : B$), and we stipulate that if $\varepsilon$ is the antecedent or succedent of a sequent, then the antecedent or succedent is left empty. Therefore, $\varepsilon$ will be implicit in labelled sequents, but will never explicitly appear.  

In the first-order setting, we syntactically distinguish between \emph{bound variables} $\{x, y, z, \ldots\}$ and \emph{free variables}, which are replaced with \emph{parameters} $\{\unda, \undb, \undc, \ldots\}$, to avoid clashes between the two categories (cf.~\cite[Sect.~8]{Fit14}). Therefore, instead of using formulae directly from the first-order language, we use formulae from the first-order language where each freely occurring variable $x$ has been replaced by a distinct parameter $\unda$. For example, we would make use of the labelled formula $w : (\forall x) p(\unda,x) \lor q(\unda,\undb)$ instead of $w : (\forall x) p(y,x) \lor q(y,z)$ in a first-order sequent. 
Last, we use the notation $A(\unda_{1}, \ldots, \unda_{n})$, with $n \in \mathbb{N}$, to denote that the parameters $\unda_{1}, \ldots, \unda_{n}$ are all parameters occurring in the formula $A$. (NB. We will occasionally abuse notation and write $A(\unda)$ to indicate that the formula $A$ contains a parameter $\unda$ in which we are interested; however, when using this notation we leave open the possibility that $A$ may contain other parameters as well.) We write $A(\vv{\unda})$ as shorthand for $A(\unda_{1}, \ldots, \unda_{n})$ and $\vv{\unda} \in D_{w}$ as shorthand for $\unda_{1} \in D_{w}, \ldots, \unda_{n} \in D_{w}$. The labelled calculi are given in Fig.~\ref{fig:labelled-calculi}.


Each labelled calculus is obtained from the models and semantic clauses of $\int$, $\intfond$, and $\intfocd$. We note that the labelled calculus $\lint$ for $\int$ is presented in~\cite{DycNeg12}, whereas the labelled calculi $\lintfond$ and $\lintfocd$ for $\intfond$ and $\intfocd$ (respectively) are new. The rules $\botl$, $\conrl$, $\conrr$, $\disrl$, $\disrr$, $\imprl$, $\imprr$, $\existsl$, $\existsr$, $\alll$, and $\allr$ are rule representations of the semantic clauses given in Def.~\ref{def:prop-semantic-clauses} and Def.~\ref{def:semantic-clauses-fo}. The $\refl$ and $\trans$ rules allow inferences arising from the fact that frames are reflexive and transitive, whereas $\ned$, $\nd$, and $\cd$ allow inferences based on the fact that the domains of frames are always inhabited (i.e. non-empty), are nested, and are constant, respectively. The rules $\id$ and $\idfo$ encode the monotonicity condition imposed on models; note that $\id$ is an instance of $\idfo$ (the same holds for the corresponding rules in the nested setting, which are introduced in the next subsection). Last, we write $\vdash_{\lint} \Lambda$, $\vdash_{\lintfond} \Lambda$, and $\vdash_{\lintfocd} \Lambda$ to indicate that the labelled sequent $\Lambda$ is derivable in $\lint$, $\lintfond$, and $\lintfocd$, respectively---this notation extends straightforwardly to the other calculi we will consider.

We define a \emph{label substitution} $(w/v)$ on a multiset of labelled formulae or relational atoms in the usual way as the replacement of all labels $v$ occurring in the mulitset with the label $w$. Similarly, we define a \emph{parameter substitution} $(\unda/\undb)$ on a multiset of labelled formulae or relational atoms as the replacement of all parameters $\undb$ occurring in the multiset with the parameter $\unda$. 

Our labelled calculi possess desirable proof-theoretic properties such as the \emph{height-preserving (hp-) admissibility} of substitutions and structural rules (e.g. $(psub)$ and $\wk$), the \emph{height-preserving (hp-) invertibility} of all rules, and admissibility of $\cut$. These properties are detailed in the theorem below (Thm.~\ref{thm:lint-properties}) with the (hp-)admissible rules shown in Fig.~\ref{fig:lab-struc-rules}.

\begin{figure}[t]
\noindent\hrule
\begin{center}
\begin{tabular}{c @{\hskip 1em} c}
\AxiomC{$\R,\Gamma \Rightarrow \Delta$}
\RightLabel{$(lsub)$}
\UnaryInfC{$\R(w/v),\Gamma(w/v) \Rightarrow \Delta(w/v)$}
\DisplayProof

&

\AxiomC{$\R,\Gamma \Rightarrow \Delta$}
\RightLabel{$(psub)$}
\UnaryInfC{$\R(\unda/\undb),\Gamma(\unda/\undb) \Rightarrow \Delta(\unda/\undb)$}
\DisplayProof
\end{tabular}
\end{center}

\begin{center}
\begin{tabular}{c @{\hskip 1em} c @{\hskip 1em} c}
\AxiomC{$\R,\Gamma \Rightarrow \Delta$}
\RightLabel{$\wk$}
\UnaryInfC{$\R',\R,\Gamma',\Gamma \Rightarrow \Delta',\Delta$}
\DisplayProof

&

\AxiomC{$\R,\R',\R',\Gamma \Rightarrow \Delta$}
\RightLabel{$\ctrrel$}
\UnaryInfC{$\R,\R',\Gamma \Rightarrow \Delta$}
\DisplayProof

&

\AxiomC{$\R,\Gamma',\Gamma',\Gamma \Rightarrow \Delta$}
\RightLabel{$\ctrl$}
\UnaryInfC{$\R,\Gamma',\Gamma \Rightarrow \Delta$}
\DisplayProof
\end{tabular}
\end{center}

\begin{center}
\begin{tabular}{c @{\hskip 1em} c}
\AxiomC{$\R,\Gamma \Rightarrow \Delta, \Delta', \Delta'$}
\RightLabel{$\ctrr$}
\UnaryInfC{$\R,\Gamma \Rightarrow \Delta, \Delta'$}
\DisplayProof

&

\AxiomC{$\R,\Gamma \Rightarrow \Delta, w :A$}
\AxiomC{$\R,w :A,\Gamma \Rightarrow \Delta$}
\RightLabel{$\cut$}
\BinaryInfC{$\R,\Gamma \Rightarrow \Delta$}
\DisplayProof
\end{tabular}
\end{center}

\hrule
\caption{The set $\strucset$ of admissible labelled rules consists of all rules shown above.}
\label{fig:lab-struc-rules}
\end{figure}

\begin{theorem}
\label{thm:lint-properties} Let $\mathsf{G3X} \in \{\lintfo,\lintfocd \}$. The calculi $\lint$, $\lintfo$, and $\lintfocd$ have the following properties:
\begin{itemize}

\item[$(i)$] 

\begin{itemize}

\item[(a)] For all $A \in \lang$, $ \vdash_{\lint} \R,w \leq v, w : A, \Gamma \Rightarrow v : A, \Delta$;

\item[(b)] For all $A \in \lang$, $ \vdash_{\lint} \R,w:A,\Gamma \Rightarrow \Delta, w :A$; 

\item[(c)] For all $A \in \langfo$, $\vdash_{\mathsf{G3X}} \R,w \leq v, \vv{\unda} \in D_{w}, w : A(\vv{\unda}), \Gamma \Rightarrow v : A(\vv{\unda}), \Delta$; 

\item[(d)] For all $A \in \langfo$, $\vdash_{\mathsf{G3X}} \R, \vv{\unda} \in D_{w}, w:A(\vv{\unda}),\Gamma \Rightarrow \Delta, w :A(\vv{\unda})$;

\end{itemize}

\item[$(ii)$] All rules in $\strucset - \{\cut\}$ are hp-admissible;



\item[$(iii)$] All rules are hp-invertible;


\item[$(iv)$] The $\cut$ rule is admissible;

\item[$(v)$] $\lint$, $\lintfo$, and $\lintfocd$ are sound and complete for $\int$, $\intfo$, and $\intfocd$, respectively.

\end{itemize}
\end{theorem}

\begin{proof} We refer the reader to~\cite{DycNeg12} for proofs of properties (i)--(v) for $\lint$; note that hp-admissibility of $(psub)$ is trivial in the propositional setting since formulae do not contain parameters. The proofs of properties (i)--(v) for $\lintfo$ and $\lintfocd$ are similar to those for $\lint$ and can be found in Appendix~\ref{app:A}.
\qed
\end{proof}

\subsection{The Nested Calculi $\nint$, $\nintfo$, and $\nintfocd$}

We define a propositional (or, first-order) nested sequent $\Sigma$ to be a syntactic object defined via the following BNF grammars:
\begin{center}
\begin{tabular}{c @{\hskip 2em} c}
$X ::= \varepsilon \ | \ A \ | \ X, X$ 

&

$\Sigma ::= X \far X \ | \ X \far X, [\Sigma], \ldots, [\Sigma]$
\end{tabular}
\end{center}
where $A$ is in the propositional language $\lang$ (first-order language $\langfo$, resp.). As in the previous section, we take the comma operator to be commutative and associative, allowing us to view syntactic entities $X$ as multisets, and we let $\varepsilon$ represent the empty string which---as in the labelled setting---occurs implicitly, but not explicitly, in sequents.

In the first-order setting, we syntactically distinguish between bound variables and free variables in first-order formulae, using $\{x, y, z, \ldots\}$ for bound variables and replacing the occurrence of free variables in formulae with parameters $\{\unda, \undb, \undc, \ldots\}$. For example, we would use $p(\unda) \far p(\undb), [\bot \far \forall x q(x,\undb)]$ instead of the sequent $p(x) \far p(y), [\bot \far \forall x q(x,y)]$ in a nested derivation (where the free variable $x$ has been replaced by the parameter $\unda$ and $y$ has been replaced by $\undb$).




Nested sequents are often written as $\Sigma\{X \far Y, [\Sigma_{1}], \ldots, [\Sigma_{n}] \}$, which indicates that $X \far Y, [\Sigma_{1}], \ldots, [\Sigma_{n}]$ occurs at some depth in the nestings of the sequent $\Sigma$. For example, if $\Sigma$ is taken to be $p(\unda) \far [\bot \far \forall x q(x,\undb), [ \ \far \top ]]$, then $\Sigma\{\bot \far \forall x q(x,\undb)\}$, $\Sigma\{\bot \far \forall x q(x,\undb), [ \ \far \top ]\}$, and $\Sigma\{ \far \top\}$ are correct representations of $\Sigma$ in our notation. 


The nested calculi are given in Fig.~\ref{fig:nested-calculi} and are \emph{slight} variants of the calculi presented by Fitting in~\cite{Fit14}. The only difference is that the rules $\negl$, $\lift$, $\impl$, $\existsr$, and $\alll$ preserve copies of the principal formula in the premise of the rule. Since we will be extracting Fitting's nested calculi from our labelled calculi, it will be seen that these copies of principal formulae 
are residua of this process, which is based on the fact that our labelled calculi preserve copies of principal formulae in the premise(s) of certain rules. Howbeit, it is easy to confirm that the calculi $\nint$, $\nintfond$, and $\nintfocd$ are equivalent to Fitting's calculi from~\cite{Fit14}, as it is straightforward to transform a derivation from one calculus into a derivation in the other, and vice-versa. 

\begin{figure}[t]
\noindent\hrule

\begin{center}
\begin{tabular}{c c c}
\AxiomC{} \RightLabel{$\id$}
\UnaryInfC{$\Sigma \lcut X, p \far p, Y \rcut$}
\DisplayProof

&

\AxiomC{$\Sigma \lcut X, A,B \far Y \rcut $}
\RightLabel{$\conrl$}
\UnaryInfC{$\Sigma \lcut X, A \land B \far Y \rcut$}
\DisplayProof

&

\AxiomC{$\Sigma\{X, A \far Y, [X', A \far Y']\}$}
\RightLabel{$\lift$}
\UnaryInfC{$\Sigma\{X, A \far Y, [X' \far Y']\}$}
\DisplayProof
\end{tabular}
\end{center}

\begin{center}
\begin{tabular}{c c}
\AxiomC{$\Sigma \lcut X, A \far Y \rcut$}
\AxiomC{$\Sigma \lcut X, B \far Y \rcut$}
\RightLabel{$\disrl$}
\BinaryInfC{$\Sigma \lcut X, A \lor B \far Y \rcut$}
\DisplayProof

&

\AxiomC{$\Sigma \lcut X \far A, Y \rcut$}
\AxiomC{$\Sigma \lcut X \far B, Y \rcut$}
\RightLabel{$\conrr$}
\BinaryInfC{$\Sigma \lcut X \far A\land B, Y \rcut$}
\DisplayProof

\end{tabular}
\end{center}

\begin{center}
\begin{tabular}{c c c}
\AxiomC{$\Sigma \lcut X \far Y, [A \far ] \rcut$}
\RightLabel{$\negr$}
\UnaryInfC{$\Sigma \lcut X \far Y, \neg A \rcut$}
\DisplayProof

&

\AxiomC{$\Sigma \lcut X, \neg A \far A, Y \rcut$}
\RightLabel{$\negl$}
\UnaryInfC{$\Sigma \lcut X, \neg A \far Y \rcut$}
\DisplayProof

&

\AxiomC{$\Sigma \lcut X \far A,B, Y \rcut $}
\RightLabel{$\disrr$}
\UnaryInfC{$\Sigma \lcut X \far A\lor B, Y \rcut$}
\DisplayProof
\end{tabular}
\end{center}

\begin{center}
\begin{tabular}{c c}
\AxiomC{$\Sigma \lcut X \far Y, [A \far B] \rcut$}
\RightLabel{$\imprr$}
\UnaryInfC{$\Sigma \lcut X \far A \imp B, Y \rcut$}
\DisplayProof

&

\AxiomC{$\Sigma \lcut X, A \imp B \far A, Y \rcut$}
\AxiomC{$\Sigma \lcut X, A \imp B, B \far Y \rcut$}
\RightLabel{$\imprl$}
\BinaryInfC{$\Sigma \lcut X, A \imp B \far Y \rcut$}
\DisplayProof
\end{tabular}
\end{center}

\begin{center}
\begin{tabular}{c c}
\AxiomC{}
\RightLabel{$\idfo$}
\UnaryInfC{$\Sigma \lcut X, p(\vv{\unda}) \far p(\vv{\unda}), Y \rcut$}
\DisplayProof

&

\AxiomC{$\Sigma \lcut X \far A(\unda/x), \exists x A, Y  \rcut$}
\RightLabel{$\existsr^{\dag_{1}}$}
\UnaryInfC{$\Sigma \lcut X \far \exists x A, Y \rcut$}
\DisplayProof
\end{tabular}
\end{center}

\begin{center}
\begin{tabular}{c c}
\AxiomC{$\Sigma \lcut X, A(\unda/x) \far Y  \rcut$}
\RightLabel{$\existsl^{\dag_{2}}$}
\UnaryInfC{$\Sigma \lcut X, \exists x A \far Y \rcut$}
\DisplayProof

&

\AxiomC{$\Sigma \lcut X \far Y, [ \ \far A(\unda/x)]  \rcut$}
\RightLabel{$\allndr^{\dag_{2}}$}
\UnaryInfC{$\Sigma \lcut X \far \forall x A, Y \rcut$}
\DisplayProof
\end{tabular}
\end{center}

\begin{center}
\begin{tabular}{c c}
\AxiomC{$\Sigma \lcut X \far A(\unda/x), Y  \rcut$}
\RightLabel{$\allcdr^{\dag_{2}}$}
\UnaryInfC{$\Sigma \lcut X \far \forall x A, Y \rcut$}
\DisplayProof

&

\AxiomC{$\Sigma \lcut X, \forall x A, A(\unda/x) \far Y  \rcut$}
\RightLabel{$\alll^{\dag_{1}}$}
\UnaryInfC{$\Sigma \lcut X, \forall x A \far Y \rcut$}
\DisplayProof
\end{tabular}
\end{center}

\hrule
\caption{The nested calculus $\nint$ for propositional intuitionistic logic consists of $\id$, $\conrl$, $\conrr$, $\disrr$, $\disrl$, $\negl$, $\negr$, $\imprr$, $\imprl$, and $\lift$. The nested calculus $\nintfo$ extends $\nint$ with $\idfo$, $\existsl$, $\existsr$, $\alll$, and $\allndr$. The nested calculus $\nintfocd$ extends $\nint$ with $\idfo$, $\existsl$, $\existsr$, $\alll$, and $\allcdr$ and omits the side condition $\dag_{1}$ on the $\existsr$ and $\alll$ rules~\cite{Fit14}. The side condition $\dag_{1}$ states that the parameter $\unda$ is either available or is an eigenvariable, and $\dag_{2}$ states that $\unda$ is an eigenvariable.}
\label{fig:nested-calculi}
\end{figure} 

The two distinguishing factors between the nested calculi $\nintfo$ and $\nintfocd$ are that (i) the former uses the $\allndr$ rule whereas the latter uses the $\allcdr$ rule, and (ii) the former imposes a side condition (see Fig.~\ref{fig:nested-calculi}) on the $\existsr$ and $\alll$ rules, whereas the latter does not. This side condition ensures the soundness of the $\existsr$ and $\alll$ rules with respect to $\intfo$-models, that is, with respect to models where domains are \emph{not necessarily constant} (cf.~\cite{Fit14}). The side condition relies on the notion of an \emph{available parameter}, which we define below:

\begin{definition}[Available Parameter~\cite{Fit14}]\label{def:available-parameter} Let $\Sigma\{X \far Y, [\Sigma_{1}], \ldots, [\Sigma_{n}] \}$ be a nested sequent. If there exists a formula $A(\unda) \in X,Y$, then the parameter $\unda$ is available in $X \far Y$ and in all boxed subsequents $\Sigma_{i}$ (with $i \in \{1, \ldots, n\}$).
\end{definition}

As explained in the following section, a nested sequent can be seen as a tree of sequents of the form $X \far Y$. Intuitively, if a parameter is available at some point $X \far Y$ in the tree, then it will be available at all points in the future of that point, i.e. in the subtree with $X \far Y$ as its root. Since a nested sequent can  be seen as an abstraction of an $\intfond$-model encoding the model's underlying treelike structure, the notion of an available parameter corresponds to the fact that if a parameter denotes an object in the domain of a world, then---due to the nested domain condition (Def.~\ref{def:FO-frame-model})---that parameter denotes that object at all future worlds and is \emph{available} for use in sentences at such worlds.

A feature which distinguishes the nested formalism from the labelled is that nested sequents may readily be converted into an \emph{equivalent formula}, that is, nested sequents allow for a straightforward formula translation (see Def.~\ref{def:formula-interpretation} below). As a consequence of our refinement procedure, we will see in Sect.~\ref{section-7} (viz. Cor.~\ref{cor:nested-internality}) that a nested sequent $\Sigma$ is provable if and only if its formula interpretation is provable. 

\begin{definition}[Formula Interpretation] 
\label{def:formula-interpretation}
The \emph{formula interpretation} of a nested sequent is defined inductively as follows:
$$\iota(X \far Y) := \displaystyle{\bigwedge X \imp \bigvee Y}
$$
$$
\iota(X \far Y, [\Sigma_{1}], \ldots, [\Sigma_{n}]) := \displaystyle{ \bigwedge X \imp \bigg ( \bigvee Y \vee \bigvee_{1 \leq i \leq n} \iota(\Sigma_{i}) \bigg )}
$$
In the propositional setting, $\iota(\Sigma)$ is a formula in $\lang$. In the first-order setting, we interpret a nested sequent $\Sigma$ as a formula in $\langfo$ by taking the universal closure $\uc \iota(\Sigma)$ of $\iota(\Sigma)$. Also, as usual $\bigwedge \varepsilon := \top$ and $\bigvee \varepsilon := \bot$, i.e. empty antecedents translate to $\top$ and empty succedents translate to $\bot$. 
\end{definition}

Proof-theoretic properties of $\nint$, $\nintfond$, and $\nintfocd$ will be discussed in Sect.~\ref{section-7}, though we do state the essential properties of soundness and completeness below, which were proven by Fitting in~\cite{Fit14}. By \emph{strongly} sound and complete, we mean that a nested sequent $\Sigma$ is valid (relative to an $\int$-, $\intfond$-, or $\intfocd$-model) iff it is derivable (in $\nint$, $\nintfond$, or $\nintfocd$, respectively).

\begin{theorem}[Soundness and Completeness]\label{thm:strong-soundness-completeness-nested}
$\nint$, $\nintfond$, and $\nintfocd$ are strongly sound and complete for $\int$, $\intfond$, and $\intfocd$, respectively.
\end{theorem}

Interestingly, the correspondence we establish between $\lint$ and $\nint$, and $\lintfocd$ and $\nintfocd$, can be leveraged to show that $\nint$ and $\nintfocd$ inherit soundness and completeness from their associated labelled calculi. The correspondence between $\nintfond$ and $\lintfond$ cannot be leveraged to conclude the completeness of $\nintfond$ however, because of the use of a lemma (Lem.~\ref{lem:Nestedlike-sequents-have-nested-form-proofs} in Sect.~\ref{section-5}) that invokes the strong completeness of $\nintfond$ (i.e. it invokes Thm.~\ref{thm:strong-soundness-completeness-nested} above). Still, we note in that section how the lemma may be proved independent of the completeness of $\nintfond$, thus allowing for $\nintfond$ to inherit the completeness of $\lintfond$ via our refinement and translation processes. Last, in Sect.~\ref{section-7}, we will show that our nested calculi inherit additional properties from their corresponding labelled calculi such as the admissibilty of certain structural rules (Cor.~\ref{cor:nested-admissibility}) and the invertibility of all rules (Cor.~\ref{cor:nested-invertibility}).


\section{Fundamentals for Establishing Correspondence}\label{section-3}

This section consists of two parts: in the first subsection, we define translation functions that transform labelled sequents into nested sequents and vice-versa, as well as define classes of labelled sequents that are fruitful for carrying out our proof-theoretic transformation and translation work. In the second subsection, we establish preliminary results that are convenient for refining our labelled calculi, that is, such results will assist us in eliminating the structural rules (e.g. $\refl$ and $\nd$) from $\lint$, $\lintfond$, and $\lintfocd$ (which will ultimately yield systems that are close variants of Fitting's nested systems).

\subsection{Translating Notation: Labelled and Nested}\label{section-notation}

It is instructive to observe that both nested and labelled sequents can be viewed as graphs (with the former restricted to trees and the latter more general). Graphs of sequents are significant for two reasons: the first (technical) reason is that graphs can be leveraged to switch from labelled to nested notation; thus, graphs will play a role in deriving our nested calculi from our labelled calculi, and vice-versa. Second, 
graphs offer insight into \emph{why} structural rule elimination yields nested systems, which will be discussed in the next section (Sect.~\ref{section-4}). 

It is straightforward to define the graph of each type of sequent. To do this, we first introduce a bit of notation and define the multiset $\Gamma \restriction w := \{A \ | \ w : A \in \Gamma\}$. For a labelled sequent $\Lambda := \R, \Gamma \Rightarrow \Delta$, the \emph{graph} $G(\Lambda)$ is the tuple $( V, E, \lambda )$, where (i) $V := \{w \ | \ \text{$w$ is a label in $\Lambda$.}\}$, (ii) $(w,v) \in E$ iff $w \leq v \in \R$, and
$$
(iii) \quad \lambda := \{(w,\Gamma' \Rightarrow \Delta') \ | \ \Gamma' = \Gamma \restriction w \text{, } \Delta' = \Delta \restriction w \text{, and } w \in V\}.$$

For a nested sequent, the graph is defined inductively on the structure of the nestings; we use strings $\sigma$ of natural numbers to denote vertices in the graph, similar to the prefixes used in prefixed tableaux~\cite{Fit72,Fit12,Fit14}.

\emph{Base case.} Let our nested sequent be of the form $X \far Y$ with $X$ and $Y$ multisets of formulae. Then, $G_{\sigma}(X \far Y) := (V_{\sigma},E_{\sigma},\lambda_{\sigma})$, where (i) $V_{\sigma} := \{\sigma\}$, (ii) $E_{\sigma} := \emptyset$, and (iii) $\lambda_{\sigma} := \{(\sigma, X \far Y)\}$.

\emph{Inductive step.} Suppose our nested sequent is of the form $X \far Y, [\Sigma_{1}], \ldots, [\Sigma_{n}]$. We assume that each $G_{\sigma  . i}(\Sigma_{i}) = ( V_{ \sigma . i}, E_{\sigma . i}, \lambda_{\sigma . i} )$ (with $i \in \{1, \ldots, n\}$) is already defined, and define $G_{\sigma}(X \far Y, [\Sigma_{1}], \ldots, [\Sigma_{n}]) := ( V_{\sigma}, E_{\sigma}, \lambda_{\sigma} )$ as follows:
$$
(i) \quad V_{\sigma} := \{\sigma\} \cup \displaystyle{\bigcup_{1 \leq i \leq n} V_{\sigma . i}} \qquad (ii) \quad E_{\sigma} := \{(\sigma,\sigma . i) \ | \ 1 \leq i \leq n \} \cup \displaystyle{\bigcup_{1 \leq i \leq n} E_{\sigma . i}}
$$
$$(iii) \quad \lambda_{\sigma} := \{(\sigma, X \far Y)\} \cup \displaystyle{\bigcup_{1 \leq i \leq n} \lambda_{\sigma . i}}
$$
We will occasionally refer to a sequent of the form $X \far Y$ in a nested sequent $\Sigma$ such that $(\sigma, X \far Y) \in \lambda_{0}$ (with $G_{0}(\Sigma) = (V_{0}, E_{0}, \lambda_{0})$) as a \emph{component} of the nested sequent $\Sigma$.


\begin{definition} Let $G_{0} = (V_{0},E_{0},\lambda_{0})$ and $G_{1} = (V_{1},E_{1},\lambda_{1})$ be two graphs. We define an \emph{isomorphism} $f : V_{0} \mapsto V_{1}$ between $G_{0}$ and $G_{1}$ to be a function such that: (i) $f$ is bijective, (ii) $(x,y) \in E_{0}$ iff $(fx,fy) \in E_{1}$, (iii) $\lambda_{0}(x) = \lambda_{1}(fx)$. We say $G_{0}$ and $G_{1}$ are \emph{isomorphic} iff there exists an isomorphism between them.

\end{definition}

Although the formal definitions above may appear somewhat cumbersome, the example below shows that transforming a sequent into its graph---or conversely, obtaining the sequent from its graph---is relatively straightforward. 

\begin{example}\label{ex:graphs-of-sequents} The nested sequent $\Sigma$ is given below with its corresponding graph $G_{0}(\Sigma)$ shown on the left, and the labelled sequent $\Lambda$ is given below with its corresponding graph $G(\Lambda)$ on the right. Regarding the labelled sequent, we assume that $\Gamma_{i}$ and $\Delta_{i}$ consist solely of formulae labelled with $w_{i}$ (for $i \in \{0,1,2,3\}$).
$$
\Sigma = X_{0} \far Y_{0}, [X_{1} \far Y_{1}, [X_{2} \far Y_{2}]], [X_{3} \far Y_{3}]
$$
\vspace*{0em}
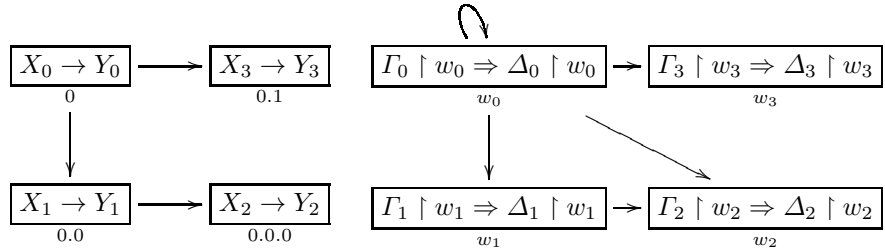
\begin{center}
\begin{tabular}{c @{\hskip 1em} c}
\xymatrix{
  \underset{0}{\boxed{X_{0} \far Y_{0}}}\ar@{->}[r]\ar@{>}[d] & \underset{0.1}{\boxed{X_{3} \far Y_{3}}}		\\
 	\underset{0.0}{\boxed{X_{1} \far Y_{1}}}\ar@{>}[r] & \underset{0.0.0}{\boxed{X_{2} \far Y_{2}}}   
}

&

\xymatrix@C=1em{
  \underset{w_{0}}{\boxed{\Gamma_{0} \restriction w_{0} \Rightarrow \Delta_{0} \restriction w_{0}}}\ar@{->}[r]\ar@{>}[d]\ar@(ul,u)\ar@{>}[dr] & \underset{w_{3}}{\boxed{\Gamma_{3} \restriction w_{3} \Rightarrow \Delta_{3} \restriction w_{3}}}		\\
 	\underset{w_{1}}{\boxed{\Gamma_{1} \restriction w_{1} \Rightarrow \Delta_{1} \restriction w_{1}}}\ar@{>}[r] & \underset{w_{2}}{\boxed{\Gamma_{2} \restriction w_{2} \Rightarrow \Delta_{2} \restriction w_{2}}}   
}
\end{tabular}
\end{center}
\begin{small}
$$
\Lambda = w_{0} \leq w_{0}, w_{0} \leq w_{1}, w_{1} \leq w_{2}, w_{0} \leq w_{2}, w_{0} \leq w_{3}, \Gamma_{0}, \Gamma_{1} , \Gamma_{2}, \Gamma_{3} \Rightarrow \Delta_{0}, \Delta_{1}, \Delta_{2}, \Delta_{3}   
$$
\end{small}
\end{example}

In the above example there is a loop from $w_{0}$ to itself in the graph of the labelled sequent; furthermore, there is an undirected cycle occurring between $w_{0}$, $w_{1}$, and $w_{2}$. As will be explained in the next section (specifically, Thm.~\ref{thm:treelike-derivations}), the $\refl$ and $\trans$ rules allow for such structures to appear in the graphs of labelled sequents used to derive theorems; however, the elimination of these rules in our labelled calculi has the effect that such structures \emph{can no longer} occur in the labelled derivation of a theorem. Consequently, it will be seen that eliminating such rules yields a labelled derivation where every sequent has a purely \emph{treelike} structure (cf.~\cite{GorRam12AIML}). This implies that each labelled sequent in the derivation has a graph isomorphic to the graph of a nested sequent. It is this idea which ultimately permits the extraction of our nested calculi from our labelled calculi.

\begin{definition}\label{def:treelike} Let $\Lambda$ be a labelled sequent and $G(\Lambda) = (V,E,\lambda)$. We say that $\Lambda$ is \emph{treelike} iff there exists a unique vertex $w \in V$, called the \emph{root}, such that there exists a unique path from $w$ to every other vertex $v \in V$.\footnote{Treelike sequents are equivalently characterized as sequents with graphs that are: (i) connected, (ii) acyclic, and (iii) contain no backwards branching.} We say that a labelled derivation is \emph{treelike} iff every labelled sequent in the derivation is treelike.
\end{definition}

If we take the graph of a treelike labelled sequent, then it can be viewed as the graph of a nested sequent, 
as the example below demonstrates:

\begin{example}\label{ex:switching-notation} The treelike labelled sequent $\Lambda'$ and its graph are given below. We assume that $\Gamma_{i}$ and $\Delta_{i}$ contain only formulae labelled with $w_{i}$ (for $i \in \{0,1,2,3\}$).
$$
\Lambda' = w_{0} \leq w_{1}, w_{1} \leq w_{2}, w_{0} \leq w_{3}, \Gamma_{0}, \Gamma_{1} , \Gamma_{2}, \Gamma_{3} \Rightarrow \Delta_{0}, \Delta_{1}, \Delta_{2}, \Delta_{3}   
$$
\begin{center}
\begin{tabular}{c} 

\xymatrix@C=1em{
\underset{w_{2}}{\boxed{\Gamma_{2}' \Rightarrow \Delta_{2}'}} & & \underset{w_{1}}{\boxed{\Gamma_{1}' \Rightarrow \Delta_{1}'}}\ar@{>}[ll] & & \underset{w_{0}}{\boxed{\Gamma_{0}' \Rightarrow \Delta_{0}'}}\ar@{->}[rr]\ar@{>}[ll] & & \underset{w_{3}}{\boxed{\Gamma_{3}' \Rightarrow \Delta_{3}'}}		
}
\end{tabular}
\end{center}
Also, if we assume that $\Gamma_{i}' = \Gamma_{i} \restriction w_{i} = X_{i}$ and $\Delta_{i} ' = \Delta_{i} \restriction w_{i} = Y_{i}$ (for $i \in \{0,1,2,3\}$), then the above graph is isomorphic to the graph of the nested sequent in Example~\ref{ex:graphs-of-sequents}, meaning that $\Lambda$ can be translated as that nested sequent. 
\end{example}

To make the above translation precise, we introduce the \emph{$w$-downward closure} of the graph of a labelled sequent. We may utilize this notion to explicitly define the \emph{translation} $\switch$ (see Def.~\ref{def:switch}) that translates labelled sequents into nested sequents. This translation will prove itself useful in extracting Fitting's nested calculi from our labelled calculi.

\begin{definition}[Downward Closure] Let $\Lambda$ be a treelike labelled sequent with graph $G(\Lambda) = (V,E,\lambda)$ and $w \in V$. We define the \emph{$w$-downward closure} $G_{w}(\Lambda) = (V',E',\lambda')$ to be the smallest subgraph of $G(\Lambda)$ such that $w \in V'$ and
\begin{itemize}

\item[$\blacktriangleright$] if $v \in V'$ and $(v,u) \in E$, then $u \in V'$;

\item[$\blacktriangleright$] $E' = \{(u,v) \ | \ u,v \in V'\}$;

\item[$\blacktriangleright$] $\lambda' = \{(v,\Gamma' \sar \Delta') \ | \ v \in V' \text{ and } (v,\Gamma' \sar \Delta') \in \lambda \}$.

\end{itemize}
\end{definition}

\begin{definition}[The Translation $\switch$]\label{def:switch} Let $\Lambda$ be a treelike labelled sequent with root $w$. We inductively define $\switch(\Lambda) := \switch(G_{w}(\Lambda))$ to be the nested sequent obtained from the graph $G_{w}(\Lambda)$ 
as follows:
\begin{itemize}

\item[$\blacktriangleright$] If $G_{u}(\Lambda) = (V,E,\lambda)$, where $V = \{u\}$, $E = \emptyset$, and $\lambda = \{(u, \Gamma' \Rightarrow \Delta')\}$, then
$$
\switch(G_{u}(\Lambda)) := \Gamma' \far \Delta'.
$$

\item[$\blacktriangleright$] If $G_{u}(\Lambda) = (V,E,\lambda)$, where $u, v_{1}, \ldots, v_{n} \in V$, $(u,v_{1}), \ldots, (u,v_{n}) \in E$, and $(u, \Gamma' \Rightarrow \Delta') \in \lambda$, then 
$$
\switch(G_{u}(\Lambda)) := \Gamma' \far \Delta', [\switch(G_{v_{1}}(\Lambda))], \ldots, [\switch(G_{v_{n}}(\Lambda))].
$$

\end{itemize}
\end{definition}

We also define a converse translation $\switchtwo$ that translates nested sequents into labelled sequents. Whereas the former translation $\switch$ is useful in extracting our nested calculi from our labelled calculi, the translation $\switchtwo$ is useful in transferring proof-theoretic properties of the latter to the former as well as establishing the converse translation (see Sect.~\ref{section-7}). In order to define this translation, we first define two operations: (i) If $X$ is a multiset of formulae from $\lang$ or $\langfo$, then $w : X := \{w : A \ | \ A \in X\}$, that is, $w : X$ is the multiset of all formulae from $X$ labelled with $w$, and (ii) If $\Lambda_{1} = \R_{1}, \Gamma_{1} \sar \Delta_{1}$ and $\Lambda_{2} = \R_{2}, \Gamma_{2} \sar \Delta_{2}$, then we define the \emph{sequent composition} $\Lambda_{1} \oplus  \Lambda_{2} := \R_{1}, \R_{2}, \Gamma_{1}, \Gamma_{2} \sar \Delta_{1}, \Delta_{2}$.

\begin{definition}[The Translation $\switchtwo$]\label{def:switchtwo} Let $\Sigma$ be a nested sequent. We inductively define $\switchtwo(\Sigma) := \switchtwo(G_{0}(\Sigma))$ to be the labelled sequent obtained from the graph $G_{0}(\Sigma)$ as follows:
\begin{itemize}

\item[$\blacktriangleright$] If $G_{\sigma}(\Sigma) = (V_{\sigma},E_{\sigma},\lambda_{\sigma})$ with $V = \{\sigma\}$, $E = \emptyset$, $\lambda = \{(\sigma, X \far Y)\}$, and $\vv{a}$ are all parameters occurring in $X,Y$, then
$$
\switchtwo(G_{\sigma}(\Sigma)) := \vv{a} \in D_{\sigma}, \sigma : X \sar \sigma : Y .
$$

\item[$\blacktriangleright$] If $G_{\sigma}(\Sigma) = (V_{\sigma},E_{\sigma},\lambda_{\sigma})$ with $\sigma, \sigma . 1, \ldots, \sigma . n \in V$, $(\sigma,\sigma . 1), \ldots, (\sigma, \sigma . n) \in E$, $(\sigma, X \far Y) \in \lambda$, and $\vv{a}$ are all parameters occurring in $X,Y$, then 
\begin{flushleft}
$\switchtwo(G_{\sigma}(\Sigma)) := \Big(\sigma \leq \sigma . 1, \ldots, \sigma \leq \sigma . n, \vv{a} \in D_{\sigma}, \sigma : X \sar \sigma : Y\Big) $
\end{flushleft}
\begin{flushright}
$\oplus \Big(\switchtwo(G_{\sigma . 1}(\Sigma))\Big) \oplus \cdots \oplus \Big(\switchtwo(G_{\sigma . n}(\Sigma))\Big)$.
\end{flushright}
\end{itemize}
To simplify notation when making use of the $\switchtwo$ translation, we will use the same labels $w$, $u$, $v$, $\ldots$ employed in our labelled calculi in place of strings $\sigma$ of natural numbers.
\end{definition}

Before finishing this subsection, we introduce a fundamental notion useful for completing our translation from labelled to nested, which is based on the above definition---the notion of a \emph{nestedlike} labelled sequent:

\begin{definition}[Nestedlike]\label{def:nestedlike} We say that a labelled sequent $\Lambda$ is \emph{nestedlike} iff there exists a nested sequent $\Sigma$ such that $\switchtwo(\Sigma) = \Lambda$.
\end{definition}

It is not difficult to see that the following lemma holds:

\begin{lemma}\label{lem:nestedlike-implies-treelike}
Let $\Lambda := \R, \Gamma \sar \Delta$ be a nestedlike labelled sequent. Then, (i) $\Lambda$ is treelike, (ii) $\switchtwo(\switch(\Lambda)) = \Lambda$ (up to a change in labels), and (iii) there exists a domain atom $\unda \in D_{w} \in \R$ iff there exists a labelled formula $w : A(\unda) \in \Gamma, \Delta$.
\end{lemma}

\subsection{Setting the Stage for Refinement: Preliminary Results}

In order to extract Fitting's nested calculi from our labelled calculi, we will expand our labelled calculi with rules sufficient for the elimination of certain structural rules. Therefore, we first inflate our calculi to a large collection of rules, and use the new additions to systematically eliminate certain rules from each calculus, thus obtaining a refined variant of each calculus. To avoid lengthy names for our inflated and refined calculi, we will use the following abbreviations throughout the remainder of the paper (see Fig.~\ref{fig:new-rules} for newly introduced rules):
\begin{definition} The inflated ($\lint^{*}$, $\lintfond^{*}$, and $\lintfocd^{*}$) and refined ($\intl$, $\intfondl$, and $\intfocdl$) labelled calculi are defined as follows:

\begin{equation*}
\begin{split}
& \lint^{*} := \lint + \{\idnew, \negl, \negr, \implnew, \lift\} \\
& \lintfond^{*} := \lintfond + \{\idfond, \negl, \negr, \implnew, \lift, \existsndr, \allndl, \existsnedr, \allnedl\} \\
& \lintfocd^{*} := \lintfocd + \{\idfocd, \negl, \negr, \implnew, \lift, \existscdr, \allcdl, \existscedr, \allcedl, \allcdr\}
\end{split}
\end{equation*}
\begin{equation*}
\begin{split}
& \intl := \lint^{*} - \{\id, \botr, \imprl, \refl, \trans\} \\
& \intfondl := \lintfond^{*} - \{\idfo, \botl, \imprl, \refl, \trans, \existsr, \alll, \nd, \ned\} \\
& \intfocdl := \\
& \hspace*{2.5em} \lintfocd^{*} - \{\idfo, \botl, \imprl, \refl, \trans, \existsr, \alll, \allr, \nd, \cd, \ned\} 
\end{split}
\end{equation*}

\end{definition}

\begin{figure}[t]
\noindent\hrule

\begin{center}
\resizebox{\columnwidth}{!}{
\begin{tabular}{c c}
\AxiomC{}
\RightLabel{$\idnew$}
\UnaryInfC{$\R, \Gamma, w : p \Rightarrow w : p, \Delta$}
\DisplayProof

&

\AxiomC{}
\RightLabel{$\idfond^{\dag_{1}}$}
\UnaryInfC{$\R, \unda_{1} \in D_{v_{1}}, \ldots, \unda_{n} \in D_{v_{n}}, \Gamma, w : p(\vv{\unda}) \Rightarrow w : p(\vv{\unda}), \Delta$}
\DisplayProof
\end{tabular}
}
\end{center}

\begin{center}
\begin{tabular}{c}
\AxiomC{}
\RightLabel{$\idfocd^{\dag_{2}}$}
\UnaryInfC{$\R, \unda_{1} \in D_{v_{1}}, \ldots, \unda_{n} \in D_{v_{n}}, \Gamma, w : p(\vv{\unda}) \Rightarrow w : p(\vv{\unda}), \Delta$}
\DisplayProof
\end{tabular}
\end{center}

\begin{center}
\begin{tabular}{c @{\hskip 1em} c}
\AxiomC{$\R,w \leq v, v:A, \Gamma \Rightarrow \Delta$}
\RightLabel{$\negr^{\dag_{3}}$}
\UnaryInfC{$\R, \Gamma \Rightarrow \Delta, w : \neg A$}
\DisplayProof

&

\AxiomC{$\R, \unda \in D_{v}, \Gamma \Rightarrow \Delta, w: A(\unda/x), w: \exists x A$}
\RightLabel{$\existsndr^{\dag_{4}}$}
\UnaryInfC{$\R, \unda \in D_{v}, \Gamma \Rightarrow \Delta, w: \exists x A$}
\DisplayProof
\end{tabular}
\end{center}

\begin{center}
\begin{tabular}{c c}
\AxiomC{$\R, w: \neg A, \Gamma \Rightarrow w:A, \Delta$}
\RightLabel{$\negl$}
\UnaryInfC{$\R, w: \neg A, \Gamma \Rightarrow \Delta$}
\DisplayProof

&

\AxiomC{$\R, \unda \in D_{v}, w : A(\unda/x), w : \forall x A, \Gamma \Rightarrow \Delta$}
\RightLabel{$\allndl^{\dag_{4}}$}
\UnaryInfC{$\R, \unda \in D_{v}, w : \forall x A, \Gamma \Rightarrow \Delta$}
\DisplayProof
\end{tabular}
\end{center}

\begin{center}
\resizebox{\columnwidth}{!}{
\begin{tabular}{c c}
\AxiomC{$\R, \unda \in D_{v}, w : A(\unda/x), w : \forall x A, \Gamma \Rightarrow \Delta$}
\RightLabel{$\allcdl^{\dag_{5}}$}
\UnaryInfC{$\R, \unda \in D_{v}, w : \forall x A, \Gamma \Rightarrow \Delta$}
\DisplayProof
&
\AxiomC{$\R, \unda \in D_{w}, \Gamma \Rightarrow w : A(\unda/x), \Delta$}
\RightLabel{$\allcdr^{\dag_{6}}$}
\UnaryInfC{$\R, \Gamma \Rightarrow w : \forall x A, \Delta$}
\DisplayProof
\end{tabular}
}
\end{center}

\begin{center}
\resizebox{\columnwidth}{!}{
\begin{tabular}{c c}
\AxiomC{$\R, w \leq u, w:A, u:A, \Gamma \Rightarrow \Delta$}
\RightLabel{$\lift$}
\UnaryInfC{$\R, w \leq u, w:A, \Gamma \Rightarrow \Delta$}
\DisplayProof

&

\AxiomC{$\R, \unda \in D_{v}, \Gamma \Rightarrow \Delta, w: A(\unda/x), w: \exists x A$}
\RightLabel{$\existscdr^{\dag_{5}}$}
\UnaryInfC{$\R, \unda \in D_{v}, \Gamma \Rightarrow \Delta, w: \exists x A$}
\DisplayProof
\end{tabular}
}
\end{center}

\begin{center}
\begin{center}
\resizebox{\columnwidth}{!}{
\begin{tabular}{c c}
\AxiomC{$\R, \unda \in D_{v}, \Gamma \Rightarrow w : A(\unda / x), w : \exists x A, \Delta$}
\RightLabel{$\existsnedr^{\dag_{7}}$}
\UnaryInfC{$\R, \Gamma \Rightarrow w : \exists x A, \Delta$}
\DisplayProof

&

\AxiomC{$\R, \unda \in D_{v}, \Gamma, w : \forall x A, w : A(\unda / x) \sar \Delta$}
\RightLabel{$\allnedl^{\dag_{7}}$}
\UnaryInfC{$\R, \Gamma, w : \forall x A, \sar \Delta$}
\DisplayProof
\end{tabular}
}
\end{center}
\end{center}

\begin{center}
\begin{center}
\resizebox{\columnwidth}{!}{
\begin{tabular}{c c}
\AxiomC{$\R, \unda \in D_{v}, \Gamma \Rightarrow w : A(\unda / x), w : \exists x A, \Delta$}
\RightLabel{$\existscedr^{\dag_{8}}$}
\UnaryInfC{$\R, \Gamma \Rightarrow w : \exists x A, \Delta$}
\DisplayProof

&

\AxiomC{$\R, \unda \in D_{v}, \Gamma, w : \forall x A, w : A(\unda / x) \sar \Delta$}
\RightLabel{$\allcedl^{\dag_{8}}$}
\UnaryInfC{$\R, \Gamma, w : \forall x A, \sar \Delta$}
\DisplayProof
\end{tabular}
}
\end{center}
\end{center}

\begin{center}
\AxiomC{$\R, w : A \imp B, \Gamma \Rightarrow \Delta, w : A$}
\AxiomC{$\R, w : A \imp B, w : B, \Gamma \Rightarrow \Delta$}
\RightLabel{$\implnew$}
\BinaryInfC{$\R, w : A \imp B, \Gamma \Rightarrow \Delta$}
\DisplayProof
\end{center}

\hrule
\caption{Rules used to derive $\nint$, $\nintfond$, and $\nintfocd$ from $\lint$, $\lintfond$, and $\lintfocd$, respectively. Let $\R$ be the relational atoms occurring in the relevant rule; the side condition $\dag_{1}$ states that $v_{i} \leadsto^{\R} w$ for each $i \in \{1, \ldots, n\}$, $\dag_{2}$ states that $v_{i} \sim^{\R} w$ for each $i \in \{1, \ldots, n\}$, $\dag_{3}$ states that $v$ is an eigenvariable, $\dag_{4}$ states that $v \leadsto^{\R} w$, $\dag_{5}$ states that $v \sim^{\R} w$, $\dag_{6}$ states that $\unda$ is an eigenvariable, $\dag_{7}$ states that $v \leadsto^{\R} w$ and $\unda$ is an eigenvariable, and $\dag_{8}$ states that $v \sim^{\R} w$ and $\unda$ is an eigenvariable.}
\label{fig:new-rules}
\end{figure}

The next three sections will focus on the non-trivial refinement and translation of our labelled calculi into our nested calculi. As such, we present two useful lemmata below which will aid us in the aforementioned endeavor and allow for us to focus our attention on the refinement and translation from labelled to nested, thus  
freeing us from being sidetracked by too many auxiliary details. Moreover, we present the rules (mentioned in the abbreviations above and lemmata below) that we expand each labelled calculus with in Fig.~\ref{fig:new-rules} below. Some of these rules depend on the notion of a \emph{directed path} or \emph{undirected path}, whose definitions are as follows:

\begin{definition}[Directed Path]\label{def:directed-path} Let $\Lambda = \R, \Gamma \sar \Delta$ be a labelled sequent. We say that there exists a \emph{directed path} from $w$ to $u$ in $\R$ (written $w \leadsto^{\R} u$) iff $w = u$, or there exist worlds $v_{i}$ (with $i \in \{1, \ldots, n\}$) such that $w \leq v_{1}, \ldots, v_{n} \leq u \in \R$ holds (we stipulate that  $w \leq u \in \R$ when $n = 0$).
\end{definition}

\begin{definition}[Undirected Path]\label{def:undirected-path} Let $\Lambda = \R, \Gamma \sar \Delta$ be a labelled sequent and $w \sim v \in \{w \leq u, u \leq w\}$. We say that there exists an \emph{undirected path} from $w$ to $u$ in $\R$ (written $w \sim^{\R} u$) iff $w = u$, or there exist worlds $v_{i}$ (with $i \in \{1, \ldots, n\}$) such that $w \sim v_{1}, \ldots, v_{n} \sim u \in \R$ holds (we stipulate that $w \sim u \in \R$ when $n = 0$).
\end{definition}

The following lemma confirms that the expansion of our labelled calculi $\lint$, $\lintfond$, and $\lintfocd$ with rules from Fig.~\ref{fig:new-rules} preserves (to a degree) favorable proof-theoretic properties:

\begin{lemma}
\label{lem:extended-lint-properties} Let $\mathsf{G3X}^{*} \in \{\lintfond^{*}, \lintfocd^{*}\}$. The calculi $\lint^{*}$, $\lintfond^{*}$, and $\lintfocd^{*}$ have the following properties:
\begin{itemize}

\item[$(i)$] 

\begin{itemize}

\item[(a)] For all $A \in \lang$, $ \vdash_{\lint^{*}} \R,w \leq v, w : A, \Gamma \Rightarrow v : A, \Delta$;

\item[(b)] For all $A \in \lang$, $ \vdash_{\lint^{*}} \R,w:A,\Gamma \Rightarrow \Delta, w :A$; 

\item[(c)] For all $A \in \langfo$, $\vdash_{\mathsf{G3X}^{*}} \R,w \leq v, \vv{\unda} \in D_{w}, w : A(\vv{\unda}), \Gamma \Rightarrow v : A(\vv{\unda}), \Delta$; 

\item[(d)] For all $A \in \langfo$, $\vdash_{\mathsf{G3X}^{*}} \R, \vv{\unda} \in D_{w}, w:A(\vv{\unda}),\Gamma \Rightarrow \Delta, w :A(\vv{\unda})$;

\end{itemize}

\item[(ii)] The rules $\{(lsub),(psub),\wk,\ctrrel,\ctrr\}$ are hp-admissible;

\item[(iii)] With the exception of $\{\conrl,\existsl\}$, all rules are hp-invertible;

\item[(iv)] The rules $\{\conrl,\existsl\}$ are invertible;

\item[(v)] The rule $\ctrl$ is admissible.

\end{itemize}
\end{lemma}

\begin{proof} See Appendix~\ref{app:A} for details.
\qed
\end{proof}


The following lemma is useful in that it explicates what rules freely permute with each other, thus letting us focus only on the non-trivial cases in the sequel.

\begin{lemma}\label{lm:structural-rules-permutation} 
The following hold in $\lint^{*}$, $\lintfond^{*}$, and $\lintfocd^{*}$:
\begin{itemize}

\item[(i)] The $\refl$ and $\trans$ rules can be permuted above $\idnew$, $\botl$, $\conrl$, $\conrr$, $\disrl$, $\disrr$, $\implnew$, $\imprr$, $\negl$, $\negr$, $\ned$, $\existsl$, $\existsr$, and $\allcdr$.

\item[(ii)] The $\nd$, $\cd$, and $\ned$ rules can be permuted above $\id$, $\idnew$, $\botl$, $\conrl$, $\conrr$, $\disrl$, $\disrr$, $\imprl$, $\implnew$, $\imprr$, $\negl$, $\negr$, $\lift$, $\ned$, $\existsl$, $\existsnedr$, $\existscedr$, $\allnedl$, $\allcedl$, $\allr$, and $\allcdr$.


\end{itemize}

\end{lemma}

\begin{proof} The claim of (i) follows from the fact that none of the rules mentioned have active relational atoms of the form $w \leq u$ in the conclusion, and so, $\refl$ and $\trans$ may be freely permuted above each rule. The claim of (ii) follows from the fact that none of the rules mentioned contain active domain atoms in the conclusion, allowing for $\nd$, $\cd$, and $\ned$ to be permuted above each rule.
\qed
\end{proof}


\section{Deriving the Calculus $\nint$ from $\lint$}\label{section-4}


Deriving the calculus $\nint$ from $\lint$ depends on a crucial observation concerning labelled derivations: \emph{rules such as $\refl$ and $\trans$ allow for theorems to be derived in proofs containing non-treelike labelled sequents}. To demonstrate this fact, observe the following derivation in $\lint$:
\begin{center}
\AxiomC{$w \leq v, v \leq v, v : p \Rightarrow v : p$}
\RightLabel{$\refl$}
\UnaryInfC{$w \leq v, v : p \Rightarrow v : p$}
\RightLabel{$\imprr$}
\UnaryInfC{$\Rightarrow w : p \imp p$}
\DisplayProof
\end{center}
The initial sequent is non-treelike due to the presence of the $v \leq v$ relational atom; however, the application of $\refl$ deletes this structure from the initial sequent and produces a treelike sequent as the conclusion.

In fact, it is true in general that every labelled derivation ending with a treelike sequent 
can be partitioned into a top derivation consisting of non-treelike sequents, and a bottom derivation consisting of treelike sequents. Note that if a derivation ends with a treelike sequent, 
then the derivation must necessarily contain a bottom treelike fragment. By contrast, the top non-treelike fragment of the derivation may be empty (e.g. the minimal derivation of $\Rightarrow w : \bot \imp p$).

To demonstrate why the aforementioned partition always exists, suppose you are given a labelled derivation of a treelike sequent and consider the derivation in a bottom-up manner. 
Observe the each bottom-up application of a rule in $\lint$---with the exception of $\refl$ and $\trans$---will produce a treelike sequent (see Thm.~\ref{thm:treelike-derivations} for auxiliary details). If, however, at some point in the derivation $\refl$ or $\trans$ is applied, then all sequents above the inference will inherit the (un)directed cycle produced by the rule, thus producing the non-treelike fragment of the proof.

One can therefore imagine that permuting instances of the $\refl$ and $\trans$ rules upward in a given derivation would potentially increase the bottom treelike fragment of the derivation and decrease the top non-treelike fragment. As it so happens, this intuition is correct so long as we choose \emph{adequate} rules---that bottom-up preserve the treelike structure of sequents---to replace certain instances of the $\refl$ and $\trans$ rules in a derivation, when necessary. We will first examine permuting instances of the $\refl$ rule, and motivate which adequate rules we ought to add to our calculus in order to achieve the complete elimination of $\refl$. After, we will turn our attention toward eliminating the $\trans$ rule, and conclude the section by leveraging our results to show that $\nint$ can be derived from $\lint$.

Let us first observe an application of $\refl$ to an initial sequent obtained via the $\id$ rule. There are two possible cases to consider: either the relational atom principal in the initial sequent is active in the $\refl$ inference (shown below left), or it is not (shown below right):
\vspace*{-1em}
\begin{center}
\begin{tabular}{c c}
\AxiomC{}
\RightLabel{$\id$}
\UnaryInfC{$\R, w \leq w, w : p, \Gamma \Rightarrow \Delta, w : p$}
\RightLabel{$\refl$}
\UnaryInfC{$\R, w : p, \Gamma \Rightarrow \Delta, w : p$}
\DisplayProof

&

\AxiomC{}
\RightLabel{$\id$}
\UnaryInfC{$\R, u \leq u, w \leq v, w : p, \Gamma \Rightarrow \Delta, v : p$}
\RightLabel{$\refl$}
\UnaryInfC{$\R, w \leq v, w : p, \Gamma \Rightarrow \Delta, v : p$}
\DisplayProof
\end{tabular}
\end{center}
In the case shown above right, the end sequent is an instance of the $\id$ rule, regardless of if $u = w$, $u = v$, or $u$ is distinct from $w$ and $v$. The case shown above left however, indicates that we ought to add the $\idnew$ rule (see Fig.~\ref{fig:new-rules}) to our calculus if we aim to eliminate $\refl$ from any given derivation. 
These facts, coupled with Lem.~\ref{lm:structural-rules-permutation}, imply that any application of $\refl$ to an initial sequent, produces an initial sequent.

Concerning the remaining rules of $\lint$, we need only investigate the permutation of $\refl$ above the $\imprl$ rule, if we rely on Lem.~\ref{lm:structural-rules-permutation}. 
There are two cases: either the relational atom principal in the $\imprl$ inference is active in the $\refl$ inference, or it is not. The latter case is easily resolved, so we observe the former:
\vspace*{-2em}
\begin{center}
\scalebox{.95}{
\begin{tabular}{c}
\AxiomC{$\R, w \leq w, w : A \imp B, \Gamma \Rightarrow \Delta, w : A$}
\AxiomC{$\R, w \leq w, w : A \imp B, w : B, \Gamma \Rightarrow \Delta$}
\RightLabel{$\imprl$}
\BinaryInfC{$\R, w \leq w, w : A \imp B, \Gamma \Rightarrow \Delta$}
\RightLabel{$\refl$}
\UnaryInfC{$\R, w : A \imp B, \Gamma \Rightarrow \Delta$}
\DisplayProof
\end{tabular}
}
\end{center}
Applying $\refl$ to each premise of the $\imprl$ inference yields the following:
\begin{center}
\scalebox{.85}{
\begin{tabular}{c @{\hskip 1em} c}
\AxiomC{$\R, w \leq w, w : A \imp B, \Gamma \Rightarrow \Delta, w : A$}
\RightLabel{$\refl$}
\UnaryInfC{$\R, w : A \imp B, \Gamma \Rightarrow \Delta, w : A$}
\DisplayProof

&

\AxiomC{$\R, w \leq w, w : A \imp B, w : B, \Gamma \Rightarrow \Delta$}
\RightLabel{$\refl$}
\UnaryInfC{$\R, w : A \imp B, w : B, \Gamma \Rightarrow \Delta$}
\DisplayProof
\end{tabular}
}
\end{center}
The above observation suggests that we ought to add the $(\imp_{l}^{*})$ rule (see Fig.~\ref{fig:new-rules}) to our calculus if we wish to permute $\refl$ above the $\imprl$ rule; a single application of the $\implnew$ rule to the end sequents above gives the desired conclusion. 
With the $\implnew$ rule added to our calculus, we may freely permute the $\refl$ rule above any $\imprl$ inference. 

On the basis of our investigation, together with Lem.~\ref{lm:structural-rules-permutation}, we may conclude the following lemma:

\begin{lemma}\label{lm:refl-admiss} The $\refl$ rule is eliminable in $\lint + \{(id^{*}),(\imp^{*}_{l})\} - \trans$.

\end{lemma}

Let us turn our attention toward eliminating the $\trans$ rule from a labelled derivation. Since our aim is to eliminate \emph{both} $\refl$ and $\trans$ from any derivation, we assume that the rules $\{(id^{*}),(\imp^{*}_{l})\}$ have been added to our calculus.

By Lem.~\ref{lm:structural-rules-permutation}, $\trans$ permutes with $\botr$ and $\idnew$, so we only consider the $\id$ case. As with the $\refl$ rule, there are two cases to consider when permuting $\trans$ above an $\id$ inference: either, the active formula of $\trans$ is principal in $\id$, or it is not. In the latter case, the result of the $\trans$ rule is an initial sequent, implying that the $\trans$ rule may be eliminated from the derivation. The former case proves trickier and is explicitly given below:
\begin{center}
\begin{tabular}{c}
\AxiomC{}
\RightLabel{$\id$}
\UnaryInfC{$\R, w \leq u, u \leq v, w \leq v, w : p, \Gamma \Rightarrow \Delta, v : p$}
\RightLabel{$\trans$}
\UnaryInfC{$\R, w \leq u, u \leq v, w : p, \Gamma \Rightarrow \Delta, v : p$}
\DisplayProof
\end{tabular}
\end{center}
Observe that the end sequent is not an initial sequent as it is not obtainable from an $\id$, $(id^{*})$, or $\botr$ rule. The issue is solved by considering the $\lift$ rule (see Fig.~\ref{fig:new-rules}), 
which allows us to obtain the desired end sequent without the use of $\trans$, as the following derivation demonstrates:
\begin{center}
\AxiomC{}
\RightLabel{$(id^{*})$}
\UnaryInfC{$\R, w \leq u, u \leq v, w:p, u:p, v:p, \Gamma \Rightarrow v : p, \Delta$}
\RightLabel{$\lift$}
\UnaryInfC{$\R, w \leq u, u \leq v, w:p, u:p, \Gamma \Rightarrow v : p, \Delta$}
\RightLabel{$\lift$}
\UnaryInfC{$\R, w \leq u, u \leq v, w:p, \Gamma \Rightarrow v : p, \Delta$}
\DisplayProof
\end{center}
Thus, the addition of $\lift$ to our calculus resolves the issue of permuting $\trans$ above any initial sequent. Nevertheless, by Lem.~\ref{lm:structural-rules-permutation}, we still need to consider the permutation of $\trans$ above the $\imprl$ and $\lift$ rules. 
However, due to the following lemma, we may omit consideration of the $\imprl$ case.

\begin{lemma}\label{lm:implies-left-deriv} The rule $\imprl$ is admissible in $\lint + \{(id^{*}), (\imp^{*}_{l}), \lift\}$.

\end{lemma}

\begin{proof} The derivation below proves the admissibility of the rule:

\begin{center}
\resizebox{\columnwidth}{!}{
\begin{tabular}{c}
\AxiomC{$\R,x \leq y, x:A \imp B, \Gamma \Rightarrow \Delta, y:A$}
\RightLabel{$\wk$}
\dashedLine
\UnaryInfC{$\R,x \leq y, x:A \imp B,y:A \imp B, \Gamma \Rightarrow \Delta, y:A$}
\AxiomC{$\R,x \leq y, x:A \imp B, y:B, \Gamma \Rightarrow \Delta$}
\RightLabel{$\wk$}
\dashedLine
\UnaryInfC{$\R,x \leq y, x:A \imp B,y:A \imp B, y:B, \Gamma \Rightarrow \Delta$}
\RightLabel{$(\imp^{*}_{l})$}
\BinaryInfC{$\R, x \leq y, x:A \imp B, y:A \imp B, \Gamma \Rightarrow \Delta$}
\RightLabel{$\lift$}
\UnaryInfC{$\R, x \leq y, x:A \imp B, \Gamma \Rightarrow \Delta$}
\DisplayProof 
\end{tabular}
}
\end{center}
\qed
\end{proof}

Last, the $\trans$ rule is permutable with the $\lift$ rule. In the case where the principal relational atom of $\lift$ is not active in the ensuing $\trans$ application, the two rules freely permute. The alternative case is resolved as shown below:
\begin{flushleft}
\AxiomC{$\R, w \leq u, u \leq v, w \leq v, w: A, v : A, \Gamma \Rightarrow \Delta$}
\RightLabel{$\lift$}
\UnaryInfC{$\R, w \leq u, u \leq v, w \leq v, w: A, \Gamma \Rightarrow \Delta$}
\RightLabel{$\trans$}
\UnaryInfC{$\R, w \leq u, u \leq v, w: A, \Gamma \Rightarrow \Delta$}
\DisplayProof
\end{flushleft}
\begin{flushright}
\AxiomC{$\R, w \leq u, u \leq v, w \leq v, w: A,v : A, \Gamma \Rightarrow \Delta$}
\RightLabel{$\trans$}
\UnaryInfC{$\R, w \leq u, u \leq v, w: A,v : A, \Gamma \Rightarrow \Delta$}
\RightLabel{$\wk$}
\dashedLine
\UnaryInfC{$\R, w \leq u, u \leq v, w: A, u:A, v : A, \Gamma \Rightarrow \Delta$}
\RightLabel{$\lift$}
\UnaryInfC{$\R, w \leq u, u \leq v, w: A, u:A, \Gamma \Rightarrow \Delta$}
\RightLabel{$\lift$}
\UnaryInfC{$\R, w \leq u, u \leq v, w: A, \Gamma \Rightarrow \Delta$}
\DisplayProof
\end{flushright}



Hence, we obtain the following:

\begin{lemma}\label{lm:trans-admiss} The $\trans$ rule is eliminable in $\lint + \{(id^{*}),(\imp^{*}_{l}),\lift\} - \refl$.

\end{lemma}

Enough groundwork has been laid to state one of our main lemmata, which is also a consequence of the work in~\cite{Pim18}.

\begin{lemma}\label{lm:refl-trans-admiss} The $\refl$ and $\trans$ rules are admissible in the calculus $\lint + \{(id^{*}),(\imp^{*}_{l}),\lift\}$.

\end{lemma}

\begin{proof} Suppose we are given a derivation in $\lint + \{(id^{*}),(\imp^{*}_{l}),\lift\}$, and consider the topmost occurrence of either $\refl$ or $\trans$. If we can show that the $\refl$ rule permutes above the $\lift$ rule, then we may invoke Lem.~\ref{lm:refl-admiss} and Lem.~\ref{lm:trans-admiss} to conclude that each topmost occurrence of $\refl$ and $\trans$ can be eliminated from the given derivation in succession. This yields a $\refl$- and $\trans$-free proof of the end sequent and establishes the claim. Thus, we prove that the $\refl$ rule permutes above the $\lift$ rule.

In the case where the relational atom active in $\refl$ is not principal in the $\lift$ inference, the two rules may be permuted; the alternative case is resolved as shown below:
\begin{center}
\begin{tabular}{c @{\hskip 1em} c}

\AxiomC{$\R, w \leq w, w: A, w : A, \Gamma \Rightarrow \Delta$}
\RightLabel{$\lift$}
\UnaryInfC{$\R, w \leq w, w:A, \Gamma \Rightarrow \Delta$}
\RightLabel{$\refl$}
\UnaryInfC{$\R, w:A, \Gamma \Rightarrow \Delta$}
\DisplayProof

&

\AxiomC{}
\RightLabel{IH}
\dashedLine
\UnaryInfC{$\R, w: A, w : A, \Gamma \Rightarrow \Delta$}
\RightLabel{$\ctrl$}
\dashedLine
\UnaryInfC{$\R, w: A, \Gamma \Rightarrow \Delta$}
\DisplayProof

\end{tabular}
\end{center}
\qed
\end{proof}

The addition of the rules $\{\idnew, (\imp_{l}^{*}), \lift\}$ to our calculus and the above admissibility results demonstrate that we are readily advancing toward our goal of deriving $\nint$. Howbeit, 
our labelled calculus is still distinct since it makes use of the logical signature $\{\bot, \land, \lor, \imp\}$, whereas $\nint$ uses the signature $\{\neg, \land,\lor, \imp\}$. Therefore, we need to show that $\botr$ (we define $\bot := p \land \neg p$) is admissible in the presence of (labelled versions of) the $\negr$ and $\negl$ rules (see Fig.~\ref{fig:new-rules}). 
This admissibility result is explained in the theorem below.



\begin{theorem}
\label{thm:admiss-all-rules-propositional}
The rules $\{\id, \botr, \imprl, \refl, \trans\}$ are admissible in $\intl$.
\end{theorem}

\begin{proof} Follows from Lem.~\ref{lm:structural-rules-permutation}, Lem.~\ref{lm:refl-trans-admiss}, the fact that $\id$ is derivable using $\idnew$ and $\lift$, and the fact that $\botr$ is derivable from $(\neg_{l})$ and $\conrl$. Also, the admissibility of $\imprl$ is shown as in Lem.~\ref{lm:implies-left-deriv}.
\qed
\end{proof}

Although the above theorem is sufficient to conclude the completeness of $\intl$, we obtain a stronger result. As shown in the theorem below, we may conclude that $\intl$ is \emph{complete relative to treelike derivations} possessing the \emph{fixed root property}; i.e. for every $\int$-valid formula $A \in \lang$, (i) $\sar w :A$ is derivable in $\intl$ with a treelike derivation, and (ii) the label $w$ is the root of each treelike labelled sequent in the derivation. 

\begin{theorem}
\label{thm:treelike-derivations}
(i) The calculus $\intl$ is complete relative to treelike derivations with the fixed root property. (ii) Every derivation in $\intl$ of a treelike labelled sequent is a treelike derivation with the fixed root property.
\end{theorem}

\begin{proof} We prove (i), as (ii) is proven similarly. By Thm.~\ref{thm:admiss-all-rules-propositional} above, we know that $\intl$ is complete. To prove claim (i) then, we assume that we are given a derivation in $\intl$ of a labelled theorem, i.e. the end sequent is of the form $\sar w : A$. 
 It is not difficult to see that the derivation of $\Rightarrow w : A$ must be treelike with the fixed root property since applying inference rules from the calculus bottom-up to $\Rightarrow w : A$ either preserves relational structure or adds forward relational structure (e.g. $\imprr$ and $(\neg_{r})$), thus constructing a tree emanating from the root $w$.
\qed
\end{proof}

We may leverage the above theorem and the translation $\switch$ (from Def.~\ref{def:switch}) to complete the extraction of the nested calculus $\nint$ from the labelled calculus $\lint$ (via the calculus $\intl$). The theorem below shows that (the treelike fragment of) $\intl$ and $\nint$ are notational variants of one another:

\begin{theorem}\label{thm:refinement-prop}
(i) Every derivation in $\intl$ of a nestedlike labelled sequent $\Lambda$ is effectively translatable to a derivation of the nested sequent $\switch(\Lambda)$ in $\nint$. (ii) Every derivation in $\nint$ of a nested sequent $\Sigma$ is effectively translatable to a derivation of the labelled sequent $\switchtwo(\Sigma)$ in $\intl$.
\end{theorem}

\begin{proof} By Lem.~\ref{lem:nestedlike-implies-treelike}, we know that $\Lambda$ is treelike, implying that our derivation of $\Lambda$ is treelike by Thm.~\ref{thm:treelike-derivations}, which further implies that the translation function $\switch$ is defined for each labelled sequent in the derivation. It is straightforward to verify that each rule in $\intl$ translates to an instance of the corresponding rule in $\nint$ under the $\switch$ translation, and every rule of $\nint$ translates to an instance of the corresponding rule in $\intl$ under the $\switchtwo$ translation. 
\qed
\end{proof}

\section{Deriving the Calculus $\nintfo$ from $\lintfo$}\label{section-5}


In the preceding section, we observed that the elimination of the $\refl$ and $\trans$ rules yielded derivations where every labelled sequent was of a treelike shape. In this section (and the succeeding section) the same effect obtains via the elimination of these two structural rules. Therefore, the majority of this section is dedicated to investigating the effects of eliminating the $\nd$ and $\ned$ structural rules instead. 

Similar to our analysis of eliminating the $\refl$ and $\trans$ rules, we motivate and explain the \emph{adequate} rules we ought to add to $\lintfond$ in order to obtain the complete elimination of $\nd$ and $\ned$. 
Although the reader may refer to Fig.~\ref{fig:new-rules} to witness what rules are needed for the admissibility of $\nd$ and $\ned$, a central aim of this section is to demonstrate \emph{how} such rules emerge from our analysis when attempting to eliminate $\nd$ and $\ned$ from a derivation---this begets a deeper understanding of how the semantics of $\intfond$ affects the shape of rules in 
$\nintfond$ (via the labelled calculus $\lintfond$). Ultimately, our admissibility results will permit the extraction of $\nintfond$ from $\lintfond$.

\begin{lemma}
\label{lem:trans-refl-fond-elim}
The $\refl$ and $\trans$ rules are admissible in the calculus $\lintfond^{*}$.
\end{lemma}

\begin{proof} We prove the result by induction on the height of the given derivation and assume w.l.o.g. that the derivation ends with $\refl$ or $\trans$, and that no other instances of $\refl$ or $\trans$ occur in the derivation. The general result follows by considering a topmost occurrence of $\refl$ or $\trans$ and eliminating each topmost occurrence in succession until a proof free of $\refl$ and $\trans$ inferences is obtained. 

By Lem.~\ref{lm:structural-rules-permutation}, we need only show that $\refl$ and $\trans$ permute above the rules $\idfo, \idnd, \imprl$, $\lift$, $\existsndr$, $\existsnedr$, $\alll$, $\allndl$, $\allnedl$, and $\nd$. The cases of permuting $\refl$ and $\trans$ above $\lift$ are similar as for Thm.~\ref{thm:admiss-all-rules-propositional} and Lem.~\ref{lm:trans-admiss}, respectively. Also, $\imprl$ is admissible in the presence of $(\imp^{*}_{l})$ (similar to Lem.~\ref{lm:implies-left-deriv}), $\idfo$ is an instance of $\idnd$ (by Lem.~\ref{lem:old-instance-of-new-nd} below), and $\alll$ is an instance of $\allndl$ (also by Lem.~\ref{lem:old-instance-of-new-nd} below), so these cases may be omitted. 
Hence, we focus only on the nontrivial cases involving the $\idnd$, $\existsndr$, $\existsnedr$, $\allndl$, $\allnedl$, and $\nd$ rules, which are shown below to the left and resolved as shown below on the right. We prove the elimination of $\refl$ and refer the reader to Appendix~\ref{app:A} for the similar proof of $\trans$ elimination. 

\textit{Base case.} We let $\R' := \R, \unda_{1} \in D_{v_{1}}, \ldots, \unda_{n} \in D_{v_{n}}$ in the $\idnd$ case below. Observe that the end sequent of the derivation below left is an instance of $\idnd$ since the side condition that $v_{i} \leadsto^{\R'} w$ for $i \in \{1, \ldots, n\}$ holds: if none of the directed paths from $v_{i}$ to $w$ contain $u \leq u$, then the paths are present in $\R'$, and if a directed path from $v_{i}$ to $w$ does contain $u \leq u$, then deleting each occurrence of $u \leq u$ from the directed path gives a new path from $v_{i}$ to $w$ in $\R'$.
\begin{flushleft}
\AxiomC{}
\RightLabel{$\idnd$}
\UnaryInfC{$\R', u \leq u, \Gamma, w : p(\vv{\unda}) \Rightarrow w : p(\vv{\unda}), \Delta$}
\RightLabel{$\refl$}
\UnaryInfC{$\R', \Gamma, w : p(\vv{\unda}) \Rightarrow w : p(\vv{\unda}), \Delta$}
\DisplayProof
\end{flushleft}
\begin{flushright}
\AxiomC{}
\RightLabel{$\idnd$}
\UnaryInfC{$\R', \Gamma, w : p(\vv{\unda}) \Rightarrow w : p(\vv{\unda}), \Delta$}
\DisplayProof
\end{flushright}

\textit{Inductive step.} For the inductive step, we consider the non-trivial $\existsndr$, $\existsnedr$, $\allndl$, $\allnedl$, and $\nd$ cases.
\begin{flushleft}
\AxiomC{$\R, u \leq u, \unda \in D_{v}, \Gamma \Rightarrow w : A(\unda/x), w : \exists x A, \Delta$}
\RightLabel{$\existsndr$}
\UnaryInfC{$\R, u \leq u, \unda \in D_{v}, \Gamma \Rightarrow w : \exists x A, \Delta$}
\RightLabel{$\refl$}
\UnaryInfC{$\R, \unda \in D_{v}, \Gamma \Rightarrow w : \exists x A, \Delta$}
\DisplayProof
\end{flushleft}
\begin{flushright}
\AxiomC{}
\RightLabel{IH}
\dashedLine
\UnaryInfC{$\R, \unda \in D_{v}, \Gamma \Rightarrow w : A(\unda/x), w : \exists x A, \Delta$}
\RightLabel{$\existsndr$}
\UnaryInfC{$\R, \unda \in D_{v}, \Gamma \Rightarrow w : \exists x A, \Delta$}
\DisplayProof
\end{flushright}

\begin{flushleft}
\AxiomC{$\R, u \leq u, \unda \in D_{v}, \Gamma \Rightarrow w : A(\unda/x), w : \exists x A, \Delta$}
\RightLabel{$\existsnedr$}
\UnaryInfC{$\R, u \leq u, \Gamma \Rightarrow w : \exists x A, \Delta$}
\RightLabel{$\refl$}
\UnaryInfC{$\R, \Gamma \Rightarrow w : \exists x A, \Delta$}
\DisplayProof
\end{flushleft}
\begin{flushright}
\AxiomC{}
\RightLabel{IH}
\dashedLine
\UnaryInfC{$\R, \unda \in D_{v}, \Gamma \Rightarrow w : A(\unda/x), w : \exists x A, \Delta$}
\RightLabel{$\existsnedr$}
\UnaryInfC{$\R, \Gamma \Rightarrow w : \exists x A, \Delta$}
\DisplayProof
\end{flushright}

\begin{flushleft}
\AxiomC{$\R, u \leq u, \unda \in D_{v}, w : A(\unda/x), w : \forall x A, \Gamma \Rightarrow \Delta$}
\RightLabel{$\allndl$}
\UnaryInfC{$\R, u \leq u, \unda \in D_{v}, w : \forall x A, \Gamma \Rightarrow \Delta$}
\RightLabel{$\refl$}
\UnaryInfC{$\R, \unda \in D_{v}, w : \forall x A, \Gamma \Rightarrow \Delta$}
\DisplayProof
\end{flushleft}
\begin{flushright}
\AxiomC{}
\RightLabel{IH}
\dashedLine
\UnaryInfC{$\R, \unda \in D_{v}, w : A(\unda/x), w : \forall x A, \Gamma \Rightarrow \Delta$}
\RightLabel{$\allndl$}
\UnaryInfC{$\R, \unda \in D_{v}, w : \forall x A, \Gamma \Rightarrow \Delta$}
\DisplayProof
\end{flushright}

\begin{flushleft}
\AxiomC{$\R, u \leq u, \unda \in D_{v}, \Gamma, w : A(\unda/x), w : \forall x A \Rightarrow \Delta$}
\RightLabel{$\allnedl$}
\UnaryInfC{$\R, u \leq u, \Gamma, w : \forall x A \Rightarrow \Delta$}
\RightLabel{$\refl$}
\UnaryInfC{$\R, w : \forall x A, \Gamma \Rightarrow \Delta$}
\DisplayProof
\end{flushleft}
\begin{flushright}
\AxiomC{}
\RightLabel{IH}
\dashedLine
\UnaryInfC{$\R, \unda \in D_{v}, \Gamma, w : A(\unda/x), w : \forall x A, \Gamma \Rightarrow \Delta$}
\RightLabel{$\allnedl$}
\UnaryInfC{$\R, w : \forall x A, \Gamma \Rightarrow \Delta$}
\DisplayProof
\end{flushright}

\begin{flushleft}
\AxiomC{$\R, w \leq w, \unda \in D_{w}, \unda \in D_{w}, \Gamma \Rightarrow \Delta$}
\RightLabel{$\nd$}
\UnaryInfC{$\R, w \leq w, \unda \in D_{w}, \Gamma \Rightarrow \Delta$}
\RightLabel{$\refl$}
\UnaryInfC{$\R, \unda \in D_{w}, \Gamma \Rightarrow \Delta$}
\DisplayProof
\end{flushleft}
\begin{flushright}
\AxiomC{$\R, w \leq w, \unda \in D_{w}, \unda \in D_{w}, \Gamma \Rightarrow \Delta$}
\RightLabel{$\ctrrel$}
\dashedLine
\UnaryInfC{$\R, w \leq w, \unda \in D_{w}, \Gamma \Rightarrow \Delta$}
\RightLabel{IH}
\dashedLine
\UnaryInfC{$\R, \unda \in D_{w}, \Gamma \Rightarrow \Delta$}
\DisplayProof
\end{flushright}
We now argue that the side condition ($v \leadsto^{\R} w$) of $\existsndr$, $\existsnedr$, $\allndl$, and $\allnedl$ continues to hold after applying IH above. If the directed path from $v$ to $w$ does not go through $u$, then the side condition trivially holds. Alternatively, if the directed path from $v$ to $w$ does go through $u$, then by deleting each occurrence of $u \leq u$ from the path, we obtain a new directed path from $v$ to $w$, which continues to be present after the invocation of IH.
\qed
\end{proof}


We now analyze the admissibility of $\nd$ in $\lintfond^{*} - \{\refl, \trans\}$, and afterward, analyze the admissibility of the $\ned$ rule. 
It will be instructive to first consider the case of permuting $\nd$ above the $\existsr$ rule solely. It will be seen that in order to permute $\nd$ above $\existsr$, it is sufficient to strengthen the $\existsr$ rule through the addition of a side condition; this motivates a similar side condition that we will impose on $\idfo$ (yielding the rule $\idnd$) and on $\alll$ (yielding the rule $\allndl$).


To begin our analysis, observe (1) a non-trivial application of $\nd$ after an $\existsr$ inference, and (2) the result of applying $\nd$ to the top sequent in (1). 
\begin{equation}
\AxiomC{$\R, v \leq w, \unda \in D_{v}, \unda \in D_{w}, \Gamma \Rightarrow \Delta, w : A(\unda/x), w : \exists x A$}
\RightLabel{$\existsr$}
\UnaryInfC{$\R, v \leq w, \unda \in D_{v}, \unda \in D_{w}, \Gamma \Rightarrow \Delta, w : \exists x A$}
\RightLabel{$\nd$}
\UnaryInfC{$\R, v \leq w, \unda \in D_{v}, \Gamma \Rightarrow \Delta, w : \exists x A$}
\DisplayProof
\end{equation}

\begin{equation}
\AxiomC{$\R, v \leq w, \unda \in D_{v}, \unda \in D_{w}, \Gamma \Rightarrow \Delta, w : A(\unda/x), w : \exists x A$}
\RightLabel{$\nd$}
\UnaryInfC{$\R, v \leq w, \unda \in D_{v}, \Gamma \Rightarrow \Delta, w : A(\unda/x), w : \exists x A$}
\DisplayProof
\end{equation}

In (2), the application of $\nd$ has generated a labelled sequent where the domain atom (i.e. $\unda \in D_{v}$) is associated with the label $v$ \emph{one step in the past} of $w$ (due to the presence of the $v \leq w$ relational atom). One possible solution to overcome this obstacle, and to allow for the permutation of $\nd$ above $\existsr$, would be to add the following rule to our calculus:
\begin{center}
\begin{tabular}{c}
\AxiomC{$\R, v \leq w, \unda \in D_{v}, \Gamma \Rightarrow \Delta, w: A(\unda / x), w: \exists x A$}
\RightLabel{$(\exists_{r})'$}
\UnaryInfC{$\R, v \leq w, \unda \in D_{v}, \Gamma \Rightarrow \Delta, w: \exists x A$}
\DisplayProof
\end{tabular}
\end{center}
In the rule above, the domain atom $\unda \in D_{v}$ is associated with the label $v$, which is \emph{one step in the past} of $w$. 
Although the addition of the above rule solves the problem of permuting $\nd$ above the $\existsr$ rule, its inclusion into our calculus would inevitably place us in a similar, undesirable circumstance:
\begin{equation*}
\AxiomC{$\R, u \leq v, v \leq w, \unda \in D_{u}, \unda \in D_{v}, \Gamma \Rightarrow \Delta, w: A(\unda / x), w: \exists x A$}
\RightLabel{$(\exists_{r})'$}
\UnaryInfC{$\R, u \leq v, v \leq w, \unda \in D_{u}, \unda \in D_{v}, \Gamma \Rightarrow \Delta, w: \exists x A$}
\RightLabel{$\nd$}
\UnaryInfC{$\R, u \leq v, v \leq w, \unda \in D_{u}, \Gamma \Rightarrow \Delta, w: \exists x A$}
\DisplayProof
\end{equation*}
\begin{equation*}
\AxiomC{$\R, u \leq v, v \leq w, \unda \in D_{u}, \unda \in D_{v}, \Gamma \Rightarrow \Delta, w: A(\unda / x), w: \exists x A$}
\RightLabel{$\nd$}
\UnaryInfC{$\R, u \leq v, v \leq w, \unda \in D_{u}, \Gamma \Rightarrow \Delta, w: A(\unda / x), w: \exists x A$}
\DisplayProof
\end{equation*}
It appears that now we would need to add a new version of the $\existsr$ rule allowing us to delete $w : A(\unda)$ when the required relational atom $\unda \in D_{u}$ is associated with a label (viz. $u$) that is \emph{two steps in the past} of $w$ (due to the occurrence of $u \leq v, v \leq w$). Including such a rule however, would inevitably force us to add another existential rule where the required relational atom is associated with a label \emph{three steps in the past}---such rule inclusions would then continue ad infinitum. To resolve the problem, and capture all such possibilities, we add the following rule to our calculus, which has the side condition $v \leadsto^{\R} w$ (see Def.~\ref{def:directed-path} and Fig.~\ref{fig:new-rules}):
\begin{center}
\AxiomC{$\R, \unda \in D_{v}, \Gamma \Rightarrow \Delta, w: A(\unda/x), w: \exists x A$}
\RightLabel{$\existsndr$}
\UnaryInfC{$\R, \unda \in D_{v}, \Gamma \Rightarrow \Delta, w: \exists x A$}
\DisplayProof
\end{center}
This rule generalizes the shifting behaviour of the domain atom witnessed above to include all possible \emph{shifts}.
Furthermore, due to the fact that the labelled sequents we are considering encode an $\intfond$-model---which satisfies the \emph{nested domain} condition---it is appropriate to impose the above side condition, which essentially states that \emph{if an object is an element of some domain $D_{v}$ in the past of $w$, then it is an element of the domain $D_{w}$}. It is straightforward to verify that the $\existsndr$ rule is sound relative to $\intfond$-models. 

If one analyzes the cases of permuting $\nd$ above the $\idfo$ and $\alll$ rules, then they will observe the same shifting behavior of domain atoms as in the $\existsr$ case. We therefore impose a similar side condition on $\idfo$ and $\alll$ to obtain the $\idnd$ and $\allndl$ rules (see Fig.~\ref{fig:new-rules}). It is not difficult to see that the following holds:

\begin{lemma}\label{lem:old-instance-of-new-nd}
The $\idfo$, $\existsr$, and $\alll$ rules are instances of the $\idnd$, $\existsndr$, and $\allndl$ rules, respectively.
\end{lemma}

The lemma below shows that the $\idnd$, $\allndl$, and $\existsndr$ rules are sufficient to prove the admissibility of $\nd$:

\begin{lemma}
\label{lm:nd-admiss}
The rule $\nd$ is admissible in the calculus $\lintfond^{*} - \{\refl,\trans\}$.
\end{lemma}

\begin{proof} We prove the result by induction on the height of the given derivation. By Lem.~\ref{lm:structural-rules-permutation} and Lem.~\ref{lem:trans-refl-fond-elim}, we need only consider the non-trivial $\idnd$, $\existsndr$, and $\allndl$ cases.\\

\textit{Base case.} We let $\R' := \R, u \leq v, \unda_{1} \in D_{v_{1}}, \ldots, \unda_{n} \in D_{v_{n}}$ and assume that both $v \leadsto^{\R'} w$ and $v_{i} \leadsto^{\R'} w$ (for $i \in \{1, \ldots, n\}$) hold. Observe that the end sequent is an instance of $\idnd$ since (i) the paths from $v_{i}$ to $w$ are still present and (ii) there is a directed path from $u$ to $w$ composed of the relational atom $u \leq v$ and the direct path from $v$ to $w$ (i.e. $u \leadsto^{\R'} w$ holds). 
\begin{center}
\begin{tabular}{c}
\AxiomC{}
\RightLabel{$\idnd$}
\UnaryInfC{$\R', \undb \in D_{u}, \undb \in D_{v}, \Gamma, w : p(\vv{\unda},\undb) \Rightarrow w : p(\vv{\unda},\undb), \Delta$}
\RightLabel{$\nd$}
\UnaryInfC{$\R', \undb \in D_{u}, \Gamma, w : p(\vv{\unda},\undb) \Rightarrow w : p(\vv{\unda},\undb), \Delta$}
\DisplayProof
\end{tabular}
\end{center}

\textit{Inductive step.} We provide each proof below showing that $\nd$ can be permuted above $\existsndr$ and $\allndl$ (we consider only the non-trivial cases). Also, we assume that there is a directed path from $v$ to $w$, i.e. the side condition $v \leadsto^{\R} w$ holds. 
\begin{flushleft}
\AxiomC{$\R, u \leq v, \unda \in D_{u}, \unda \in D_{v}, \Gamma \Rightarrow \Delta, w : A(\unda/x), w : \exists x A$}
\RightLabel{$\existsndr$}
\UnaryInfC{$\R, u \leq v, \unda \in D_{u}, \unda \in D_{v}, \Gamma \Rightarrow \Delta, w : \exists x A$}
\RightLabel{$\nd$}
\UnaryInfC{$\R, u \leq v, \unda \in D_{u}, \Gamma \Rightarrow \Delta, w : \exists x A$}
\DisplayProof
\end{flushleft}
\begin{flushright}
\AxiomC{ }
\RightLabel{IH}
\dashedLine
\UnaryInfC{$\R, u \leq v, \unda \in D_{u}, \Gamma \Rightarrow \Delta, w : A(\unda/x), w : \exists x A$}
\RightLabel{$\existsndr$}
\UnaryInfC{$\R, u \leq v, \unda \in D_{u}, \Gamma \Rightarrow \Delta,  w : \exists x A$}
\DisplayProof
\end{flushright}

\begin{flushleft}
\AxiomC{$\R, u \leq v, \unda \in D_{u}, \unda \in D_{v}, \Gamma, w : \forall x A, w : A(\unda/x) \Rightarrow \Delta$}
\RightLabel{$\allndl$}
\UnaryInfC{$\R, u \leq v, \unda \in D_{u}, \unda \in D_{v}, \Gamma, w : \forall x A \Rightarrow \Delta$}
\RightLabel{$\nd$}
\UnaryInfC{$\R, u \leq v, \unda \in D_{u}, \Gamma, w : \forall x A \Rightarrow \Delta$}
\DisplayProof
\end{flushleft}
\begin{flushright}
\AxiomC{ }
\RightLabel{IH}
\dashedLine
\UnaryInfC{$\R, u \leq v, \unda \in D_{u}, \Gamma, w : \forall x A, w : A(\unda/x) \Rightarrow \Delta$}
\RightLabel{$\allndl$}
\UnaryInfC{$\R, u \leq v, \unda \in D_{u}, \Gamma, w : \forall x A \Rightarrow \Delta$}
\DisplayProof
\end{flushright}
By assumption, there is a directed path from $v$ to $w$ in $\R$. Observe that the side condition is still satisfied after the invocation of IH since there is a directed path from $u$ to $w$ composed of the relational atom $u \leq v$ and the directed path from $v$ to $w$ in $\R$.
\qed
\end{proof}

We now turn our attention toward analyzing the admissibility of $\ned$ in the calculus $\lintfond^{*} - \{\refl,\trans,\nd\}$. We will see that sufficient modifications must be made to the $\existsndr$ and $\allndl$ rules to necessitate the elimination of $\ned$ from a given derivation. (NB. Due to Lem.~\ref{lem:old-instance-of-new-nd}, we may omit consideration of permuting $\ned$ above $\existsr$ and $\alll$.)
Let us consider non-trivial applications of $\ned$ after an $\existsndr$ and $\allndl$ inference: 
\begin{center}
\resizebox{\columnwidth}{!}{
\begin{tabular}{c c}
\AxiomC{$\R, \unda \in D_{v}, \Gamma \Rightarrow w : A(\unda / x), w : \exists x A, \Delta$}
\RightLabel{$\existsndr$}
\UnaryInfC{$\R, \unda \in D_{v}, \Gamma \Rightarrow w : \exists x A, \Delta$}
\RightLabel{$\ned$}
\UnaryInfC{$\R, \Gamma \Rightarrow w : \exists x A, \Delta$}
\DisplayProof

&

\AxiomC{$\R, \unda \in D_{v}, \Gamma, w : \forall x A, w : A(\unda / x) \sar \Delta$}
\RightLabel{$\allndl$}
\UnaryInfC{$\R, \unda \in D_{v}, \Gamma, w : \forall x A, \sar \Delta$}
\RightLabel{$\ned$}
\UnaryInfC{$\R, \Gamma, w : \forall x A, \sar \Delta$}
\DisplayProof
\end{tabular}
}
\end{center}
Observe that the $\ned$ rule cannot be permuted above the $\existsndr$ or $\allndl$ rules due to the existence of the labelled formula $w : A(\unda / x)$ in the top sequent which violates the side condition on $\ned$ stating that $\unda$ must be an eigenvariable. Our solution to this issue is to \emph{absorb} the $\ned$ rule into the $\existsndr$ and $\allndl$ rules yielding the new versions $\existsnedr$ and $\allnedl$:
\begin{center}
\resizebox{\columnwidth}{!}{
\begin{tabular}{c c}
\AxiomC{$\R, \unda \in D_{v}, \Gamma \Rightarrow w : A(\unda / x), w : \exists x A, \Delta$}
\RightLabel{$\existsnedr$}
\UnaryInfC{$\R, \Gamma \Rightarrow w : \exists x A, \Delta$}
\DisplayProof

&

\AxiomC{$\R, \unda \in D_{v}, \Gamma, w : \forall x A, w : A(\unda / x) \sar \Delta$}
\RightLabel{$\allnedl$}
\UnaryInfC{$\R, \Gamma, w : \forall x A, \sar \Delta$}
\DisplayProof
\end{tabular}
}
\end{center}
where we impose the side condition that $v \leadsto^{\R} w$ and $\unda$ is an eigenvariable. This side condition essentially states that \emph{there exists an element of $D_{w}$, which is also in a domain $D_{v}$ in the past of $w$}. It can be shown that each rule is sound relative to $\intfond$-models, and relies both on the nested domain condition and the fact that such domains are \emph{non-empty}, i.e. \emph{inhabited}.

\begin{lemma}
\label{lm:ned-admiss}
The rule $\ned$ is admissible in $\lintfond^{*} - \{\refl,\trans,\nd\}$.
\end{lemma}

\begin{proof} We prove the result by induction on the height of the given derivation. The base case is simple as any application of $\ned$ to $\idnd$ or $\botl$ yields another instance of the rule. By Lem.~\ref{lm:structural-rules-permutation} and Lem.~\ref{lem:old-instance-of-new-nd}, the only cases we need to consider in the inductive step are the non-trivial $\existsndr$ and $\allndl$ cases; however, as we saw above, the non-trivial cases can be handled by applying $\existsnedr$ or $\allnedl$, respectively.
\qed
\end{proof}

\begin{theorem}
\label{thm:admiss-all-rules-fond}
The rules $\{\idfo, \botl, \imprl, \refl, \trans, \existsr, \alll, \nd, \ned\}$ are admissible in the calculus $\lintfond^{*}$.
\end{theorem}

\begin{proof} The admissibility of $\{\refl,\trans\}$, $\{\idfo, \existsr, \alll\}$, $\nd$, and $\ned$ follow from Lem.~\ref{lem:trans-refl-fond-elim}, Lem.~\ref{lem:old-instance-of-new-nd}, Lem.~\ref{lm:nd-admiss}, and Lem.~\ref{lm:ned-admiss}, respectively. Admissibility of $\impl$ is similar to Lem.~\ref{lm:implies-left-deriv} and admissibility of $\botl$ follows from the fact that the rule can be derived using $\negl$ and $\landl$.
\qed
\end{proof}


\begin{theorem}
\label{thm:treelike-derivations-fond}
(i) The calculus $\intfondl$ is complete relative to treelike derivations with the fixed root property. (ii) Every derivation in $\intfondl$ of a treelike labelled sequent is a treelike derivation with the fixed root property.
\end{theorem}

\begin{proof}
Similar to Thm.~\ref{thm:treelike-derivations}.
\qed
\end{proof}

Although the above theorem is fundamental for translating derivations in $\intfondl$ into derivations in $\nintfond$, additional work needs to be done to ensure that the $\existsndr$ and $\allndl$ rules properly translate to instances of the $\existsr$ and $\alll$ rules. The side condition imposed on $\existsndr$ and $\allndl$ relies on the existence of a domain atom containing a parameter $\unda$, whereas the side condition imposed on the nested $\existsr$ and $\alll$ rules stipulates that the parameter $\unda$ must either be an eigenvariable or available. The question arises: how do we reconcile these seemingly distinct side conditions?

Interestingly, our process of refinement has offered us an unadulterated perspective of the nested $\existsr$ and $\alll$ rules.  As will be argued below (Thm.~\ref{thm:Intfondl-to-Nintfond}), the rules $\existsnedr$ and $\allnedl$ in $\intfondl$ translate to instances of the nested $\existsr$ and $\alll$ rules where the relevant parameter of the inference is an \emph{eigenvariable}, and the rules $\existsndr$ and $\allndl$ translate to instances of the nested $\existsr$ and $\alll$ rules where the relevant parameter of the inference is \emph{available}. In other words, what are single rules with a disjunctive side condition in the nested setting have bifurcated into distinct sets of rules with atomic side conditions in the labelled setting.

In order to ensure that instances of $\existsndr$ and $\allndl$ properly translate to instances of the nested $\existsr$ and $\alll$ rules then, we must show that if the side condition of the former rules holds, then the side condition of the latter rules holds. This motivates the definition of a \emph{nested form} derivation:

\begin{definition}[Nested Form]\label{def:nested-form} Let $\Pi$ be a derivation of a labelled sequent $\Lambda$ in $\intfondl$. We define $\Pi$ to be in \emph{nested form} iff (i) $\Lambda$ is nestedlike, and (ii) for each instance of $\existsndr$ and $\allndl$ in $\Pi$:
\begin{center}
\resizebox{\columnwidth}{!}{
\begin{tabular}{c c}
\AxiomC{$\R, \unda \in D_{v}, \Gamma \Rightarrow \Delta, w: A(\unda/x), w: \exists x A$}
\RightLabel{$\existsndr$}
\UnaryInfC{$\R, \unda \in D_{v}, \Gamma \Rightarrow \Delta, w: \exists x A$}
\DisplayProof

&

\AxiomC{$\R, \unda \in D_{v}, w : A(\unda/x), w : \forall x A, \Gamma \Rightarrow \Delta$}
\RightLabel{$\allndl$}
\UnaryInfC{$\R, \unda \in D_{v}, w : \forall x A, \Gamma \Rightarrow \Delta$}
\DisplayProof
\end{tabular}
}
\end{center}
there exists a formula $u :B(\unda)$ such that $u \leadsto^{\R} w$ in the conclusion of the inference.
\end{definition}

Assuming we are given a nested form derivation in $\intfondl$, condition (ii) of the above definition ensures that 
 instances of $\existsndr$ and $\allndl$ will properly translate into instances of $\existsr$ and $\alll$ in $\nintfond$. Additionally, nested form proofs are significant because one can view the space of nested form proofs in $\intfondl$ as being equivalent to the space of proofs generated by $\nintfond$, as confirmed by Thm.~\ref{thm:Intfondl-to-Nintfond} and~\ref{thm:Nintfond-to-Intfondl} below. Before we establish this correspondence however, the following lemma is needed for the proof of  Thm.~\ref{thm:Intfondl-to-Nintfond}, that is, for showing that if a nestedlike labelled sequent $\Lambda$ is provable in $\intfondl$, then $\switch(\Lambda)$ is provable in $\nintfond$.

\begin{lemma}\label{lem:Nestedlike-sequents-have-nested-form-proofs}
If a nestedlike labelled sequent $\Lambda$ is derivable in $\intfondl$, then $\Lambda$ has a nested form derivation in $\intfondl$.
\end{lemma}

The truth of the above lemma can be seen to follow from Thm.~\ref{thm:Nintfond-to-Intfondl} below (which states that each derivation in $\nintfond$ of a nested sequent $\Sigma$ can be translated into a nested form derivation of $\switchtwo(\Sigma)$ in $\intfondl$) in conjunction with a few facts: (i) $\intfondl$ is sound, (ii) one can show by semantic arguments that if a nestedlike labelled sequent $\Lambda$ is valid, then so is $\switch(\Lambda)$, (iii) by Fitting's work in~\cite{Fit14}, we know that $\nintfond$ is strongly complete (Thm.~\ref{thm:strong-soundness-completeness-nested}), and (iv) for any nestedlike labelled sequent $\Lambda$, $\switchtwo(\switch(\Lambda)) = \Lambda$ (Lem.~\ref{lem:nestedlike-implies-treelike}).\footnote{The soundness of $\intfondl$ is established by Thm.~\ref{thm:refined-to-labelled} in Sect.~\ref{section-7}, where it is shown that every derivation in $\intfondl$ can be effectively transformed into a derivation of the same labelled sequent in $\lintfond$. Since $\lintfond$ is sound by Thm.~\ref{thm:lint-properties}, this establishes the soundness of $\intfondl$.} Hence, given a derivable nestedlike labelled sequent $\Lambda$, which is therefore valid since $\intfondl$ is sound by (i), it follows that $\switch(\Lambda)$ is valid by (ii), meaning that $\switch(\Lambda)$ is derivable by (iii), and so by Thm.~\ref{thm:Nintfond-to-Intfondl} below, $\switchtwo(\switch(\Lambda))$ has a nested form derivation, i.e. $\Lambda$ has a nested form derivation by (iv). Although this argument serves as a proof of Lem.~\ref{lem:Nestedlike-sequents-have-nested-form-proofs}, it is not entirely satisfactory. 

The above argument relies on the completeness of $\nintfond$, and so, if our aim were to extract $\nintfond$ from $\lintfond$ without any prior knowledge of $\nintfond$ or its properties, or to show $\nintfond$ complete via this extraction, then since the above argument relies on the completeness of $\nintfond$ (i.e. Thm.~\ref{thm:strong-soundness-completeness-nested}), we would fail to reach our aim. The author sees at least two possible strategies for confirming the above lemma without relying on the completeness of $\nintfond$. In the first method, one could write a (potentially non-terminating) proof-search procedure for $\intfondl$ that takes a nestedlike labelled sequent $\Lambda$ as input and attempts to construct a nested form proof of $\Lambda$, showing that a counter-model for $\Lambda$ can be constructed given that the procedure does not halt with a nested form proof of the input. Such a method of proof would effectively prove the completeness of $\intfondl$ relative to valid nestedlike labelled sequents while showing that every such sequent possesses a nested form proof. The second method consists of showing how to algorithmically transform any derivation of a nestedlike labelled sequent $\Lambda$ in $\intfondl$ into a nested form proof of $\Lambda$. Such methods of proof are desirable as they would effectively demonstrate that $\nintfond$ inherits completeness from its associated labelled calculus $\lintfond$, however, we leave such results to future work, and continue our investigation of the correspondence between semantic and nested systems.

\begin{theorem}\label{thm:Intfondl-to-Nintfond}
If a nestedlike labelled sequent $\Lambda$ is derivable in $\intfondl$, then $\switch(\Lambda)$ is derivable in $\nintfond$.
\end{theorem}

\begin{proof} Since $\Lambda$ is nestedlike and derivable by assumption, we know that $\Lambda$ is treelike (by Lem.~\ref{lem:nestedlike-implies-treelike}) and has a nested form derivation $\Pi$ (by Lem.~\ref{lem:Nestedlike-sequents-have-nested-form-proofs}), implying that $\Pi$ is treelike by Thm.~\ref{thm:treelike-derivations-fond}. Hence, the translation function $\switch$ is defined for each labelled sequent in $\Pi$. It is straightforward to verify that each rule in $\intfondl$---with the exception of $\existsndr$, $\existsnedr$, $\allndl$, and $\allnedl$---translates to an instance of the corresponding rule in $\nintfond$ under the $\switch$ translation. We show the translation of the $\existsndr$ and $\existsnedr$ inferences below; the $\allndl$ and $\allnedl$ inferences are similar.

In the $\existsndr$ inference top-left below, we know that there exists a labelled formula $u : B(\unda) \in \Gamma, \Delta$ where $u \leadsto^{\R} w$ since our given derivation is in nested form (Def.~\ref{def:nested-form}). 
 Hence, after applying $\switch$ to the premise of the $\existsndr$ inference, the parameter $\unda$ will occur in the formula $B(\unda)$ that is either in the component $\Gamma \restriction w \far A(\unda/x), \exists x A, \Delta \restriction w$ (and is different than the occurrence of $A(\unda/x)$) or is in the past of the component in the nested sequent $\Sigma\{\Gamma \restriction w \far A(\unda/x), \exists x A, \Delta \restriction w\}$. Hence, the parameter $\unda$ is available, and so, the necessary side condition of the $\existsr$ inference is met. In the $\existsnedr$ inference below, the parameter $\unda$ is an eigenvariable and continues to be so after applying the translation $\switch$, implying that the nested $\existsr$ inference is valid.
\begin{flushleft}
\AxiomC{$\R, \unda \in D_{v}, \Gamma \Rightarrow \Delta, w: A(\unda/x), w: \exists x A$}
\RightLabel{$\existsndr$}
\UnaryInfC{$\R, \unda \in D_{v}, \Gamma \Rightarrow \Delta, w: \exists x A$}
\DisplayProof
\end{flushleft}
\begin{flushright}
\AxiomC{$\switch(\R, \unda \in D_{v}, \Gamma \Rightarrow \Delta, w: A(\unda/x), w: \exists x A)$}
\dottedLine
\RightLabel{=}
\UnaryInfC{$\Sigma\{\Gamma \restriction w \far A(\unda/x), \exists x A, \Delta \restriction w\}$}
\RightLabel{$\existsr$}
\UnaryInfC{$\Sigma\{\Gamma \restriction w \far \exists x A, \Delta \restriction w\}$}
\dottedLine
\RightLabel{=}
\UnaryInfC{$\switch(\R, \unda \in D_{v}, \Gamma \Rightarrow \Delta, w: \exists x A)$}
\DisplayProof
\end{flushright}

\begin{flushleft}
\AxiomC{$\R, \unda \in D_{v}, \Gamma \Rightarrow \Delta, w: A(\unda/x), w: \exists x A$}
\RightLabel{$\existsnedr$}
\UnaryInfC{$\R, \Gamma \Rightarrow \Delta, w: \exists x A$}
\DisplayProof
\end{flushleft}
\begin{flushright}
\AxiomC{$\switch(\R, \unda \in D_{v}, \Gamma \Rightarrow \Delta, w: A(\unda/x), w: \exists x A)$}
\dottedLine
\RightLabel{=}
\UnaryInfC{$\Sigma\{\Gamma \restriction w \far A(\unda/x), \exists x A, \Delta \restriction w\}$}
\RightLabel{$\existsr$}
\UnaryInfC{$\Sigma\{\Gamma \restriction w \far \exists x A, \Delta \restriction w\}$}
\dottedLine
\RightLabel{=}
\UnaryInfC{$\switch(\R, \Gamma \Rightarrow \Delta, w: \exists x A)$}
\DisplayProof
\end{flushright}
\qed 
\end{proof}

\begin{theorem}\label{thm:Nintfond-to-Intfondl}
Every derivation in $\nintfond$ of a nested sequent $\Sigma$ is effectively translatable to a nested form derivation of the labelled sequent $\switchtwo(\Sigma)$ in $\intfondl$.
\end{theorem}

\begin{proof} We show how to translate the $\disrl$, $\lift$, $\impr$ and $\existsr$ rules as all other cases are simple or similar. Last, we argue that the derivation obtained in $\intfondl$ is in nested form.

$\disrl$. Our given $\disrl$ inference is shown below top. Let $\vec{\unda} := \unda_{1}, \ldots, \unda_{n}$ be all free parameters in $A, X, Y$ that do not occur in $B, X, Y$, and let $\vec{\undb} := \undb_{1}, \ldots, \undb_{k}$ be all free parameters in $B, X, Y$ that do not occur in $A, X, Y$. To obtain the desired conclusion, we first apply the hp-admissibility of $\wk$ (Lem.~\ref{lem:extended-lint-properties}) in order to ensure that the contexts of the premises match, followed by an application of the $\disrl$ rule. We assume that $w$ is the label assigned to the formulae of the components $X, A \far Y$ and $X, B \far Y$ in $\Sigma \lcut X, A \far Y \rcut$ and $\Sigma \lcut X, B \far Y \rcut$, respectively.
\begin{center}
\AxiomC{$\Sigma \lcut X, A \far Y \rcut$}
\AxiomC{$\Sigma \lcut X, B \far Y \rcut$}
\RightLabel{$\disrl$}
\BinaryInfC{$\Sigma \lcut X, A \lor B \far Y \rcut$}
\DisplayProof
\end{center}
\begin{center}
\resizebox{\columnwidth}{!}{
\AxiomC{$\switchtwo(\Sigma \lcut X, A \far Y \rcut)$}
\RightLabel{=}
\dottedLine
\UnaryInfC{$\R, \vec{\unda} \in D_{w}, w :A, \Gamma \Rightarrow \Delta$}
\RightLabel{$\wk$}
\dashedLine
\UnaryInfC{$\R, \vec{\unda} \in D_{w}, \vec{\undb} \in D_{w}, w :A, \Gamma \Rightarrow \Delta$}

\AxiomC{$\switchtwo(\Sigma \lcut X, B \far Y \rcut)$}
\RightLabel{=}
\dottedLine
\UnaryInfC{$\R, \vec{\undb} \in D_{w}, w :B, \Gamma \Rightarrow \Delta$}
\RightLabel{$\wk$}
\dashedLine
\UnaryInfC{$\R, \vec{\unda} \in D_{w}, \vec{\undb} \in D_{w}, w :B, \Gamma \Rightarrow \Delta$}
\RightLabel{$\disrl$}
\BinaryInfC{$\R, \vec{\unda} \in D_{w}, \vec{\undb} \in D_{w}, w :A \vee B, \Gamma \Rightarrow \Delta$}
\RightLabel{=}
\dottedLine
\UnaryInfC{$\switchtwo(\Sigma \lcut X, A \lor B \far Y \rcut)$}
\DisplayProof
}
\end{center}

$\lift$. Our given $\lift$ inference is shown below top-left. Let $\vec{\unda} := \unda_{1}, \ldots, \unda_{n}$ be all free parameters occurring in $A$ that do not occur in $X', Y'$, $w$ be the label assigned to the formulae of the component $X, A \far Y$, and $u$ be the label assigned to the formulae of the component $X', A \far Y'$. If we invoke the admissibility of $\nd$ (Thm.~\ref{thm:admiss-all-rules-fond}) $n$ times to delete all domain atoms $\vec{\unda} \in D_{u}$, then after an application of $\lift$, we obtain the desired conclusion. 

\begin{flushleft}
\AxiomC{$\Sigma\{X, A \far Y, [X', A \far Y']\}$}
\RightLabel{$\lift$}
\UnaryInfC{$\Sigma\{X, A \far Y, [X' \far Y']\}$}
\DisplayProof
\end{flushleft}
\begin{flushright}
\AxiomC{$\switchtwo(\Sigma\{X, A \far Y, [X', A \far Y']\})$}
\RightLabel{=}
\dottedLine
\UnaryInfC{$\R, \vec{\unda} \in D_{w}, \vec{\unda} \in D_{u},  w \leq u, w:A, u:A, \Gamma \Rightarrow \Delta$}
\RightLabel{$\nd \times n$}
\dashedLine
\UnaryInfC{$\R, \vec{\unda} \in D_{w},  w \leq u, w:A, u:A, \Gamma \Rightarrow \Delta$}
\RightLabel{$\lift$}
\UnaryInfC{$\R, \vec{\unda} \in D_{w}, w \leq u, w:A, \Gamma \Rightarrow \Delta$}
\RightLabel{=}
\dottedLine
\UnaryInfC{$\switchtwo(\Sigma\{X, A \far Y, [X' \far Y']\})$}
\DisplayProof
\end{flushright}

$\imprr$. Let $\vec{\unda} := \unda_{1}, \ldots, \unda_{n}$ be the free parameters occurring in the active formulae $A$ and $B$ (shown in the $\imprr$ inference below). Let $w$ be the label assigned to the formulae of the component $X \far Y$ and $v$ be the label assigned to the formulae of the component $A \far B$ after applying the translation function $\switchtwo$. In order to obtain the desired conclusion, we first apply the hp-admissibility of $\wk$ (Lem.~\ref{lem:extended-lint-properties}), which introduces the domain atoms $\vec{\unda}' \in D_{w} := \unda_{i_{1}} \in D_{w}, \ldots, \unda_{i_{k}} \in D_{w}$ such that $\unda_{i_{j}} \in D_{w} \not\in \R$ for $i_{j} \in \{1, \ldots, n\}$ and $j \in \{1, \ldots, k\}$, i.e. for each parameter $\unda_{i_{j}}$ in $\vec{\unda}$ we introduce the domain atom $\unda_{i_{j}} \in D_{w}$ given that it does not already occur in $\R$. Next, we apply admissibility of $\nd$ (Thm.~\ref{thm:admiss-all-rules-fond}) $n$ times to delete all domain atoms $\vec{\unda} \in D_{v}$, making $v$ an eigenvarible, and permitting $\impr$ to be applied.

\begin{flushleft}
\AxiomC{$\Sigma \lcut X \far Y, [A \far B] \rcut$}
\RightLabel{$\imprr$}
\UnaryInfC{$\Sigma \lcut X \far A \imp B, Y \rcut$}
\DisplayProof
\end{flushleft}
\begin{flushright}
\AxiomC{$\switchtwo(\Sigma \lcut X \far Y, [A \far B] \rcut)$}
\RightLabel{=}
\dottedLine
\UnaryInfC{$\R, \vec{\unda} \in D_{v}, w \leq v, v :A, \Gamma \Rightarrow \Delta, v :B$}
\dashedLine
\RightLabel{$\wk$}
\UnaryInfC{$\R, \vec{\unda}' \in D_{w}, \vec{\unda} \in D_{v}, w \leq v, v :A, \Gamma \Rightarrow \Delta, v :B$}
\dashedLine
\RightLabel{$\nd \times n$}
\UnaryInfC{$\R, \vec{\unda}' \in D_{w}, w \leq v, v :A, \Gamma \Rightarrow \Delta, v :B$}
\RightLabel{$\imprr$}
\UnaryInfC{$\R, \vec{\unda}' \in D_{w}, \Gamma \Rightarrow \Delta, w :A \imp B$}
\RightLabel{=}
\dottedLine
\UnaryInfC{$\switchtwo(\Sigma \lcut X \far A \imp B, Y \rcut)$}
\DisplayProof
\end{flushright}

$\existsr$. If the relevant parameter $\unda$ (in the active formula $A(\unda/x)$) is an \emph{eigenvariable} in the $\existsr$ inference below, then the inference translates to an instance of $\existsnedr$. If the relevant parameter is \emph{available} in the inference, then the translation requires more care, so we focus on showing how to translate this case below.

\begin{center}
\AxiomC{$\Sigma\{X \far A(\unda/x), \exists x A, Y\}$}
\RightLabel{$\existsr$}
\UnaryInfC{$\Sigma\{X \far \exists x A, Y\}$}
\DisplayProof
\end{center}

Since the parameter $\unda$ is available in $\existsr$, this implies that there exists some formula $B(\unda)$ (different than the occurrence of $A(\unda/x)$) occurring in, or in the past of, the displayed component $X \far A(\unda/x), \exists x A, Y$ of the nested sequent $\Sigma\{X \far A(\unda/x), \exists x A, Y\}$ (i.e. the premise of the $\existsr$ inference above). We assume that the translation $\switchtwo$ assigns the label $u$ to all formulae in the component where $B(\unda)$ occurs, and note that $u : B(\unda) \in \Gamma,\Delta$ below. By the definition of $\switchtwo$ (Def.~\ref{def:switchtwo}), this implies the existence of a domain atom $\unda \in D_{u}$ in the image of $\Sigma\{X \far A(\unda/x), \exists x A, Y\}$ under $\switchtwo$ (which we have explicitly included in the derivation below). Because $B(\unda)$ is available and is associated with the label $u$, this implies that there is a directed path $u \leq v_{1}$, $\ldots$, $v_{n} \leq w$ from $u$ to $w$ in $\R$, i.e. $u \leadsto^{\R} w$. (NB. If $u = w$, then the translation is straightforward, so we omit consideration of that case.) If we weaken in domain atoms of the form $\unda \in D_{v_{i}}$ (with $i \in \{1, \ldots, n\}$) along this directed path, invoke the admissibility of $\nd$ (Thm.~\ref{thm:admiss-all-rules-fond}), and then apply $\existsndr$, we obtain the image of the conclusion of the $\existsr$ inference above under the $\switchtwo$ translation. This entire process is succinctly demonstrated in the derivation below:

\begin{center}
\resizebox{\columnwidth}{!}{
\begin{tabular}{c}
\AxiomC{$\switchtwo(\Sigma\{X \far A(\unda/x), \exists x A, Y\})$}
\dottedLine
\RightLabel{=}
\UnaryInfC{$\R, \unda \in D_{u}, \unda \in D_{w}, \Gamma \sar w : A(\unda/x), w : \exists x A, \Delta$}
\dashedLine
\RightLabel{$\wk$}
\UnaryInfC{$\R, \unda \in D_{u}, \unda \in D_{v_{1}}, \ldots, \unda \in D_{v_{n}}, \unda \in D_{w}, \Gamma \sar w : A(\unda/x), w : \exists x A, \Delta$}
\dashedLine
\RightLabel{$\nd \times (n+1)$}
\UnaryInfC{$\R, \unda \in D_{u}, \Gamma \sar w : A(\unda/x), w : \exists x A, \Delta$}
\RightLabel{$\existsndr$}
\UnaryInfC{$\R, \unda \in D_{u}, \Gamma \sar w : \exists x A, \Delta$}
\dottedLine
\RightLabel{=}
\UnaryInfC{$\switchtwo(\Sigma\{X \far \exists x A, Y\})$}
\DisplayProof
\end{tabular}
}
\end{center}
To see why the output derivation is in nested form, first observe that the conclusion of the derivation will be $\switchtwo(\Sigma)$, which, 
\emph{ipso facto}, makes the conclusion a nestedlike labelled sequent; hence, the derivation satisfies property (i) of Def.~\ref{def:nested-form}. Second, observe that the output derivation in $\intfondl$ satisfies property (ii) of Def.~\ref{def:nested-form} since each $\existsndr$ and $\allndr$ inference is obtained from an $\existsr$ or $\alll$ inference in the input derivation, ensuring the existence of a labelled formula $u : B(\unda)$ in the conclusion of any $\existsndr$ or $\allndr$ inference such that there is a directed path from $u$ to the label $w$ prefixing the principal formula.
\qed 
\end{proof}

\section{Deriving the Calculus $\nintfocd$ from $\lintfocd$}\label{section-6}

In the previous section, we saw that strengthening the $\existsr$ and $\alll$ rules allowed for the rules $\nd$ and $\ned$ to be eliminated from any given derivation. In this section, we will focus on how to strengthen these rules even further to additionally obtain the elimination of $\cd$ from any given derivation. We begin our analysis after first stating the following admissibility result: 

\begin{lemma}
\label{lem:trans-refl-focd-elim}
The $\refl$ and $\trans$ rules are admissible in the calculus $\lintfocd^{*}$.
\end{lemma}

\begin{proof} The proof is similar to Lem.~\ref{lem:trans-refl-fond-elim} and is detailed in Appendix~\ref{app:A} for the interested reader.
\qed
\end{proof}

The exclusion of $\refl$ and $\trans$ implies the completeness of $\lintfocd^{*} - \{\refl, \trans\}$ relative to treelike derivations, similar to Lem.~\ref{thm:treelike-derivations-fond}. We therefore assume for the rest of the section that all labelled sequents we consider are treelike---this assumption will ease our admissibility analyses of $\nd$, $\cd$, and $\ned$, though note that the assumption can be omitted, if desired.

Previously, we saw that permutations of the $\nd$ rule above the $\idfo$, $\existsr$, and $\alll$ rules shifted domain atoms to the past. We now investigate how permutations of the $\cd$ rule affect domain atoms in $\idfo$, $\existsr$, and $\alll$. As before, we first concentrate on permuting $\cd$ above the $\existsr$ rule solely. Observing the effect of permuting $\cd$ above $\existsr$ will motivate how we strengthen the rule---and the $\idfo$ and $\alll$ rules---to obtain the complete elimination of $\cd$ from any given derivation.

To begin our analysis, observe (3) a non-trivial application of $\cd$ after an $\existsr$ inference, and (4) the result of applying $\cd$ to the top sequent in (3). 

\begin{equation}
\AxiomC{$\R, w \leq v, \unda \in D_{w}, \unda \in D_{v}, \Gamma \Rightarrow \Delta, w : A(\unda/x), w : \exists x A$}
\RightLabel{$\existsr$}
\UnaryInfC{$\R, w \leq v, \unda \in D_{w}, \unda \in D_{v}, \Gamma \Rightarrow \Delta, w : \exists x A$}
\RightLabel{$\cd$}
\UnaryInfC{$\R, w \leq v, \unda \in D_{v}, \Gamma \Rightarrow \Delta, w : \exists x A$}
\DisplayProof
\end{equation}

\begin{equation}
\AxiomC{$\R, w \leq v, \unda \in D_{w}, \unda \in D_{v}, \Gamma \Rightarrow \Delta, w : A(\unda/x), w : \exists x A$}
\RightLabel{$\cd$}
\UnaryInfC{$\R, w \leq v, \unda \in D_{v}, \Gamma \Rightarrow \Delta, w : A(\unda/x), w : \exists x A$}
\DisplayProof
\end{equation}

The application of $\cd$ has generated a labelled sequent where the domain atom (i.e. $\unda \in D_{v}$) is associated with the label $v$ \emph{one step in the future} of $w$ (due to the presence of the $w \leq v$ relational atom). By analogy to what we saw in the last section (in the case of permuting $\nd$ above $\existsr$), it is not hard to see that successive (non-trivial) permutations of $\cd$ above an $\existsr$ inference will continually shift the domain atom to the future. If---at the same time---we also consider (non-trivial) permutations of the $\nd$ rule above the $\existsr$ rule, then 
the domain atom can shift either to the future or the past.


Recall that we have restricted our analysis to only consider derivations in $\lintfocd^{*} - \{\refl, \trans\}$ that employ treelike labelled sequents. Observe that \emph{any node} in the graph of a treelike labelled sequent is reachable from any other node by a sequence of forward and backwards edges. This property of treelike labelled sequents, accompanied by the witnessed shifting behavior discussed in the previous paragraph, suggests that we ought to add the following $\existscdr$ rule (see Fig.~\ref{fig:new-rules}) to our calculus:
\begin{center}
\AxiomC{$\R, \unda \in D_{v}, \Gamma \Rightarrow \Delta, w: A(\unda/x), w: \exists x A$}
\RightLabel{$\existscdr$}
\UnaryInfC{$\R, \unda \in D_{v}, \Gamma \Rightarrow \Delta, w: \exists x A$}
\DisplayProof
\end{center} 
with the side condition that $v \sim^{\R} w$, i.e. there is an undirected path from the label $v$ to the label $w$. Observe that this side condition generalizes the shifting behaviour of the domain atom to include all possible forward and backward \emph{shifts}. Intuitively, this side condition states that \emph{if an object is an element of some domain $D_{v}$, then it is an element of the domain $D_{w}$}, which always holds in the \emph{constant domain} setting.


The same shifting behavior can be witnessed by considering the permutation of $\nd$ and $\cd$ above the $\idfo$ and $\alll$ rules. Thus, we strengthen those rules in a similar way to obtain the $\idcd$ and $\allcdl$ rules (see Fig.~\ref{fig:new-rules}).


\begin{lemma}\label{lem:old-instance-of-new-cd}
The $\idfo$, $\existsr$, and $\alll$ rules are instances of the $\idcd$, $\existscdr$, and $\allcdl$ rules, respectively.
\end{lemma}

The lemma below confirms that the strengthened versions of $\idfo$, $\existsr$, and $\alll$ are sufficient to prove the admissibility of the $\nd$ and $\cd$ rules in the calculus $\lintfocd^{*} - \{\refl, \trans\}$:

\begin{lemma}
\label{lm:nd-cd-admiss}
The rules $\nd$ and $\cd$ are admissible in the calculus $\lintfocd^{*} - \{\refl,\trans\}$.
\end{lemma}

\begin{proof} The proof is similar to the proof of Lem.~\ref{lm:nd-admiss}, but where undirected paths (Def.~\ref{def:undirected-path}) are used instead of directed paths (Def.~\ref{def:directed-path}).
\qed
\end{proof}

In the previous section, we saw that permutations of $\ned$ above the $\existsndr$ and $\allndl$ rules could not be performed in certain cases. To resolve this issue, we introduced versions of the rules that absorbed the $\ned$ rule. The same phenomenon occurs in the case of the $\existscdr$ and $\allcdl$ rules, implying that we need to add $\existscedr$ and $\allcedl$ to our calculus (see Fig.~\ref{fig:new-rules}).

\begin{lemma}
\label{lm:ned-admiss-intfocd}
The rule $\ned$ is admissible in $\lintfocd^{*} - \{\refl,\trans,\nd,\cd\}$.
\end{lemma}

\begin{proof} Similar to the proof of Lem.~\ref{lm:ned-admiss}.
\qed
\end{proof}


\begin{theorem}
\label{thm:admiss-all-rules-focd}
The rules $\idfo$, $\botl$, $\imprl$, $\refl$, $\trans$, $\existsr, \alll$, $\allr$, $\nd$, $\cd$, and $\ned$ are admissible in the calculus $\intfocdl$.
\end{theorem}

\begin{proof} The admissibility of $\{\refl,\trans\}$, $\{\idfo, \existsr, \alll\}$, $\{\nd,\cd\}$, and $\ned$ follow from Lem.~\ref{lem:trans-refl-focd-elim}, Lem.~\ref{lem:old-instance-of-new-cd}, Lem.~\ref{lm:nd-cd-admiss}, and Lem.~\ref{lm:ned-admiss-intfocd}, respectively. Admissibility of $\impl$ and $\botl$ is shown as stated in Thm.~\ref{thm:admiss-all-rules-fond}. Admissibility of $\allr$ in the presence of $\allcdr$ is shown below:
\begin{center}
\AxiomC{$\R, w \leq v, \unda \in D_{v}, \Gamma \Rightarrow v : A(\unda/x), \Delta$}
\RightLabel{$(lsub)$}
\dashedLine
\UnaryInfC{$\R, w \leq w, \unda \in D_{w}, \Gamma \Rightarrow w : A(\unda/x), \Delta$}
\RightLabel{$\refl$}
\dashedLine
\UnaryInfC{$\R, \unda \in D_{w}, \Gamma \Rightarrow w : A(\unda/x), \Delta$}
\dashedLine
\RightLabel{$\allcdr$}
\UnaryInfC{$\R, \Gamma \Rightarrow w : \forall x A, \Delta$}
\DisplayProof
\end{center}
\qed
\end{proof}

\begin{theorem}
\label{thm:treelike-derivations-focd}
(i) The calculus $\intfocdl$ is complete relative to treelike derivations with the fixed root property. (ii) Every derivation in $\intfocdl$ of a treelike labelled sequent is a treelike derivation with the fixed root property.
\end{theorem}

\begin{proof}
Similar to Thm.~\ref{thm:treelike-derivations}.
\qed
\end{proof}

Although we relied on the notion of a nested form derivation (Def.~\ref{def:nested-form}) in the previous section to translate proofs from the refined labelled calculus $\intfondl$ into proofs in the nested calculus $\nintfond$, we note that this notion, or an analog thereof, is unnecessary to translate derivations from $\intfocdl$ to $\nintfocd$. This follows from the fact that side conditions are not imposed on $\existsr$ and $\alll$ in $\nintfocd$---contrary to what is the case in $\nintfond$. Hence, applications of $\existscdr$ and $\allcdl$ in $\intfocdl$ need not satisfy certain condition in order to be translated, thus simplifying the translation from labelled to nested in the constant domain setting.

\begin{theorem}\label{thm:Intfocdl-to-Nintfocd}
Every derivation in $\intfocdl$ of a nestedlike labelled sequent $\Lambda$ is effectively translatable to a derivation of the nested sequent $\switch(\Lambda)$ in $\nintfocd$.
\end{theorem}

\begin{proof} By Lem.~\ref{lem:nestedlike-implies-treelike}, we know that $\Lambda$ is treelike, implying that the derivation of $\Lambda$ is treelike by Thm.~\ref{thm:treelike-derivations-focd}, and so, the translation function $\switch$ is defined. When translating from labelled to nested, the existential rules $\existscdr$ and $\existscedr$, and the universal rules $\allcdl$ and $\allcedl$, translate to instances of $\existsr$ and $\alll$, respectively. All remaining rules in $\intfocdl$ translate to their corresponding version in the nested calculus.
\qed
\end{proof}

\begin{theorem}\label{thm:Nintfocd-to-Intfocdl}
Every derivation in $\nintfocd$ of a nested sequent $\Sigma$ is effectively translatable to a derivation of the labelled sequent $\switchtwo(\Sigma)$ in $\intfocdl$.
\end{theorem}

\begin{proof} We show how to translate an instance of $\existsr$; the remaining cases are similar to those in the proof of Thm.~\ref{thm:Intfondl-to-Nintfond}. Suppose we aim to translate an $\existsr$ inference occurring in our input derivation:

\begin{center}
\AxiomC{$\Sigma\{X \far A(\unda/x), \exists x A, Y\}$}
\RightLabel{$\existsr$}
\UnaryInfC{$\Sigma\{X \far \exists x A, Y\}$}
\DisplayProof
\end{center}

If the parameter $\unda$ in the above inference is an eigenvariable, then we translate the inference to an instance of $\existsnedr$ in $\intfocdl$. However, if the parameter occurs in a side formula, then the translation is non-trivial. Suppose then, that there exists some formula $B(\unda)$ (different than the occurrence of $A(\unda/x)$) occurring in the premise of the $\existsr$ inference above. Suppose that the translation $\switchtwo$ assigns the label $u$ to all formulae in the component where $B(\unda)$ occurs, and note that $u : B(\unda) \in \Gamma,\Delta$ below. This, in conjunction with the definition of $\switchtwo$ (Def.~\ref{def:switchtwo}), implies the existence of a domain atom $\unda \in D_{u}$ in the image of $\Sigma\{X \far A(\unda/x), \exists x A, Y\}$ under $\switchtwo$ (which is explicitly included in the derivation below). Since $\switchtwo(\Sigma\{X \far A(\unda/x), \exists x A, Y\})$ is a treelike labelled sequent, there must exist an undirected path $u \sim v_{1}$, $\ldots$, $v_{n} \sim w$ from $u$ to $w$ in $\R$, i.e. $u \sim^{\R} w$. (NB. If $u = w$, then the translation is straightforward, so we omit consideration of that case.) If we weaken in domain atoms of the form $\unda \in D_{v_{i}}$ (with $i \in \{1, \ldots, n\}$) along this undirected path, invoke the admissibility of $\nd$ and $\cd$ (Thm.~\ref{thm:admiss-all-rules-focd}), and then apply $\existscdr$, we obtain the image of the conclusion of the $\existsr$ inference above under the $\switchtwo$ translation. This entire process is succinctly demonstrated in the derivation below:
\begin{center}
\resizebox{\columnwidth}{!}{
\begin{tabular}{c}
\AxiomC{$\switchtwo(\Sigma\{X \far A(\unda/x), \exists x A, Y\})$}
\dottedLine
\RightLabel{=}
\UnaryInfC{$\R, \unda \in D_{u}, \unda \in D_{w}, \Gamma \sar w : A(\unda/x), w : \exists x A, \Delta$}
\dashedLine
\RightLabel{$\wk$}
\UnaryInfC{$\R, \unda \in D_{u}, \unda \in D_{v_{1}}, \ldots, \unda \in D_{v_{n}}, \unda \in D_{w}, \Gamma \sar w : A(\unda/x), w : \exists x A, \Delta$}
\dashedLine
\RightLabel{$\nd \times k_{1} + \cd \times k_{2}$}
\UnaryInfC{$\R, \unda \in D_{u}, \Gamma \sar w : A(\unda/x), w : \exists x A, \Delta$}
\RightLabel{$\existscdr$}
\UnaryInfC{$\R, \unda \in D_{u}, \Gamma \sar w : \exists x A, \Delta$}
\dottedLine
\RightLabel{=}
\UnaryInfC{$\switchtwo(\Sigma\{X \far \exists x A, Y\})$}
\DisplayProof
\end{tabular}
}
\end{center}
Note that $k_{1} + k_{2} = n+1$, i.e. $k_{1} + k_{2}$ is the number of labels occurring in the undirected path from $u$ to $w$ occurring after $u$ and including $w$.
\qed
\end{proof}

\section{Reverse Translations, Corollaries, and Inheritance}\label{section-7}

The previous three sections were dedicated to refining the labelled calculi $\lint$, $\lintfond$, and $\lintfocd$ to obtain refined labelled variants that were translatable into nested calculi. In this section, we analyze and provide the reverse translations from $\nint$, $\nintfond$, and $\nintfocd$ to $\lint$, $\lintfond$, and $\lintfocd$, respectively. These reverse translations are interesting in that they demonstrate how to add further syntactic structures and semantic information in order to obtain the more complicated semantic systems from the more refined nested systems. Additionally, these reverse translations can be leveraged to show that each nested calculus inherits proof theoretic properties from each labelled calculus, which we discuss in the second subsection.

\subsection{Reverse Translations}\label{subsec:reverse-trans}

When translating from labelled to nested, we discovered what new rules and what strengthened versions of rules were needed in order to allow the elimination of certain structural rules (e.g. $\refl$ and $\nd$) from a derivation. Obtaining the reverse translation is much simpler, as we already know what rules need to be shown admissible. Lem.~\ref{lem:admiss-rules-lint}--\ref{lem:admiss-rules-lintfocd} below show how to effectively transform derivations from $\intl$, $\intfondl$, and $\intfocdl$ into $\lint$, $\lintfond$, and $\lintfocd$, respectively. These transformations, coupled with the translations from Thm.~\ref{thm:refinement-prop}, \ref{thm:Nintfond-to-Intfondl}, and \ref{thm:Nintfocd-to-Intfocdl}, yield effective translations from our nested calculi to the $\mathsf{G3}$-style labelled calculi.

\begin{lemma}[\cite{Pim18}]\label{lem:admiss-rules-lint}
The rules $\idnew$, $\implnew$, and $\lift$ are admissible in $\lint$.
\end{lemma}

\begin{proof} The proof of the admissibility of each rule in $\lint$ is shown below:

\begin{center}
\AxiomC{}
\RightLabel{$\id$}
\UnaryInfC{$\R, w \leq w, \Gamma, w : p \sar w : p, \Delta$}
\RightLabel{$\refl$}
\UnaryInfC{$\R, \Gamma, w : p \sar w : p, \Delta$}
\DisplayProof
\end{center}

\begin{center}
\resizebox{\columnwidth}{!}{
\AxiomC{$\R, w :A \imp B, \Gamma \Rightarrow \Delta, w :A$}
\RightLabel{$\wk$}
\dashedLine
\UnaryInfC{$\R,w \leq w, w :A \imp B, \Gamma \Rightarrow \Delta, w :A$}

\AxiomC{$\R, w :A \imp B, w :B, \Gamma \Rightarrow \Delta$}
\RightLabel{$\wk$}
\dashedLine
\UnaryInfC{$\R,w \leq w, w :A \imp B, w :B, \Gamma \Rightarrow \Delta$}
\RightLabel{$\imprl$}
\BinaryInfC{$\R,w \leq w, w :A \imp B, \Gamma \Rightarrow \Delta$}
\RightLabel{$\refl$}
\UnaryInfC{$\R, w :A \imp B, \Gamma \Rightarrow \Delta$}
\DisplayProof 
}
\end{center}

\begin{center}
\resizebox{\columnwidth}{!}{
\AxiomC{}
\RightLabel{Thm.~\ref{thm:lint-properties}-(i)-(a)}
\dashedLine
\UnaryInfC{$\R, w \leq v, \Gamma, w : A \sar v : A, \Delta$}

\AxiomC{$\R, w \leq v, \Gamma, w : A, v : A \sar \Delta$}
\RightLabel{$\cut$}
\dashedLine
\BinaryInfC{$\R, w \leq v, \Gamma, w : A \sar \Delta$}
\DisplayProof
}
\end{center}
\qed
\end{proof}

\begin{lemma}\label{lem:admiss-rules-lintfond}
The rules $\idnd$, $\negl$, $\negr$, $\implnew$, $\lift$, $\existsndr$, $\allndl$, $\existsnedr$, and $\allnedl$ are admissible in $\lintfond$.
\end{lemma}

\begin{proof} The $\idnd$ rule can be derived by using $\idfo$, $\refl$, and $\nd$. The admissibility of $\implnew$ and $\lift$ are shown similarly as in Lem.~\ref{lem:admiss-rules-lint}. Recall that $\neg A := A \imp \bot$; using this fact, the proofs of the admissibility of $\negl$ and $\negr$ are as shown below:
\begin{center}
\resizebox{\columnwidth}{!}{
\AxiomC{$\R, w :A \imp \bot, \Gamma \Rightarrow \Delta, w :A$}
\RightLabel{$\wk$}
\dashedLine
\UnaryInfC{$\R,w \leq w, w :A \imp \bot, \Gamma \Rightarrow \Delta, w :A$}

\AxiomC{}
\RightLabel{$\botl$}
\UnaryInfC{$\R, w :A \imp \bot, w :\bot, \Gamma \Rightarrow \Delta$}
\RightLabel{$\wk$}
\dashedLine
\UnaryInfC{$\R,w \leq w, w :A \imp \bot, w :\bot, \Gamma \Rightarrow \Delta$}
\RightLabel{$\imprl$}
\BinaryInfC{$\R,w \leq w, w :A \imp \bot, \Gamma \Rightarrow \Delta$}
\RightLabel{$\refl$}
\UnaryInfC{$\R, w :A \imp \bot, \Gamma \Rightarrow \Delta$}
\DisplayProof 
}
\end{center}

\begin{center}
\AxiomC{$\R, w \leq v, \Gamma, v : A \sar \Delta$}
\RightLabel{$\wk$}
\dashedLine
\UnaryInfC{$\R, w \leq v, \Gamma, v : A \sar v : \bot, \Delta$}
\RightLabel{$\imprr$}
\UnaryInfC{$\R, \Gamma \sar w : A \imp \bot, \Delta$}
\DisplayProof
\end{center}

The admissibility of $\existsndr$ and $\allndl$ are shown below. We assume that there exists a directed path from $v$ to $w$ of a length greater than $0$ in $\R$, i.e. $v \leq u_{1}, \ldots, u_{n} \leq w \in \R$. The cases where the directed path from $v$ to $w$ is of length $0$ (i.e. $v = w$) are simple and are handled similarly.

\begin{center}
\resizebox{\columnwidth}{!}{
\AxiomC{$\R, \unda \in D_{v}, \Gamma \Rightarrow \Delta, w: A(\unda/x), w: \exists x A$}
\RightLabel{$\wk$}
\dashedLine
\UnaryInfC{$\R, \unda \in D_{v}, \unda \in D_{u_{1}}, \ldots, \unda \in D_{u_{n}}, \unda \in D_{w},  \Gamma \Rightarrow \Delta, w: A(\unda/x), w: \exists x A$}
\RightLabel{$\existsr$}
\UnaryInfC{$\R, \unda \in D_{v}, \unda \in D_{u_{1}}, \ldots, \unda \in D_{u_{n}}, \unda \in D_{w},  \Gamma \Rightarrow \Delta, w: \exists x A$}
\RightLabel{$\nd \times (n+1)$}
\UnaryInfC{$\R, \unda \in D_{v}, \Gamma \Rightarrow \Delta, w: \exists x A$}
\DisplayProof
}
\end{center}

\begin{center}
\resizebox{\columnwidth}{!}{
\AxiomC{$\R, \unda \in D_{v}, \Gamma, w: A(\unda/x), w : \forall x A \Rightarrow \Delta$}
\RightLabel{$\wk$}
\dashedLine
\UnaryInfC{$\R, w \leq w, \unda \in D_{v}, \unda \in D_{u_{1}}, \ldots, \unda \in D_{u_{n}}, \unda \in D_{w},  \Gamma, w: A(\unda/x), w : \forall x A \Rightarrow \Delta$}
\RightLabel{$\alll$}
\UnaryInfC{$\R, w \leq w, \unda \in D_{v}, \unda \in D_{u_{1}}, \ldots, \unda \in D_{u_{n}}, \unda \in D_{w},  \Gamma, w : \forall x A \Rightarrow \Delta$}
\RightLabel{$\refl$}
\UnaryInfC{$\R, \unda \in D_{v}, \unda \in D_{u_{1}}, \ldots, \unda \in D_{u_{n}}, \unda \in D_{w},  \Gamma, w : \forall x A \Rightarrow \Delta$}
\RightLabel{$\nd \times (n+1)$}
\UnaryInfC{$\R, \unda \in D_{v}, \Gamma, w : \forall x A \Rightarrow \Delta$}
\DisplayProof
}
\end{center}

Similar to the previous two proofs, to show the admissibility of $\existsnedr$ and $\allnedl$, we assume that $v \leq u_{1}, \ldots, u_{n} \leq w \in \R$, i.e. there exists a directed path of length greater than $0$ in $\R$. The cases where the directed path from $v$ to $w$ is of length $0$ (i.e. $v = w$) are simple and are handled similarly.

\begin{center}
\resizebox{\columnwidth}{!}{
\AxiomC{$\R, \unda \in D_{v}, \Gamma \Rightarrow \Delta, w: A(\unda/x), w: \exists x A$}
\RightLabel{$\wk$}
\dashedLine
\UnaryInfC{$\R, \unda \in D_{v}, \unda \in D_{u_{1}}, \ldots, \unda \in D_{u_{n}}, \unda \in D_{w},  \Gamma \Rightarrow \Delta, w: A(\unda/x), w: \exists x A$}
\RightLabel{$\existsr$}
\UnaryInfC{$\R, \unda \in D_{v}, \unda \in D_{u_{1}}, \ldots, \unda \in D_{u_{n}}, \unda \in D_{w},  \Gamma \Rightarrow \Delta, w: \exists x A$}
\RightLabel{$\nd \times (n+1)$}
\UnaryInfC{$\R, \unda \in D_{v}, \Gamma \Rightarrow \Delta, w: \exists x A$}
\RightLabel{$\ned$}
\UnaryInfC{$\R, \Gamma \Rightarrow \Delta, w: \exists x A$}
\DisplayProof
}
\end{center}

\begin{center}
\resizebox{\columnwidth}{!}{
\AxiomC{$\R, \unda \in D_{v}, \Gamma, w: A(\unda/x), w : \forall x A \Rightarrow \Delta$}
\RightLabel{$\wk$}
\dashedLine
\UnaryInfC{$\R, w \leq w, \unda \in D_{v}, \unda \in D_{u_{1}}, \ldots, \unda \in D_{u_{n}}, \unda \in D_{w},  \Gamma, w: A(\unda/x), w : \forall x A \Rightarrow \Delta$}
\RightLabel{$\alll$}
\UnaryInfC{$\R, w \leq w, \unda \in D_{v}, \unda \in D_{u_{1}}, \ldots, \unda \in D_{u_{n}}, \unda \in D_{w},  \Gamma, w : \forall x A \Rightarrow \Delta$}
\RightLabel{$\refl$}
\UnaryInfC{$\R, \unda \in D_{v}, \unda \in D_{u_{1}}, \ldots, \unda \in D_{u_{n}}, \unda \in D_{w},  \Gamma, w : \forall x A \Rightarrow \Delta$}
\RightLabel{$\nd \times (n+1)$}
\UnaryInfC{$\R, \unda \in D_{v}, \Gamma, w : \forall x A \Rightarrow \Delta$}
\RightLabel{$\ned$}
\UnaryInfC{$\R, \Gamma, w : \forall x A \Rightarrow \Delta$}
\DisplayProof
}
\end{center}
\qed
\end{proof}

\begin{lemma}\label{lem:admiss-rules-lintfocd}
The rules $\idcd$, $\negl$, $\negr$, $\implnew$, $\lift$, $\allcdr$, $\existscdr$, $\allcdl$, $\existscedr$, and $\allcedl$ are admissible in $\lintfocd$.
\end{lemma}

\begin{proof} The admissibility of $\idcd$, $\negl$, $\negr$, $\implnew$, $\lift$, $\existscdr$, $\allcdl$, $\existscedr$, and $\allcedl$ are shown similarly to the admissibility of the corresponding rules in $\lintfond$ (Lem.~\ref{lem:admiss-rules-lintfond} above). The admissibility of $\allcdr$ is shown below; we assume that $v$ is a fresh label.
\begin{center}
\AxiomC{$\Pi_{1}$}

\AxiomC{$\Pi_{2}$}

\RightLabel{$\cut$}
\dashedLine
\BinaryInfC{$\R, w \leq v, \unda \in D_{w}, \unda \in D_{v}, \Gamma \Rightarrow v : A(\unda/x), \Delta$}
\RightLabel{$\cd$}
\UnaryInfC{$\R, w \leq v, \unda \in D_{v}, \Gamma \Rightarrow v : A(\unda/x), \Delta$}
\RightLabel{$\allr$}
\UnaryInfC{$\R, \Gamma \Rightarrow w : \forall x A, \Delta$}
\DisplayProof
\end{center}
\begin{center}
\begin{tabular}{c c c}
$\Pi_{1}$

&

$= \Big \{$

&

\AxiomC{$\R, \unda \in D_{w}, \Gamma \Rightarrow w : A(\unda/x), \Delta$}
\RightLabel{$\wk$}
\dashedLine
\UnaryInfC{$\R, w \leq v, \unda \in D_{w}, \unda \in D_{v}, \Gamma \Rightarrow w : A(\unda/x), v : A(\unda/x), \Delta$}
\DisplayProof
\end{tabular}
\end{center}
\begin{center}
\resizebox{\columnwidth}{!}{
\begin{tabular}{c c c}
$\Pi_{2}$

&

$= \Big \{$

&

\AxiomC{}
\RightLabel{Thm.~\ref{thm:lint-properties}-(i)-(c)}
\dashedLine
\UnaryInfC{$\R, w \leq v, \unda \in D_{w}, \unda \in D_{v}, \Gamma, w : A(\unda/x) \Rightarrow v : A(\unda/x), \Delta$}
\DisplayProof
\end{tabular}
}
\end{center}
\qed
\end{proof}

\begin{theorem}\label{thm:refined-to-labelled}
Every derivation of a labelled sequent $\Lambda$ in $\intl$, $\intfondl$, and $\intfocdl$ can be effectively transformed into a derivation of $\Lambda$ in $\lint$, $\lintfond$, and $\lintfocd$, respectively, and vice-versa.
\end{theorem}

\begin{proof} The forward directions follow from Thm.~\ref{thm:admiss-all-rules-propositional}, \ref{thm:admiss-all-rules-fond}, and \ref{thm:admiss-all-rules-focd}. The backward directions follow from Lem.~\ref{lem:admiss-rules-lint}, \ref{lem:admiss-rules-lintfond}, and \ref{lem:admiss-rules-lintfocd}, respectively.
\qed
\end{proof}

\begin{theorem}\label{thm:nested-to-labelled}
(i) If a nestedlike labelled sequent $\Lambda$ is derivable in $\lint$, $\lintfond$, and $\lintfocd$, then $\switch(\Lambda)$ is derivable in $\nint$, $\nintfond$, and $\nintfocd$, respectively. (ii) Every derivation in $\nint$, $\nintfond$, and $\nintfocd$ of a nested sequent $\Sigma$ can be effectively translated into a derivation of the labelled sequent $\switchtwo(\Sigma)$ in $\lint$, $\lintfond$, and $\lintfocd$, respectively.
\end{theorem}

\begin{proof} Follows from Thm.~\ref{thm:refinement-prop}, \ref{thm:Intfondl-to-Nintfond}, \ref{thm:Nintfond-to-Intfondl}, \ref{thm:Intfocdl-to-Nintfocd}, \ref{thm:Nintfocd-to-Intfocdl}, and \ref{thm:refined-to-labelled}.  
\qed
\end{proof}

\subsection{Corollaries and Inheritance}\label{subsec:cor-and-inheritance}

An appealing consequence of our refinement and translation procedures is that our nested calculi inherit desirable proof-theoretic properties from their associated labelled calculi. Such properties are stated in Cor.~\ref{cor:nested-complete-sound}--\ref{cor:nested-internality} below. 

As mentioned near the end of Sect.~\ref{section-5}, the translation of proofs from $\intfondl$ into proofs in $\nintfond$ (Thm.~\ref{thm:Intfondl-to-Nintfond}) relies on the completeness of the latter (Thm.~\ref{thm:strong-soundness-completeness-nested}); hence, we are barred from leveraging this proof-theoretic translation to conclude that $\nintfond$ ultimately inherits completeness from $\lintfond$. (NB. In Sect.~\ref{section-5}, we also discussed methods of proof that would free us from relying on the completeness of $\nintfond$, thus allowing for this inheritance result to go through.) Nevertheless, this issue does not arise in our work concerning propositional intuitionistic logic $\int$ and first-order intuitionistic logic with constant domains $\intfocd$, and so, we may conclude the following corollary:

\begin{corollary}
\label{cor:nested-complete-sound}
The calculi $\nint$ and $\nintfocd$ have inherited soundness and completeness from $\lint$ and $\lintfocd$, respectively.
\end{corollary}

\begin{proof} Soundness follows from Thm.~\ref{thm:lint-properties} and \ref{thm:nested-to-labelled}. Regarding completeness, observe that if $A$ is a theorem of $\int$ or $\intfocd$, then $\vec{a} \in D_{w} \sar w : A$ is derivable in $\lint$ or $\lintfocd$ (respectively) by Thm.~\ref{thm:lint-properties}. By Thm.~\ref{thm:admiss-all-rules-propositional} and~\ref{thm:refinement-prop}, or Thm.~\ref{thm:admiss-all-rules-focd} and~\ref{thm:Intfocdl-to-Nintfocd}, the derivation of $\vec{a} \in D_{w} \sar w : A$ may be transformed and translated into a derivation of $A$ in $\nint$ or $\nintfocd$, respectively.
\qed
\end{proof}

\begin{figure}[t]
\noindent\hrule

\begin{center}
\begin{tabular}{c c c}
\AxiomC{$\Sigma\{X \far Y\}$}
\RightLabel{$(wk_{l})$}
\UnaryInfC{$\Sigma\{X,Z \far Y\}$}
\DisplayProof

&

\AxiomC{$\Sigma\{X \far Y\}$}
\RightLabel{$(wk_{r})$}
\UnaryInfC{$\Sigma\{X \far Y,Z\}$}
\DisplayProof

&

\AxiomC{$\Sigma$}
\RightLabel{$(psub)$}
\UnaryInfC{$\Sigma(\unda / \undb)$}
\DisplayProof
\end{tabular}
\end{center}

\begin{center}
\begin{tabular}{c c c} 
\AxiomC{$\Sigma\{X,Z,Z \far Y\}$}
\RightLabel{$(ctr_{l})$}
\UnaryInfC{$\Sigma\{X,Z \far Y\}$}
\DisplayProof

&

\AxiomC{$\Sigma\{X \far Z,Z,Y\}$}
\RightLabel{$(ctr_{r})$}
\UnaryInfC{$\Sigma\{X \far Z,Y\}$}
\DisplayProof

&

\AxiomC{$\Sigma$}
\RightLabel{$(nr)$}
\UnaryInfC{$\ \far [\Sigma]$}
\DisplayProof
\end{tabular}
\end{center}

\begin{center}
\begin{tabular}{c c}
\AxiomC{$\Sigma\{X \far Y, [\Sigma'], [\Sigma']\}$}
\RightLabel{$(mrg_{1})$}
\UnaryInfC{$\Sigma\{X \far Y, [\Sigma']\}$}
\DisplayProof

&

\AxiomC{$\Sigma\{X_{1} \far Y_{1}, [X_{2} \far Y_{2}, [\Sigma_{1}], \ldots, [\Sigma_{n}]] \}$}
\RightLabel{$(mrg_{2})$}
\UnaryInfC{$\Sigma\{X_{1}, X_{2} \far Y_{1}, Y_{2}, [\Sigma_{1}], \ldots, [\Sigma_{n}] \}$}
\DisplayProof
\end{tabular}
\end{center}

\begin{center}
\begin{tabular}{c c}
\AxiomC{$\Sigma\{X \far Y\}$}
\RightLabel{$(ew_{1})$}
\UnaryInfC{$\Sigma\{X \far Y, [\Sigma']\}$}
\DisplayProof

&

\AxiomC{$\Sigma\{X_{1} \far Y_{1}, [X_{2} \far Y_{2}, [\Sigma_{1}], \ldots, [\Sigma_{n}]] \}$}
\RightLabel{$(ew_{2})$}
\UnaryInfC{$\Sigma\{X_{1} \far Y_{1}, [\ \far \ [X_{2} \far Y_{2}, [\Sigma_{1}], \ldots, [\Sigma_{n}]]] \}$}
\DisplayProof
\end{tabular}
\end{center}

\begin{center}
\begin{tabular}{c @{\hskip 1em} c}
\AxiomC{$\Sigma\{X_{1} \far A, Y_{1}, [X_{2} \far A, Y_{2}]\}$}
\RightLabel{$(lwr)$}
\UnaryInfC{$\Sigma\{X_{1} \far Y_{1}, [X_{2} \far A, Y_{2}]\}$}
\DisplayProof

&

\AxiomC{$\Sigma \{X \far A, Y\}$}
\AxiomC{$\Sigma \{X, A \far Y\}$}
\RightLabel{$\cut$}
\BinaryInfC{$\Sigma \{X \far Y\}$}
\DisplayProof
\end{tabular}
\end{center}

\hrule
\caption{The set $\nstrucset$ of admissible rules in $\nint$, $\nintfond$, and $\nintfocd$.}
\label{fig:admiss-rules-nint}
\end{figure}

Furthermore, we obtain the admissibility of the set of structural rules presented in Fig.~\ref{fig:admiss-rules-nint}. This set of structural rules was primarily identified by investigating the variety of ways in which labelled structural rules (e.g. $\wk$, $\ctrr$, etc.) can be applied to manipulate treelike labelled sequents. Therefore, such rules offer a more fine-grained view of the types of operations inherent in the more general functionality of the aforementioned labelled structural rules. Although the admissibility of the rules in Fig.~\ref{fig:admiss-rules-nint} can be established by confirming the soundness of each rule, and recognizing that each nested calculus is complete (Thm.~\ref{thm:strong-soundness-completeness-nested}), we provide proof-theoretic arguments which demonstrate how such rules are obtained from their associated labelled structural rules. The proof of this admissibility result (Cor.~\ref{cor:nested-admissibility} below) can be found in Appendix~\ref{app:A} for the interested reader.


\begin{corollary}
\label{cor:nested-admissibility}
All rules in $\nstrucset$ are admissible in $\nint$, $\nintfond$ and $\nintfocd$.
\end{corollary}

Additionally, we obtain the invertibility of the rules in our nested calculi:

\begin{corollary}
\label{cor:nested-invertibility}
All rules in $\nint$, $\nintfond$ and $\nintfocd$ are invertible.
\end{corollary}

\begin{proof} Follows from Lem.~\ref{lem:extended-lint-properties} and Thm.~\ref{thm:nested-to-labelled}. 
\qed
\end{proof}

By making use of the previous two corollaries, we can prove that derivability of a nested sequent is equivalent to the derivability of its formula interpretation (Def.~\ref{def:formula-interpretation}). Note that in the propositional case the universal closure $\uc$ is vacuous, and so, $\far \uc \iota(\Sigma)$ is identical to $\far \iota(\Sigma)$.

\begin{corollary}
\label{cor:nested-internality}
The nested sequent $\Sigma$ is derivable in $\nint$, $\nintfond$ and $\nintfocd$ iff the nested sequent $\far \uc \iota(\Sigma)$ is derivable in $\nint$, $\nintfond$ and $\nintfocd$, respectively.
\end{corollary}

\begin{proof} We begin with the forward direction, and assume that $\Sigma$ is derivable in one of our nested systems. By the admissibility of $(nr)$ (Cor.~\ref{cor:nested-admissibility}), we know that $\far [\Sigma]$ is derivable. Starting at the leaves of the nested sequent and continuously applying $\landl$, $\lorr$, and $\impr$, we eventually obtain a derivation of $\far \iota(\Sigma)$. From here, we have three cases to consider: (i) If we are working in the nested system $\nint$, then we are done. (ii) If we are working in the nested system $\nintfond$, and $\far \iota(\Sigma)$ contains no free variables (i.e. parameters), then we are done. Otherwise, if the free variables (i.e. parameters) $\unda_{1}, \ldots, \unda_{n}$ occur in $\iota(\Sigma)$, then we apply the $(nr)$ rule $n$ times to obtain:
$$
\underbrace{\far \ [\ldots [\far \iota(\Sigma)]\ldots]}_{n}
$$
Applying the $\allndr$ rule $n$ times gives a derivation of $\far \uc \iota(\Sigma)$. (iii) If we are working in the nested system $\nintfocd$ and $\far \iota(\Sigma)$ does not contain any occurences of free variables (i.e. parameters), then we are done. Otherwise, we repeatedly apply $\allcdr$ to obtain a derivation of the desired conclusion.

For the backwards direction, we invoke the invertibility of $\landl$, $\lorr$, $\impr$, $\allndr$, and $\allcdr$ (Cor.~\ref{cor:nested-invertibility}), as well as the admissibility of $(mrg_{2})$ (Cor.~\ref{cor:nested-invertibility}), to obtain a derivation of $\Sigma$.
\qed
\end{proof}

Last, by making use of the admissibility of $(mrg_{2})$, we can show that each derivation in $\nintfond$ can be transformed into a derivation in $\nintfocd$. We also show that the same embedding holds between $\intfondl$ and $\intfocdl$ to complete our mapping of the relationships mentioned in Fig.~\ref{fig:relationships} of Sect.~\ref{Intro}.

\begin{corollary}\label{cor:embed-Q-to-QC}
(i) Every derivation in $\intfondl$ is transformable into a derivation in $\intfocdl$. (ii) Every derivation in $\nintfond$ is transformable into a derivation in $\nintfocd$.
\end{corollary}

\begin{proof} 
Observe that the $\existsndr$, $\existsnedr$, $\allndl$, and $\allnedl$ rules in $\intfondl$ are instances of the $\existscdr$, $\existscedr$, $\allcdl$, and $\allcedl$ rules in $\intfocdl$. Additionally, the same relationship holds true of the $\existsr$ and $\alll$ rules in $\nintfond$ and $\nintfocd$. Thus, the only rules that differ between the systems $\intfondl$ and $\intfocdl$ are the $\allr$ and $\allcdr$ rules, and the only rules that differ between the systems $\nintfond$ and $\nintfocd$ are the $\allndr$ and $\allcdr$ rules. Therefore, we show below that $\allr$ and $\allndr$ are admissible in $\intfocdl$ and $\nintfocd$, respectively, which proves the claim since all remaining cases are straightforward.
\begin{center}
\resizebox{\columnwidth}{!}{
\begin{tabular}{c c}
\AxiomC{$\R, w \leq v, \unda \in D_{v}, \Gamma \Rightarrow v : A(\unda/x), \Delta$}
\RightLabel{$(lsub)$}
\dashedLine
\UnaryInfC{$\R, w \leq w, \unda \in D_{w}, \Gamma \Rightarrow w : A(\unda/x), \Delta$}
\RightLabel{$\refl$}
\dashedLine
\UnaryInfC{$\R, \unda \in D_{w}, \Gamma \Rightarrow w : A(\unda/x), \Delta$}
\dashedLine
\RightLabel{$\allcdr$}
\UnaryInfC{$\R, \Gamma \Rightarrow w : \forall x A, \Delta$}
\DisplayProof

&

\AxiomC{$\Sigma\{X \far Y, [\ \far A(\unda/x)]\}$}
\RightLabel{$(mrg_{2})$}
\dashedLine
\UnaryInfC{$\Sigma\{X \far A(\unda/x), Y\}$}
\RightLabel{$\allcdr$}
\UnaryInfC{$\Sigma\{X \far \forall x A, Y\}$}
\DisplayProof
\end{tabular}
}
\end{center}
\qed
\end{proof}

\section{Conclusion}\label{conclusion}

In this paper, we mapped out the interrelationships between semantic systems and Fitting's nested calculi for intuitionistic logics. We focused extensively on \emph{how} to eliminate the structural rules of the labelled calculi $\lint$, $\lintfond$, and $\lintfocd$ from derivations, by considering how to strengthen their logical rules, or via the addition of new rules. We referred to this process as the method of \emph{refinement} since it ultimately yielded the more refined (sound and complete) calculi $\intl$, $\intfondl$, and $\intfocdl$ which omitted the use of semantically explicit structural rules (e.g. $\refl$ and $\nd$) and only required treelike labelled sequents in derivations. 
 We saw that fragments of these refined labelled calculi were equivalent to Fitting's nested calculi~\cite{Fit14}. Moreover, we also observed that Fitting's nested calculi inherited proof-theoretic properties from their associated labelled calculi; e.g. soundness and completeness, invertibility of rules, and admissibility of $\cut$. Such results suggest that structural rule elimination is a natural method for connecting the relational semantics of a logic to a nested calculus for the logic (via the labelled sequent paradigm).

The method of refining labelled calculi into nested calculi for intuitionistic logics is part of a larger project that seeks to investigate the extent to which such refinements can be performed. 
Due to general results concerning the construction of labelled calculi, their modularity, and the confirmation of their proof-theoretic properties~\cite{CiaMafSpe13,DycNeg12,Gab96,Sim94}, the exposition of a method for deriving refined (or, nested) calculi---which utilize sequents with a simpler underlying data structure (e.g. trees) and have proven to be well-suited for certain automated reasoning tasks (e.g. proof-search~\cite{TiuIanGor12}, counter-model extraction~\cite{LyoBer19,TiuIanGor12}, and effective interpolation~\cite{LyoTiuGorClo20})---from semantic systems could prove to be of practical consequence. Moreover, since labelled calculi are easily obtained from a logic's semantics, the method of refinement provides a means by which (nested) proof calculi with favorable properties may be obtained directly from a logic's semantics. 


\ \\
\noindent
\acknowledgments{Work funded by FWF projects I2982 and W1255-N23.}

\bibliographystyle{spmpsci}
\bibliography{bib}

\newcommand{\noop}[1]{}
\begin{thebibliography}{10}
\providecommand{\url}[1]{{#1}}
\providecommand{\urlprefix}{URL }
\expandafter\ifx\csname urlstyle\endcsname\relax
  \providecommand{\doi}[1]{DOI~\discretionary{}{}{}#1}\else
  \providecommand{\doi}{DOI~\discretionary{}{}{}\begingroup
  \urlstyle{rm}\Url}\fi

\bibitem{Avr96}
Avron, A.: The method of hypersequents in the proof theory of propositional
  non-classical logics.
\newblock In: W.~Hodges, M.~Hyland, C.~Steinhorn, J.~Truss (eds.) Logic: From
  Foundations to Applications: European Logic Colloquium, p. 1–32. Clarendon
  Press, USA (1996)

\bibitem{Bel82}
Belnap, N.D.: Display logic.
\newblock Journal of philosophical logic \textbf{11}(4), 375--417 (1982)

\bibitem{Bru09}
Br{\"{u}}nnler, K.: Deep sequent systems for modal logic.
\newblock Arch. Math. Log. \textbf{48}(6), 551--577 (2009).
\newblock \doi{10.1007/s00153-009-0137-3}.
\newblock \urlprefix\url{https://doi.org/10.1007/s00153-009-0137-3}

\bibitem{Bul92}
Bull, R.A.: Cut elimination for propositional dynamic logic without *.
\newblock Z. Math. Logik Grundlag. Math. \textbf{38}(2), 85--100 (1992)

\bibitem{CiaMafSpe13}
Ciabattoni, A., Maffezioli, P., Spendier, L.: Hypersequent and labelled calculi
  for intermediate logics.
\newblock In: D.~Galmiche, D.~Larchey-Wendling (eds.) Automated Reasoning with
  Analytic Tableaux and Related Methods, \emph{Lecture Notes in Computer
  Science}, vol. 8123, pp. 81--96. Springer Berlin Heidelberg, Berlin,
  Heidelberg (2013)

\bibitem{DycNeg12}
Dyckhoff, R., Negri, S.: Proof analysis in intermediate logics.
\newblock Archive for Mathematical Logic \textbf{51}(1-2), 71--92 (2012)

\bibitem{Fit72}
Fitting, M.: Tableau methods of proof for modal logics.
\newblock Notre Dame Journal of Formal Logic \textbf{13}(2), 237--247 (1972)

\bibitem{Fit12}
Fitting, M.: Prefixed tableaus and nested sequents.
\newblock Annals of Pure and Applied Logic \textbf{163}(3), 291 -- 313 (2012).
\newblock \doi{https://doi.org/10.1016/j.apal.2011.09.004}

\bibitem{Fit14}
Fitting, M.: Nested sequents for intuitionistic logics.
\newblock Notre Dame Journal of Formal Logic \textbf{55}(1), 41--61 (2014)

\bibitem{GabSheSkv09}
Gabbay, D., Shehtman, V., Skvortsov, D.: Quantification in Non-classical
  Logics.
\newblock Studies in Logic and Foundations of Mathematics. Elsevier (2009)

\bibitem{Gab96}
Gabbay, D.M.: Labelled deductive systems, \emph{Oxford Logic guides}, vol.~33.
\newblock Clarendon Press/Oxford Science Publications (1996)

\bibitem{Gen35}
Gentzen, G.: Untersuchungen uber das logische schliessen.
\newblock Mathematische Zeitschrift \textbf{39}(3), 405--431 (1935)

\bibitem{GorRam12AIML}
Gor{\'{e}}, R., Ramanayake, R.: Labelled tree sequents, tree hypersequents and
  nested (deep) sequents.
\newblock In: T.~Bolander, T.~Bra{\"{u}}ner, S.~Ghilardi, L.S. Moss (eds.)
  Advances in Modal Logic 9, papers from the ninth conference on "Advances in
  Modal Logic," held in Copenhagen, Denmark, 22-25 August 2012, pp. 279--299.
  College Publications (2012).
\newblock
  \urlprefix\url{http://www.aiml.net/volumes/volume9/Gore-Ramanayake.pdf}

\bibitem{Kas94}
Kashima, R.: Cut-free sequent calculi for some tense logics.
\newblock Studia Logica \textbf{53}(1), 119--135 (1994)

\bibitem{Kri65}
Kripke, S.A.: Semantical analysis of intuitionistic logic i.
\newblock In: J.~Crossley, M.~Dummett (eds.) Formal Systems and Recursive
  Functions, \emph{Studies in Logic and the Foundations of Mathematics},
  vol.~40, pp. 92 -- 130. Elsevier (1965).
\newblock \doi{https://doi.org/10.1016/S0049-237X(08)71685-9}.
\newblock
  \urlprefix\url{http://www.sciencedirect.com/science/article/pii/S0049237X08716859}

\bibitem{Lyo20}
Lyon, T.: On deriving nested calculi for intuitionistic logics from semantic
  systems.
\newblock In: S.N. Art{\"{e}}mov, A.~Nerode (eds.) Logical Foundations of
  Computer Science - International Symposium, {LFCS} 2020, Deerfield Beach, FL,
  USA, January 4-7, 2020, Proceedings, \emph{Lecture Notes in Computer
  Science}, vol. 11972, pp. 177--194. Springer (2020).
\newblock \doi{10.1007/978-3-030-36755-8\_12}.
\newblock \urlprefix\url{https://doi.org/10.1007/978-3-030-36755-8\_12}

\bibitem{LyoBer19}
Lyon, T., van Berkel, K.: Automating agential reasoning: Proof-calculi and
  syntactic decidability for stit logics.
\newblock In: M.~Baldoni, M.~Dastani, B.~Liao, Y.~Sakurai, R.~Zalila~Wenkstern
  (eds.) {PRIMA} 2019: Principles and Practice of Multi-Agent Systems - 22nd
  International Conference, Turin, Italy, October 28-31, 2019, Proceedings,
  \emph{Lecture Notes in Computer Science}, vol. 11873, pp. 202--218. Springer
  International Publishing, Cham (2019)

\bibitem{LyoIttEckGra17}
Lyon, T., Ittner, C., Eckhardt, T., Gratzl, N.: The basics of display calculi.
\newblock Kriterion -- Journal of Philosophy \textbf{31}(2), 55--100 (2017)

\bibitem{LyoTiuGorClo20}
Lyon, T., Tiu, A., Gor{\'{e}}, R., Clouston, R.: Syntactic interpolation for
  tense logics and bi-intuitionistic logic via nested sequents.
\newblock In: M.~Fern{\'{a}}ndez, A.~Muscholl (eds.) 28th {EACSL} Annual
  Conference on Computer Science Logic, {CSL} 2020, January 13-16, 2020,
  Barcelona, Spain, \emph{LIPIcs}, vol. 152, pp. 28:1--28:16. Schloss Dagstuhl
  - Leibniz-Zentrum f{\"{u}}r Informatik (2020).
\newblock \doi{10.4230/LIPIcs.CSL.2020.28}.
\newblock \urlprefix\url{https://doi.org/10.4230/LIPIcs.CSL.2020.28}

\bibitem{Mina13}
Minari, P.: Labeled sequent calculi for modal logics and implicit contractions.
\newblock Arch. Math. Log. \textbf{52}(7-8), 881--907 (2013).
\newblock \doi{10.1007/s00153-013-0350-y}

\bibitem{Pim18}
Pimentel, E.: A semantical view of proof systems.
\newblock In: L.S. Moss, R.J.G.B. de~Queiroz, M.~Martinez (eds.) Logic,
  Language, Information, and Computation - 25th International Workshop, WoLLIC
  2018, Bogota, Colombia, July 24-27, 2018, Proceedings, \emph{Lecture Notes in
  Computer Science}, vol. 10944, pp. 61--76. Springer (2018).
\newblock \doi{10.1007/978-3-662-57669-4\_3}.
\newblock \urlprefix\url{https://doi.org/10.1007/978-3-662-57669-4\_3}

\bibitem{Pog08}
Poggiolesi, F.: A cut-free simple sequent calculus for modal logic $s5$.
\newblock The Review of Symbolic Logic \textbf{1}(1), 3--15 (2008)

\bibitem{Pog09Trends}
Poggiolesi, F.: The method of tree-hypersequents for modal propositional logic.
\newblock In: D.~Makinson, J.~Malinowski, H.~Wansing (eds.) Towards
  Mathematical Philosophy, \emph{Trends in logic}, vol.~28, pp. 31--51.
  Springer (2009).
\newblock \doi{10.1007/978-1-4020-9084-4\_3}.
\newblock \urlprefix\url{https://doi.org/10.1007/978-1-4020-9084-4\_3}

\bibitem{Sim94}
Simpson, A.K.: The proof theory and semantics of intuitionistic modal logic.
\newblock Ph.D. thesis, University of Edinburgh. College of Science and
  Engineering. School of Informatics (1994)

\bibitem{TiuIanGor12}
Tiu, A., Ianovski, E., Gor{\'{e}}, R.: Grammar logics in nested sequent
  calculus: Proof theory and decision procedures.
\newblock In: T.~Bolander, T.~Bra{\"{u}}ner, S.~Ghilardi, L.S. Moss (eds.)
  Advances in Modal Logic 9, papers from the ninth conference on "Advances in
  Modal Logic," held in Copenhagen, Denmark, 22-25 August 2012, pp. 516--537.
  College Publications (2012).
\newblock
  \urlprefix\url{http://www.aiml.net/volumes/volume9/Tiu-Ianovski-Gore.pdf}

\bibitem{TroDal88}
Troelstra, A., van Dalen, D.: Constructivism in Mathematics, vol.~1.
\newblock Elsevier Science (1988)

\bibitem{Vig00}
Vigan{\`o}, L.: Labelled Non-Classical Logics.
\newblock Springer Science \& Business Media (2000)

\bibitem{Wan94}
Wansing, H.: Sequent calculi for normal modal propositional logics.
\newblock Journal of Logic and Computation \textbf{4}(2), 125--142 (1994)

\end{thebibliography}

\appendix

\section{Proofs}\label{app:A}

\begin{customthm}{\ref{thm:lint-properties}} Let $\mathsf{G3X} \in \{\lintfo,\lintfocd \}$. The calculi $\lint$, $\lintfo$, and $\lintfocd$ have the following properties:
\begin{itemize}

\item[$(i)$] 

\begin{itemize}

\item[(a)] For all $A \in \lang$, $ \vdash_{\lint} \R,w \leq v, w : A, \Gamma \Rightarrow v : A, \Delta$;

\item[(b)] For all $A \in \lang$, $ \vdash_{\lint} \R,w:A,\Gamma \Rightarrow \Delta, w :A$; 

\item[(c)] For all $A \in \langfo$, $\vdash_{\mathsf{G3X}} \R,w \leq v, \vv{\unda} \in D_{w}, w : A(\vv{\unda}), \Gamma \Rightarrow v : A(\vv{\unda}), \Delta$; 

\item[(d)] For all $A \in \langfo$, $\vdash_{\mathsf{G3X}} \R, \vv{\unda} \in D_{w}, w:A(\vv{\unda}),\Gamma \Rightarrow \Delta, w :A(\vv{\unda})$;

\end{itemize}

\item[$(ii)$] All rules in $\strucset - \{\cut\}$ are hp-admissible;

\item[$(iii)$] All rules are hp-invertible;

\item[$(iv)$] The $\cut$ rule is admissible;

\item[$(v)$] $\lint$, $\lintfo$, and $\lintfocd$ are sound and complete for $\int$, $\intfo$, and $\intfocd$, respectively.

\end{itemize}
\end{customthm}

\begin{proof} We argue that properties (i)--(v) obtain for $\lintfocd$; by omitting consideration of the $\cd$ rule, we obtain proofs of these properties for $\lintfond$ as well.


\textit{Claim (i).} Claims (a) and (b) are shown in~\cite[Lem.~1]{DycNeg12}, so we focus on proving (c) and (d). We prove claims (c) and (d) together by mutual induction on the complexity of $A$.\\

\textit{Claim (i)-(c).} The base case is resolved using the $\idfo$ rule for atomic formulae and the $\botr$ rule for $\bot$. We provide the cases for $\imp$, $\exists$, and $\forall$ for the inductive step; the $\lor$ and $\land$ cases are simple to verify. Note that for the $\imp$ case the parameters $\vv{\unda_{1}}$ and  $\vv{\unda_{2}}$ are all and only those parameters that occur in $\vv{\unda}$, with $\vv{\unda_{1}}$ and  $\vv{\unda_{2}}$ potentially intersecting. Also, the $\imp$ and $\forall$ cases rely on claim (d) (shown below), and we occasionally use primes (e.g. $\Gamma'$ and $\Delta'$) to indicate that formulae explicitly occurring lower in the proof have been absorbed into the context in order to save space. Last, we let $\vec{\unda} := \unda_{1}, \ldots, \unda_{n}$.

\begin{center}
\resizebox{\columnwidth}{!}{
\begin{tabular}{c}
\AxiomC{$\Pi_{1}$}
\AxiomC{$\Pi_{2}$}
\RightLabel{$\imprl$}
\BinaryInfC{$\R, w \leq v, v \leq u, w \leq u, \vv{\unda} \in D_{w}, \vv{\unda} \in D_{u}, w : A(\vv{\unda_{1}}) \imp B(\vv{\unda_{2}}), u : A(\vv{\unda_{1}}), \Gamma \Rightarrow u : B(\vv{\unda_{2}}), \Delta$}
\RightLabel{$\nd \times n$}
\UnaryInfC{$\R, w \leq v, v \leq u, w \leq u, \vv{\unda} \in D_{w}, w : A(\vv{\unda_{1}}) \imp B(\vv{\unda_{2}}), u : A(\vv{\unda_{1}}), \Gamma \Rightarrow u : B(\vv{\unda_{2}}), \Delta$}
\RightLabel{$\trans$}
\UnaryInfC{$\R, w \leq v, v \leq u, \vv{\unda} \in D_{w}, w : A(\vv{\unda_{1}}) \imp B(\vv{\unda_{2}}), u : A(\vv{\unda_{1}}), \Gamma \Rightarrow u : B(\vv{\unda_{2}}), \Delta$}
\RightLabel{$\imprr$}
\UnaryInfC{$\R, w \leq v, \vv{\unda} \in D_{w}, w : A(\vv{\unda_{1}}) \imp B(\vv{\unda_{2}}), \Gamma \Rightarrow v : A(\vv{\unda_{1}}) \imp B(\vv{\unda_{2}}), \Delta$}
\DisplayProof
\end{tabular}
}
\end{center}

\begin{center}
\begin{tabular}{c c c}
$\Pi_{1}$

&

$= \Big \{$

&

\AxiomC{}
\RightLabel{(d)}
\dashedLine
\UnaryInfC{$\R', \vv{\unda} \in D_{w}, \vv{\unda} \in D_{u}, u : A(\vv{\unda_{1}}), \Gamma' \Rightarrow u : A(\vv{\unda_{1}}), \Delta'$} 
\DisplayProof
\end{tabular}
\end{center}
\begin{center}
\begin{tabular}{c c c}
$\Pi_{2}$

&

$= \Big \{$

&

\AxiomC{}
\RightLabel{(d)}
\dashedLine
\UnaryInfC{$\R', \vv{\unda} \in D_{w}, \vv{\unda} \in D_{u}, u : B(\vv{\unda_{2}}), \Gamma' \Rightarrow u : B(\vv{\unda_{2}}), \Delta$} 
\DisplayProof
\end{tabular}
\end{center}
\begin{center}
\begin{tabular}{c} 
\AxiomC{}
\RightLabel{IH}
\dashedLine
\UnaryInfC{$\R, w \leq v, \undb \in D_{w}, \vv{\unda} \in D_{w}, w : A(\vv{\unda})(\undb / x), \Gamma \Rightarrow v : A(\vv{\unda})(\undb / x), \Delta'$}
\RightLabel{$\existsr$}
\UnaryInfC{$\R, w \leq v, \undb \in D_{w}, \vv{\unda} \in D_{w}, w : A(\vv{a})(\undb / x), \Gamma \Rightarrow v : \exists x A(\vv{a}), \Delta$}
\RightLabel{$\existsl$}
\UnaryInfC{$\R, w \leq v, \vv{a} \in D_{w}, w : \exists x A(\vv{a}), \Gamma \Rightarrow v : \exists x A(\vv{a}), \Delta$}
\DisplayProof
\end{tabular}
\end{center}
\begin{center}
\resizebox{\columnwidth}{!}{
\begin{tabular}{c}
\AxiomC{}
\RightLabel{IH}
\dashedLine
\UnaryInfC{$\R', w \leq u, \vv{\unda} \in D_{u}, \undb \in D_{u}, u : A(\vv{\unda})(\undb / x), w : \forall x A(\vv{\unda}), \Gamma \Rightarrow u : A(\vv{\unda})(\undb / x), \Delta$}
\RightLabel{$\alll$}
\UnaryInfC{$\R', w \leq u, \vv{\unda} \in D_{u}, \undb \in D_{u}, w : \forall x A(\vv{\unda}), \Gamma \Rightarrow u : A(\vv{\unda})(\undb / x), \Delta$}
\RightLabel{$\nd \times n$}
\UnaryInfC{$\R', w \leq u, \undb \in D_{u}, w : \forall x A(\vv{\unda}), \Gamma \Rightarrow u : A(\vv{\unda})(\undb / x), \Delta$}
\RightLabel{$\trans$}
\UnaryInfC{$\R', \undb \in D_{u}, w : \forall x A(\vv{\unda}), \Gamma \Rightarrow u : A(\vv{\unda})(\undb / x), \Delta$}
\RightLabel{$\allr$}
\UnaryInfC{$\R, w \leq v, \vv{\unda} \in D_{w}, w : \forall x A(\vv{a}), \Gamma \Rightarrow v : \forall x A(\vv{a}), \Delta$}
\DisplayProof
\end{tabular}
}
\end{center}
In the $\forall$ case above, we let $\R' := \R, w \leq v, v \leq u, \vv{\unda} \in D_{w}$.

\textit{Claim (i)-(d).} The base case for atomic formulae is shown below; the case for $\bot$ is omitted as it is simple to verify using the $\botr$ rule. We provide the $\imp$, $\exists$, and $\forall$ cases as the cases for the other connectives are straightforward to prove.\\ 

\begin{center}
\begin{tabular}{c}
\AxiomC{}
\RightLabel{$\idfo$}
\UnaryInfC{$\R, w \leq w, \vv{\unda} \in D_{w}, w : p(\vv{\unda}), \Gamma \Rightarrow w : p(\vv{\unda}), \Delta$}
\RightLabel{$\refl$}
\UnaryInfC{$\R, \vv{\unda} \in D_{w}, w : p(\vv{\unda}), \Gamma \Rightarrow w : p(\vv{\unda}), \Delta$}
\DisplayProof
\end{tabular}
\end{center}
\begin{center}
\resizebox{\columnwidth}{!}{
\begin{tabular}{c} 
\AxiomC{$\Pi_{1}$}
\AxiomC{$\Pi_{2}$}
\RightLabel{$\imprl$}
\BinaryInfC{$\R, w \leq v, \vv{\unda} \in D_{w}, \vv{\unda} \in D_{v}, w : A(\vv{\unda_{1}}) \imp B(\vv{\unda_{2}}), v : A(\vv{\unda_{1}}), \Gamma \Rightarrow v : B(\vv{\unda_{2}}), \Delta$}
\RightLabel{$\nd \times n$}
\UnaryInfC{$\R, w \leq v, \vv{\unda} \in D_{w}, w : A(\vv{\unda_{1}}) \imp B(\vv{\unda_{2}}), v : A(\vv{\unda_{1}}), \Gamma \Rightarrow v : B(\vv{\unda_{2}}), \Delta$}
\RightLabel{$\imprr$}
\UnaryInfC{$\R, \vv{\unda} \in D_{w}, w : A(\vv{\unda_{1}}) \imp B(\vv{\unda_{2}}), \Gamma \Rightarrow w : A(\vv{\unda_{1}}) \imp B(\vv{\unda_{2}}), \Delta$}
\DisplayProof
\end{tabular}
}
\end{center}
\begin{center}
\begin{tabular}{c c c}
$\Pi_{1}$

&

$= \Big \{$

&

\AxiomC{}
\RightLabel{IH}
\dashedLine
\UnaryInfC{$\R, w \leq v, \vv{\unda} \in D_{w}, \vv{\unda} \in D_{v}, v : A(\vv{\unda_{1}}), \Gamma' \Rightarrow v : A(\vv{\unda_{1}}), \Delta'$}
\DisplayProof
\end{tabular}
\end{center}
\begin{center}
\begin{tabular}{c c c}
$\Pi_{2}$

&

$= \Big \{$

&

\AxiomC{}
\RightLabel{IH}
\dashedLine
\UnaryInfC{$\R, w \leq v, \vv{\unda} \in D_{w}, \vv{\unda} \in D_{v}, v : B(\vv{\unda_{2}}), \Gamma' \Rightarrow v : B(\vv{\unda_{2}}), \Delta$}
\DisplayProof
\end{tabular}
\end{center}
\begin{center}
\begin{tabular}{c} 
\AxiomC{}
\RightLabel{IH}
\dashedLine
\UnaryInfC{$\R, \undb \in D_{w}, \vv{\unda} \in D_{w}, w : A(\vv{\unda})(\undb/x), \Gamma \Rightarrow w : A(\vv{\unda})(\undb/x), \Delta'$}
\RightLabel{$\existsr$}
\UnaryInfC{$\R, \undb \in D_{w}, \vv{\unda} \in D_{w}, w : A(\vv{\unda})(\undb/x), \Gamma \Rightarrow w : \exists x A(\vv{\unda}), \Delta$}
\RightLabel{$\existsl$}
\UnaryInfC{$\R, \vv{\unda} \in D_{w}, w : \exists x A(\vv{\unda}), \Gamma \Rightarrow w : \exists x A(\vv{\unda}), \Delta$}
\DisplayProof
\end{tabular}
\end{center}

\begin{center}
\resizebox{\columnwidth}{!}{
\begin{tabular}{c}
\AxiomC{}
\RightLabel{IH}
\dashedLine
\UnaryInfC{$\R, w \leq v, \undb \in D_{u}, \vv{\unda} \in D_{w}, \vv{\unda} \in D_{v}, v : A(\vv{\unda})(\undb/x), w : \forall x A(\vv{\unda}), \Gamma \Rightarrow v : A(\vv{\unda})(\undb/x), \Delta$}
\RightLabel{$\alll$}
\UnaryInfC{$\R, w \leq v, \undb \in D_{v}, \vv{\unda} \in D_{w}, \vv{\unda} \in D_{v}, w : \forall x A(\vv{a}), \Gamma \Rightarrow v : A(\vv{\unda})(\undb/x), \Delta$}
\RightLabel{$\nd \times n$}
\UnaryInfC{$\R, w \leq v, \undb \in D_{v}, \vv{\unda} \in D_{w}, w : \forall x A(\vv{a}), \Gamma \Rightarrow v : A(\vv{\unda})(\undb/x), \Delta$}
\RightLabel{$\allr$}
\UnaryInfC{$\R, \vv{\unda} \in D_{w}, w : \forall x A(\vv{\unda}), \Gamma \Rightarrow w : \forall x A(\vv{\unda}), \Delta$}
\DisplayProof
\end{tabular}
}
\end{center}

\textit{Claim (ii).} We consider two sets of rules from $\strucset$ in turn.\\

\textit{The $\{(lsub),(psub)\}$ rules.} Hp-admissibility of $(lsub)$ and $(psub)$ are proven by induction on the height of the given derivation; both are similar to \cite[Lem.~3]{DycNeg12}. The only non-trivial cases for the former are the $\imprr$ and $\allr$ rules of the inductive step, and the non-trivial cases for the latter are the $\existsl$, $\allr$, and $\ned$ rules of the inductive step. In such cases, the side condition must be preserved if the rule is to be applied. Nevertheless, this can be ensured in the usual way (cf. \cite[Lem.~5.1]{DycNeg12}) by invoking IH twice: first, IH replaces the eigenvariable of the $\imprr$, $\existsl$, $\allr$, or $\ned$ rule with a fresh one, and second, the substitution of $(lsub)$ or $(psub)$ is performed, after which, the corresponding rule may be applied to derive the desired conclusion.\\

\textit{The $\{\wk,\ctrrel,\ctrl,\ctrr\}$ rules.} The hp-admissibility of $\wk$, $\ctrrel$, $\ctrl$, and $\ctrr$ is shown by induction on the height of the given derivation. Hp-admissibility of $\wk$ is relatively straightforward; the only non-trivial cases are the $\imprr$, $\existsl$, $\allr$, and $\ned$ rules of the inductive step due to the eigenvariable condition. Such cases are resolved, however, by potentially applying hp-admissibility of $(lsub)$ or $(psub)$ (argued above), then IH, and last the corresponding rule. Hp-admissibility of $\ctrrel$ is simple; any application of the rule to an initial sequent yields an initial sequent, and every case in the inductive step is resolved by applying IH followed by the corresponding rule. The $\ctrl$ and $\ctrr$ cases are also quite straightforward. The only non-trivial cases occur when a principal occurrence of an $\imprr$,  $\existsl$, or $\allr$ rule is contracted. In such cases, the inductive step is solved by invoking the hp-invertibility of each rule (property (iii) below), followed by a potential application of $(lsub)$ or $(psub)$, an application of IH, and finally, the hp-admissibility of $\ctrrel$ if needed.\\

\textit{Claim (iii).} The cases for the propositional rules are straightforward to check, so we omit them. By hp-admissibility of $\wk$, we know that $\existsr$, $\alll$, $\ned$, $\nd$, and $\cd$ are hp-invertible (note that the proof of $\wk$ admissibility does not depend on the hp-invertibility of rules holding, so we may invoke it). We therefore only need to check that $\allr$ and $ \existsl$ are hp-invertible. We prove the claim by induction on the height of the given derivation of $\R, \Gamma \Rightarrow w : \forall x A, \Delta$ for the $\allr$ rule; the proof for $\existsl$ is similar.\\

\textit{Base case.} If the height of the derivation is $0$, then $\R, w \leq v, a \in D_{v}, \Gamma \Rightarrow v : A(\unda / x), \Delta$ is either an instance of $\botr$ or $\idfo$.\\

\textit{Inductive step.} If the last rule applied in the derivation is $\conrl$, $\conrr$, $\disrl$, $\disrr$, $\imprl$, $\refl$, $\trans$, $\existsr$, $\alll$, $\nd$, or $\cd$, then the conclusion follows by applying IH followed by an application of the associated rule. If the last rule applied is $\imprr$, $\existsl$, $\allr$, or $\ned$ (where the principal formula of $\allr$ is in $\Delta$), then we potentially apply hp-admissibility of $(lsub)$ or $(psub)$ (property (ii) above), followed by IH, and then an application of the corresponding rule. If the last rule applied is $\allr$, with $w : \forall x A$ the principal formula, then the premise of the inference gives the desired result.\\

\textit{Claim (iv).} We prove the admissibility of $\cut$ by induction on the lexicographic ordering of tuples $(|A|,h_{1},h_{2})$, where $|A|$ is the complexity of the cut formula $A$, $h_{1}$ is the height of the left premise of $\cut$, and $h_{2}$ is the height of the right premise of $\cut$. We only consider the cases where the cut formula is principal in both premises of $\cut$; the other cases where the cut formula is not principal in one premise, or in both premises, are relatively straightforward to verify, though the large number of cases makes the proof tedious. By the cut admissibility theorem for $\lint$~\cite[Thm.~5.6]{DycNeg12}, we need only verify the cases where the cut formula is of the form $\forall x B$ or $\exists x B$. We show each case in turn below:

\begin{center}
\begin{tabular}{c} 
\AxiomC{$\Pi_{1}$}

\AxiomC{$\Pi_{2}$}

\RightLabel{$\cut$}
\BinaryInfC{$\R, w \leq v, \unda \in D_{v}, \Gamma \Rightarrow \Delta$}

\DisplayProof
\end{tabular}
\end{center}

\begin{center}
\begin{tabular}{c c c}
$\Pi_{1}$

&

$= \Bigg \{$

&

\AxiomC{$\R, w \leq v, w \leq u, \unda \in D_{v}, \undb \in D_{u}, \Gamma \Rightarrow \Delta, u : B(\undb/x)$}
\RightLabel{$\allr$}
\UnaryInfC{$\R, w \leq v, \unda \in D_{v}, \Gamma \Rightarrow \Delta, w : \forall x B$}
\DisplayProof
\end{tabular}
\end{center}

\begin{center}
\begin{tabular}{c c c}
$\Pi_{2}$

&

$= \Bigg \{$

&

\AxiomC{$\R, w \leq v, \unda \in D_{v}, v : B(\unda / x), w : \forall x B, \Gamma \Rightarrow \Delta$}
\RightLabel{$\alll$}
\UnaryInfC{$\R, w \leq v, \unda \in D_{v}, w : \forall x B, \Gamma \Rightarrow \Delta$}
\DisplayProof
\end{tabular}
\end{center}

The case is resolved as follows:

\begin{center}
\begin{tabular}{c}
\AxiomC{$\Pi_{1}'$}
\AxiomC{$\Pi_{2}'$}
\RightLabel{$\cut$}
\BinaryInfC{$\R, w \leq v, \unda \in D_{v}, \Gamma \Rightarrow \Delta$}
\DisplayProof
\end{tabular}
\end{center}

\begin{center}
\begin{tabular}{c c c}
$\Pi_{1}'$

&

$= \Bigg \{$

&

\AxiomC{$\R, w \leq u, \undb \in D_{u}, \Gamma \Rightarrow \Delta, u : B(\undb / x)$}
\RightLabel{$(lsub)$}
\dashedLine
\UnaryInfC{$\R, w \leq v, \undb \in D_{v}, \Gamma \Rightarrow \Delta, v : B(\undb / x)$}
\RightLabel{$(psub)$}
\dashedLine
\UnaryInfC{$\R, w \leq v, \unda \in D_{v}, \Gamma \Rightarrow \Delta, v : B(\unda / x)$}
\DisplayProof
\end{tabular}
\end{center}

\begin{center}
\resizebox{\columnwidth}{!}{
\begin{tabular}{c c c}
$\Pi_{2}'$

&

$= \Bigg \{$

&

\AxiomC{$\R, w \leq v, w \leq u, \unda \in D_{v}, \undb \in D_{u}, \Gamma \Rightarrow \Delta, u : B(\undb / x)$}
\RightLabel{$\wk$}
\dashedLine
\UnaryInfC{$\R, w \leq v, w \leq u, \unda \in D_{v}, \undb \in D_{u}, v : B(\unda / x), \Gamma \Rightarrow \Delta, u : B(\undb / x)$}
\RightLabel{$\allr$}
\UnaryInfC{$\R, w \leq v, \unda \in D_{v}, v : B(\unda / x), \Gamma \Rightarrow \Delta, w : \forall x B$}

\AxiomC{$\Lambda$}

\RightLabel{$\cut$}
\BinaryInfC{$\R, w \leq v, \unda \in D_{v}, v : B(\unda / x), \Gamma \Rightarrow \Delta$}
\DisplayProof
\end{tabular}
}
\end{center}
$$
\Lambda = \R, w \leq v, \unda \in D_{v}, v : B(\unda / x), w : \forall x B, \Gamma \Rightarrow \Delta
$$
Observe that the $\cut$ in $\Pi_{2}'$ has a height $h_{2}$ that is one less than the original $\cut$, and the second $\cut$ is on a formula $B(\unda / x)$ that is of less complexity than the original $\cut$. Let us examine the $\exists x B$ case.

\begin{center}
\resizebox{\columnwidth}{!}{
\begin{tabular}{c}
\AxiomC{$\R, \unda \in D_{w}, \Gamma \Rightarrow \Delta, w: B(\unda/x), w: \exists x B$}
\RightLabel{$\existsr$}
\UnaryInfC{$\R, \unda \in D_{w}, \Gamma \Rightarrow \Delta, w: \exists x B$}

\AxiomC{$\R, \unda \in D_{w}, \undb \in D_{w}, w: B(\undb/x), \Gamma \Rightarrow \Delta$}
\RightLabel{$\existsl$}
\UnaryInfC{$\R, \unda \in D_{w}, w : \exists x B, \Gamma \Rightarrow \Delta$}

\RightLabel{$\cut$}

\BinaryInfC{$\R, \unda \in D_{w}, \Gamma \Rightarrow \Delta$}
\DisplayProof
\end{tabular}
}
\end{center}
The case is resolved as follows:
\begin{center}
\begin{tabular}{c}
\AxiomC{$\Pi_{1}'$}
\AxiomC{$\Pi_{2}'$}
\RightLabel{$\cut$}
\BinaryInfC{$\R, a \in D_{w}, \Gamma \Rightarrow \Delta$}
\DisplayProof
\end{tabular}
\end{center}

\begin{center}
\resizebox{\columnwidth}{!}{
\begin{tabular}{c c c}
$\Pi_{1}'$

&

$= \Bigg \{ \qquad$

&

\AxiomC{$\Lambda$}

\AxiomC{$\R, \unda \in D_{w}, \undb \in D_{w}, w: B(\undb/x), \Gamma \Rightarrow \Delta$}
\RightLabel{$\wk$}
\dashedLine
\UnaryInfC{$\R, \unda \in D_{w}, \undb \in D_{w}, w: B(\undb/x), \Gamma \Rightarrow w : B(\unda/x), \Delta$}
\RightLabel{$\existsl$}
\UnaryInfC{$\R, \unda \in D_{w}, w : \exists x B, \Gamma \Rightarrow w : B(\unda/x), \Delta$}

\RightLabel{$\cut$}

\BinaryInfC{$\R, \unda \in D_{w}, \Gamma \Rightarrow w : B(\unda/x), \Delta$}

\DisplayProof
\end{tabular}
}
\end{center}

$$
\Lambda = \R, \unda \in D_{w}, \Gamma \Rightarrow \Delta, w: B(\unda/x), w: \exists x B
$$

\begin{center}
\begin{tabular}{c c c}
$\Pi_{2}'$

&

$= \Bigg \{$

&

\AxiomC{$\R, \unda \in D_{w}, \undb \in D_{w}, w: B(\undb/x), \Gamma \Rightarrow \Delta$}
\RightLabel{$(psub)$}
\dashedLine
\UnaryInfC{$\R, \unda \in D_{w}, \unda \in D_{w}, w: B(\unda/x), \Gamma \Rightarrow \Delta$}
\RightLabel{$\ctrrel$}
\dashedLine
\UnaryInfC{$\R, \unda \in D_{w}, w: B(\unda/x), \Gamma \Rightarrow \Delta$}
\DisplayProof
\end{tabular}
\end{center}
Observe that the $\cut$ in the $\Pi_{1}'$ inference has a height $h_{1}$ that is one less than the original $\cut$, and the second $\cut$ is on a formula of smaller complexity.\\

\textit{Claim (v).} Soundness is straightforward and is shown by interpreting labelled sequents on $\intfocd$-models (or $\intfond$- models in the case of $\lintfond$) and proving that validity is preserved from premise to conclusion. By~\cite{DycNeg12}, we know that $\lint$ is complete relative to $\int$, so we need only show that $\lintfocd$ can derive quantifier axioms and simulate the inference rules of the axiomatization for $\intfocd$ (see~\cite[Ch.~2.6]{GabSheSkv09}); we assume that none of the quantifiers occurring in formulae are vacuous, and note that the vacuous cases can be derived in a similar fashion. We say that a formula $A(\vv{\unda})$ is derivable in $\lintfocd$ if and only if $\vv{\unda} \in D_{w} \Rightarrow w : A(\vv{\unda})$ is derivable in $\lintfocd$. Also, we use the notation $\nd \times n$ to indicate that the rule $\nd$ was applied some number of times to shift a group of domain atoms forward (which will be made clear by the context).


\begin{center}
\resizebox{\columnwidth}{!}{
\begin{tabular}{c} 
\AxiomC{}
\RightLabel{Prop.~(i)}
\dashedLine
\UnaryInfC{$w \leq u, u \leq u, \vv{\unda} \in D_{w}, \unda \in D_{w}, \unda \in D_{u}, \vv{\unda} \in D_{u}, u : \forall x A(\vv{\unda},x), u :A(\vv{\unda},\unda) \Rightarrow u : A(\vv{\unda},\unda)$}
\RightLabel{$\alll$}
\UnaryInfC{$w \leq u, u \leq u, \vv{\unda} \in D_{w}, \unda \in D_{w}, \unda \in D_{u}, \vv{\unda} \in D_{u}, u : \forall x A(\vv{\unda},x) \Rightarrow u : A(\vv{\unda},\unda)$}
\RightLabel{$\nd \times n$}
\UnaryInfC{$w \leq u, u \leq u, \vv{\unda} \in D_{w}, \unda \in D_{w}, u : \forall x A(\vv{\unda},x) \Rightarrow u : A(\vv{\unda},\unda)$}
\RightLabel{$\refl$}
\UnaryInfC{$w \leq u, \vv{\unda} \in D_{w}, \unda \in D_{w}, u : \forall x A(\vv{\unda},x) \Rightarrow u : A(\vv{\unda},\unda)$}
\RightLabel{$\imprr$}
\UnaryInfC{$\vv{\unda} \in D_{w}, \unda \in D_{w} \Rightarrow w : \forall x A(\vv{\unda},x) \imp A(\vv{\unda},\unda)$}
\DisplayProof
\end{tabular}
}
\end{center}

\begin{center}
\resizebox{\columnwidth}{!}{
\begin{tabular}{c}
\AxiomC{}
\RightLabel{Prop.~(i)}
\dashedLine
\UnaryInfC{$w \leq u, \vv{\unda} \in D_{w}, \unda \in D_{w}, \vv{\unda} \in D_{u}, \unda \in D_{u}, u : A(\vv{\unda},\unda) \Rightarrow u : A(\vv{\unda},\unda), u : \exists x A(\vv{\unda},x)$}
\RightLabel{$\existsr$}
\UnaryInfC{$w \leq u, \vv{\unda} \in D_{w}, \unda \in D_{w}, \vv{\unda} \in D_{u}, \unda \in D_{u}, u : A(\vv{\unda},\unda) \Rightarrow u : \exists x A(\vv{\unda},x)$}
\RightLabel{$\nd \times n$}
\UnaryInfC{$w \leq u, \vv{\unda} \in D_{w}, \unda \in D_{w},  u : A(\vv{\unda},\unda) \Rightarrow u : \exists x A(\vv{\unda},x)$}
\RightLabel{$\imprr$}
\UnaryInfC{$\vv{\unda} \in D_{w}, \unda \in D_{w} \Rightarrow w : A(\vv{\unda},\unda) \imp \exists x A(\vv{\unda},x)$}
\DisplayProof
\end{tabular}
}
\end{center}

\begin{center}
\resizebox{\columnwidth}{!}{
\begin{tabular}{c}
\AxiomC{$\Pi_{1}$}
\AxiomC{$\Pi_{2}$}
\RightLabel{$\imprl$}
\BinaryInfC{$\R, v \leq z, z \leq z, z : B(\vv{\undb}) \imp A(\vv{\unda},\unda), v : \forall x (B(\vv{\undb}) \imp A(\vv{\unda},x)), u : B(\vv{\undb}) \Rightarrow z : A(\vv{\unda},\unda)$}
\RightLabel{$\refl$}
\UnaryInfC{$\R, v \leq z, z : B(\vv{\undb}) \imp A(\vv{\unda},\unda), v : \forall x (B(\vv{\undb}) \imp A(\vv{\unda},x)), u : B(\vv{\undb}) \Rightarrow z : A(\vv{\unda},\unda)$}
\RightLabel{$\alll$}
\UnaryInfC{$\R, v \leq z, v : \forall x (B(\vv{\undb}) \imp A(\vv{\unda},x)), u : B(\vv{\undb}) \Rightarrow z : A(\vv{\unda},\unda)$}
\RightLabel{$\trans$}
\UnaryInfC{$\R, v : \forall x (B(\vv{\undb}) \imp A(\vv{\unda},x)), u : B(\vv{\undb}) \Rightarrow z : A(\vv{\unda},\unda)$}
\RightLabel{$\allr$}
\UnaryInfC{$w \leq v, v \leq u, \vv{\unda} \in D_{w}, \vv{\undb} \in D_{w}, v : \forall x (B(\vv{\undb}) \imp A(\vv{\unda},x)), u : B(\vv{\undb}) \Rightarrow u : \forall x A(\vv{\unda},x)$}
\RightLabel{$\imprr$}
\UnaryInfC{$w \leq v, \vv{\unda} \in D_{w}, \vv{\undb} \in D_{w}, v : \forall x (B(\vv{\undb}) \imp A(\vv{\unda},x)) \Rightarrow v : B(\vv{\undb}) \imp \forall x A(\vv{\unda},x)$}
\RightLabel{$\imprr$}
\UnaryInfC{$\vv{\unda} \in D_{w}, \vv{\undb} \in D_{w} \Rightarrow w : \forall x (B(\vv{\undb}) \imp A(\vv{\unda},x)) \imp (B(\vv{\undb}) \imp \forall x A(\vv{\unda},x))$}
\DisplayProof
\end{tabular}
}
\end{center}
To save space, we let $\R := w \leq v, v \leq u, u \leq z, \vv{\unda} \in D_{w}, \vv{\undb} \in D_{w}, \unda \in D_{z}$ and $\Gamma := z : B(\vv{\undb}) \imp A(\vv{\unda},\unda), v : \forall x (B(\vv{\undb}) \imp A(\vv{\unda},x))$. The proofs $\Pi_{1}$ and $\Pi_{2}$ are as follows:

\begin{center}
\resizebox{\columnwidth}{!}{
\begin{tabular}{c c c}
$\Pi_{1}$

&

$= \Bigg \{$

&

\AxiomC{}
\RightLabel{Prop.~(i)}
\dashedLine
\UnaryInfC{$\R, v \leq z, z \leq z, \vv{\undb} \in D_{v}, \vv{\undb} \in D_{u}, \Gamma, u : B(\vv{\undb}) \Rightarrow z : B(\vv{\undb}), z : A(\vv{\unda},\unda)$}
\RightLabel{$\nd \times k_{1}$}
\UnaryInfC{$\R, v \leq z, z \leq z, \Gamma, u : B(\vv{\undb}) \Rightarrow z : B(\vv{\undb}), z : A(\vv{\unda},\unda)$}
\DisplayProof
\end{tabular}
}
\end{center}

\begin{center}
\resizebox{\columnwidth}{!}{
\begin{tabular}{c c c}
$\Pi_{2}$

&

$= \Bigg \{$

&

\AxiomC{}
\RightLabel{Prop.~(i)}
\dashedLine
\UnaryInfC{$\R, v \leq z, z \leq z, \vv{\unda} \in D_{v}, \vv{\unda} \in D_{z}, \Gamma, u : B(\vv{\undb}), z : A(\vv{\unda},\unda) \Rightarrow z : A(\vv{\unda},\unda)$}
\RightLabel{$\nd \times k_{2}$}
\UnaryInfC{$\R, v \leq z, z \leq z, \Gamma, u : B(\vv{\undb}), z : A(\vv{\unda},\unda) \Rightarrow z : A(\vv{\unda},\unda)$}
\DisplayProof
\end{tabular}
}
\end{center}

The proof of the axiom $\forall x (A(x) \imp B) \imp (\exists x A(x) \imp B)$ is similar to the previous proof. The generalization rule is simulated as shown below:\\

\begin{center}
\begin{tabular}{c}
\AxiomC{$\vv{\unda} \in D_{w}, \unda \in D_{w} \Rightarrow w : A(\vv{\unda},\unda)$}
\RightLabel{$\wk$}
\dashedLine
\UnaryInfC{$u \leq w, \vv{\unda} \in D_{u}, \vv{\unda} \in D_{w}, \unda \in D_{w} \Rightarrow w : A(\vv{\unda},\unda)$}
\RightLabel{$\nd \times n$}
\UnaryInfC{$u \leq w, \vv{\unda} \in D_{u}, \unda \in D_{w} \Rightarrow w : A(\vv{\unda},\unda)$}
\RightLabel{$\allr$}
\UnaryInfC{$\vv{\unda} \in D_{u} \Rightarrow u : \forall x A(\vv{\unda},x)$}
\RightLabel{$(lsub)$}
\dashedLine
\UnaryInfC{$\vv{\unda} \in D_{w} \Rightarrow w : \forall x A(\vv{\unda},x)$}
\DisplayProof
\end{tabular}
\end{center}

\begin{center}
\resizebox{\columnwidth}{!}{
\begin{tabular}{c}
\AxiomC{$\Pi_{1}$}
\AxiomC{$\Pi_{2}$}
\RightLabel{$\disrl$}
\BinaryInfC{$\R, v \leq v, \unda \in D_{v}, v : A(\vv{\unda},\unda) \lor B(\vv{\undb}), v : \forall x (A(\vv{\unda},x) \lor B(\vv{\undb})) \Rightarrow u : A(\vv{\unda},\unda), v : B(\vv{\undb})$}
\RightLabel{$\alll$}
\UnaryInfC{$\R, v \leq v, \unda \in D_{v}, v : \forall x (A(\vv{\unda},x) \lor B(\vv{\undb})) \Rightarrow u : A(\vv{\unda},\unda), v : B(\vv{\undb})$}
\RightLabel{$\cd$}
\UnaryInfC{$\R, v \leq v, v : \forall x (A(\vv{\unda},x) \lor B(\vv{\undb})) \Rightarrow u : A(\vv{\unda},\unda), v : B(\vv{\undb})$}
\RightLabel{$\refl$}
\UnaryInfC{$\R, v : \forall x (A(\vv{\unda},x) \lor B(\vv{\undb})) \Rightarrow u : A(\vv{\unda},\unda), v : B(\vv{\undb})$}
\RightLabel{$\allr$}
\UnaryInfC{$w \leq v, \vv{\unda} \in D_{w}, \vv{\undb} \in D_{w}, v : \forall x (A(\vv{\unda},x) \lor B(\vv{\undb})) \Rightarrow v : \forall x A(\vv{\unda},x), v : B(\vv{\undb})$}
\RightLabel{$\disrr$}
\UnaryInfC{$w \leq v, \vv{\unda} \in D_{w}, \vv{\undb} \in D_{w}, v : \forall x (A(\vv{\unda},x) \lor B(\vv{\undb})) \Rightarrow v : \forall x A(\vv{\unda},x) \lor B(\vv{\undb})$}
\RightLabel{$\imprr$}
\UnaryInfC{$\vv{\unda} \in D_{w}, \vv{\undb} \in D_{w} \Rightarrow w : \forall x (A(\vv{\unda},x) \lor B(\vv{\undb})) \imp \forall x A(\vv{\unda},x) \lor B(\vv{\undb})$}
\DisplayProof
\end{tabular}
}
\end{center}
to save space, we let $\R := w \leq v, v \leq u, \vv{\unda} \in D_{w}, \vv{\undb} \in D_{w}, \unda \in D_{u}$. The proofs $\Pi_{1}$ and $\Pi_{2}$ are as follows (resp.).

\begin{center}
\resizebox{\columnwidth}{!}{
\begin{tabular}{c}
\AxiomC{}
\RightLabel{Prop.~(i)}
\dashedLine
\UnaryInfC{$\R, v \leq v, \vv{\unda} \in D_{v}, \unda \in D_{v}, v : A(\vv{\unda},\unda), v : \forall x (A(\vv{\unda},x) \lor B(\vv{\undb})) \Rightarrow u : A(\vv{\unda},\unda), v : B(\vv{\undb})$}
\RightLabel{$\nd \times k_{1}$}
\UnaryInfC{$\R, v \leq v, \unda \in D_{v}, v : A(\vv{\unda},\unda), v : \forall x (A(\vv{\unda},x) \lor B(\vv{\undb})) \Rightarrow u : A(\vv{\unda},\unda), v : B(\vv{\undb})$}
\DisplayProof
\end{tabular}
}
\end{center}

\begin{center}
\resizebox{\columnwidth}{!}{
\begin{tabular}{c}
\AxiomC{}
\RightLabel{Prop.~(i)}
\dashedLine
\UnaryInfC{$\R, v \leq v, \unda \in D_{v}, \vv{\undb} \in D_{v}, v : B(\vv{\undb}), v : \forall x (A(\vv{\unda},x) \lor B(\vv{\undb})) \Rightarrow u : A(\vv{\unda},\unda), v : B(\vv{\undb})$}
\RightLabel{$\nd \times k_{2}$}
\UnaryInfC{$\R, v \leq v, \unda \in D_{v}, v : B(\vv{\undb}), v : \forall x (A(\vv{\unda},x) \lor B(\vv{\undb})) \Rightarrow u : A(\vv{\unda},\unda), v : B(\vv{\undb})$}
\DisplayProof
\end{tabular}
}
\end{center}

To show that modus ponens can be simulated, we let $\vv{a}$ be all parameters occurring in $A$, $\vv{b}$ be all parameters occurring in $B$, and let $\vv{c}$ consist of all parameters occurring in $B$, but not $A$. The last inference consists of a sequence of $k$ $\ned$ applications that delete all domains atoms containing parameters from $A$, but not $B$.

\begin{center}
\resizebox{\columnwidth}{!}{
\begin{tabular}{c}
\AxiomC{$\vv{\unda} \in D_{w} \sar w : A(\vv{\unda})$}
\dashedLine
\RightLabel{$\wk$}
\UnaryInfC{$\vv{\unda} \in D_{w}, \vv{\undc} \in D_{w} \sar w : A(\vv{\unda})$}

\AxiomC{$\vv{\unda} \in D_{w}, \vv{\undc} \in D_{w} \sar w : A(\vv{\unda}) \imp B(\vv{\undb})$}
\RightLabel{Prop.~(iii)}
\dashedLine
\UnaryInfC{$w \leq u, \vv{\unda} \in D_{w}, \vv{\undc} \in D_{w}, u : A(\vv{\unda}) \sar u : B(\vv{\undb})$}
\RightLabel{$(lsub)$}
\UnaryInfC{$w \leq w, \vv{\unda} \in D_{w}, \vv{\undc} \in D_{w}, w : A(\vv{\unda}) \sar w : B(\vv{\undb})$}
\RightLabel{$\refl$}
\UnaryInfC{$\vv{\unda} \in D_{w}, \vv{\undc} \in D_{w}, w : A(\vv{\unda}) \sar w : B(\vv{\undb})$}

\RightLabel{$\cut$}
\BinaryInfC{$\vv{\unda} \in D_{w}, \vv{\undc} \in D_{w} \sar w : B(\vv{\undb})$}
\RightLabel{$\ned \times k$}
\UnaryInfC{$\vv{\undb} \in D_{w} \sar w : B(\vv{\undb})$}
\DisplayProof
\end{tabular}
}
\end{center}
\qed
\end{proof}


\begin{customlem}{\ref{lem:extended-lint-properties}} Let $\mathsf{G3X}^{*} \in \{\lintfond^{*}, \lintfocd^{*}\}$. The calculi $\lint^{*}$, $\lintfond^{*}$, and $\lintfocd^{*}$ have the following properties:
\begin{itemize}

\item[$(i)$] 

\begin{itemize}

\item[(a)] For all $A \in \lang$, $ \vdash_{\lint^{*}} \R,w \leq v, w : A, \Gamma \Rightarrow v : A, \Delta$;

\item[(b)] For all $A \in \lang$, $ \vdash_{\lint^{*}} \R,w:A,\Gamma \Rightarrow \Delta, w :A$; 

\item[(c)] For all $A \in \langfo$, $\vdash_{\mathsf{G3X}^{*}} \R,w \leq v, \vv{\unda} \in D_{w}, w : A(\vv{\unda}), \Gamma \Rightarrow v : A(\vv{\unda}), \Delta$; 

\item[(d)] For all $A \in \langfo$, $\vdash_{\mathsf{G3X}^{*}} \R, \vv{\unda} \in D_{w}, w:A(\vv{\unda}),\Gamma \Rightarrow \Delta, w :A(\vv{\unda})$;

\end{itemize}

\item[(ii)] The rules $\{(lsub),(psub),\wk,\ctrrel,\ctrr\}$ are hp-admissible;

\item[(iii)] With the exception of $\{\conrl,\existsl\}$, all rules are hp-invertible;

\item[(iv)] The rules $\{\conrl,\existsl\}$ are invertible;

\item[(v)] The rule $\ctrl$ is admissible.

\end{itemize}
\end{customlem}

\begin{proof} We show that the results hold when all rules from $\lint^{*}$, $\lintfond^{*}$, and $\lintfocd^{*}$ are considered.

\textit{Claim (i).} Similar to Thm.~\ref{thm:lint-properties}-(i).

\textit{Claim (ii).} By induction on the height of the given derivation---similar to the proofs given for Thm.~\ref{thm:lint-properties}-(ii). Note that the hp-admissibility of $\ctrr$ invokes property (iii) below.

\textit{Claim (iii).} Hp-invertibility of $\imprl$, $\implnew$, $\negl$, $\refl$, $\trans$, $\lift$, $\ned$, $\existsr$, $\existsndr$, $\existsnedr$, $\existscdr$, $\existscedr$, $\alll$, $\allndl$, $\allnedl$, $\allcdl$, $\allcedl$, $\nd$, $\cd$ follow from hp-admissibility of $\wk$ (property (ii) above). Hence, we need only prove invertibility of the $\impr$, $\conrr$, $\disrl$, $\disrr$, $\negr$, $\allr$, and $\allcdr$ rules. The result is shown by induction on the height of the given derivation and is similar to Thm.~\ref{thm:lint-properties}-(iii). 

\textit{Claim (iv).} The addition of $\lift$ and $\nd$ to our calculus breaks the \emph{height preserving} invertibility of $\conrl$ and $\existsl$. Nevertheless, it is worthwhile to note that the rule $(lift')$ (shown below) ought to allow for hp-invertibility of $\conrl$ in the propositional calculus, and $(nd')$ (shown below) ought to allow for the hp-invertibility of $\conrl$ and $\existsl$ in the first-order calculi (which would have the effect that $\ctrl$ is \emph{hp-admissible} in both calculi) while retaining the soundness and cut-free completeness of each system. The notation $v : \Gamma'$ is used to represent multisets of formulae labelled with $v$.
\begin{center}
\resizebox{\columnwidth}{!}{
\begin{tabular}{c c}
\AxiomC{$\R, w \leq u, \Gamma, w : \Gamma', u : \Gamma' \Rightarrow \Delta$}
\RightLabel{$(lift')$}
\UnaryInfC{$\R, w \leq u, \Gamma, w : \Gamma' \Rightarrow \Delta$}
\DisplayProof

&

\AxiomC{$\R, w \leq u, \vv{\unda} \in D_{w}, \vv{\unda} \in D_{u}, \Gamma, w : \Gamma', u : \Gamma' \Rightarrow \Delta$}
\RightLabel{$(nd')$}
\UnaryInfC{$\R, w \leq u, \vv{\unda} \in D_{w}, \Gamma, w : \Gamma' \Rightarrow \Delta$}
\DisplayProof
\end{tabular}
}
\end{center}
Despite this shortcoming, the rules are still invertible. To show this, we prove the following two claims by induction on the height of the given derivation.\\

\begin{itemize}

\item[(a)] If $\R, w_{1} : A \land B, \ldots, w_{n} : A \land B, \Gamma \Rightarrow \Delta$ is provable, then so is the sequent $\R, w_{1} : A, w_{1} : B, \ldots, w_{n} : A, w_{n} : B, \Gamma \Rightarrow \Delta$.

\item[(b)] If $\R, w_{1} : \exists x A, \ldots, w_{n} : \exists x A, \Gamma \Rightarrow \Delta$ is provable, then so is the sequent $\R, \unda_{1} \in D_{w_{1}}, \ldots, \unda_{n} \in D_{w_{n}}, w_{1} : A(\unda_{1} / x), \ldots, w_{n} : A(\unda_{n} / x), \Gamma \Rightarrow \Delta$.

\end{itemize}

\textit{Claim (a).} The base case is trivial, so we move on to the inductive step.\\

\textit{Inductive step.} Excluding the case of $\lift$, the result follows by applying IH followed by the relevant rule, or in the case where one of the conjunction formulae is principal in a $\conrr$ inference, apply IH to the premise (if needed) to obtain the desired result. If none of the conjunctions are active in a $\lift$ inference, then apply IH followed by an application of $\lift$. If one of the conjunctions is principal in an application of $\lift$ (as shown below top), then resolve the case as shown below bottom; we assume w.l.o.g. that the label of the principal formula is $w_{1}$.
\begin{center}
\AxiomC{$\R, u \leq w_{1}, u : A \land B, w_{1} : A \land B, \ldots, w_{n} : A \land B, \Gamma \Rightarrow \Delta$}
\RightLabel{$\lift$}
\UnaryInfC{$\R, u \leq w_{1}, u : A \land B, \ldots, w_{n} : A \land B, \Gamma \Rightarrow \Delta$}
\DisplayProof
\end{center}
\begin{center}
\AxiomC{}
\RightLabel{IH}
\dashedLine
\UnaryInfC{$\R, u \leq w_{1}, u : A, u :B, w_{1} : A, w_{1} : B, \ldots, w_{n} : A, w_{n} : B, \Gamma \Rightarrow \Delta$}
\RightLabel{$\lift$}
\UnaryInfC{$\R, u \leq w_{1}, u :A, u :B, w_{1} : B, \ldots, w_{n} : A, w_{n} : B, \Gamma \Rightarrow \Delta$}
\RightLabel{$\lift$}
\UnaryInfC{$\R, u \leq w_{1}, u : A, u :B, \ldots, w_{n} : A, w_{n} : B, \Gamma \Rightarrow \Delta$}
\DisplayProof
\end{center}
Notice that the two applications of $\lift$ needed to derive the desired conclusion break the hp-invertibility of the $\conrl$ rule.\\

\textit{Claim (b).} The base case is trivial, so we move on to the inductive step.\\

\textit{Inductive step.} All cases, with the exception of the one given below top (where one of our existential formulae is principal in an application of $\lift$), are either resolved by applying IH followed by the corresponding rule, or if one of the existential formulae is principal in an $\existsr$ inference, then apply IH to the premise (if needed) to obtain the desired result. The non-trivial case given below top is resolved as shown below bottom; we assume w.l.o.g. that the label of the principal formula is $w_{1}$.
\begin{center}
\AxiomC{$\R, u \leq w_{1}, u : \exists x A, w_{1} : \exists x A, \ldots, w_{n} : \exists x A, \Gamma \Rightarrow \Delta$}
\RightLabel{$\lift$}
\UnaryInfC{$\R, u \leq w_{1}, u : \exists x A, \ldots, w_{n} : \exists x A, \Gamma \Rightarrow \Delta$}
\DisplayProof
\end{center}
\begin{center}
\resizebox{\columnwidth}{!}{
\AxiomC{}
\RightLabel{IH}
\dashedLine
\UnaryInfC{$\R, u \leq w_{1}, \unda \in D_{u}, \unda_{1} \in D_{w_{1}}, \ldots, \unda_{n} \in D_{w_{n}}, u : A(\unda / x), w_{1} : A(\unda_{1} / x), \ldots, w_{n} : A(\unda_{n} / x), \Gamma \Rightarrow \Delta$}
\RightLabel{$(psub)$}
\dashedLine
\UnaryInfC{$\R, u \leq w_{1}, \unda \in D_{u}, \unda \in D_{w_{1}}, \ldots, \unda_{n} \in D_{w_{n}}, u : A(\unda / x), w_{1} : A(\unda / x), \ldots, w_{n} : A(\unda_{n} / x), \Gamma \Rightarrow \Delta$}
\RightLabel{$\lift$}
\UnaryInfC{$\R, u \leq w_{1}, \unda \in D_{u}, \unda \in D_{w_{1}}, \ldots, \unda_{n} \in D_{w_{n}}, u : A(\unda / x),\ldots, w_{n} : A(\unda_{n} / x), \Gamma \Rightarrow \Delta$}
\RightLabel{$\nd$}
\UnaryInfC{$\R, u \leq w_{1}, \unda \in D_{u}, \ldots, \unda_{n} \in D_{w_{n}}, u : A(\unda / x),\ldots, w_{n} : A(\unda_{n} / x), \Gamma \Rightarrow \Delta$}
\DisplayProof
}
\end{center}
Observe that the use of $\lift$ and $\nd$ breaks height-preserving invertibility of the rule.\\

Propositions (a) and (b) imply the invertibility of $\conrl$ and $\existsl$.\\

\textit{Claim (v).} By induction on pairs of the form $(|A|,h)$ where $|A|$ is the complexity of the contraction formula $A$ and $h$ is the height of the derivation. The proof makes use of properties (ii)-- (iv).
\qed
\end{proof}

\begin{customlem}{\ref{lem:trans-refl-fond-elim}}
The $\refl$ and $\trans$ rules are admissible in the calculus $\lintfond^{*}$.
\end{customlem}

\begin{proof} The $\trans$ elimination cases are given below with each non-trivial case being followed by a proof of how it is resolved. 
In the $\idnd$ case, We let $\R' := \R, \unda_{1} \in D_{v_{1}}, \ldots, \unda_{n} \in D_{v_{n}}$.

\begin{flushleft}
\begin{tabular}{c}
\AxiomC{}
\RightLabel{$\idfond$}
\UnaryInfC{$\R', u \leq z, z \leq z', u \leq z', \Gamma, w : p(\vv{\unda}) \Rightarrow w : p(\vv{\unda}), \Delta$}
\RightLabel{$\trans$}
\UnaryInfC{$\R', u \leq z, z \leq z', \Gamma, w : p(\vv{\unda}) \Rightarrow w : p(\vv{\unda}), \Delta$}
\DisplayProof
\end{tabular}
\end{flushleft}

\begin{flushright}
\begin{tabular}{c}
\AxiomC{}
\RightLabel{$\idfond$}
\UnaryInfC{$\R', u \leq z, z \leq z', \Gamma, w : p(\vv{\unda}) \Rightarrow w : p(\vv{\unda}), \Delta$}
\DisplayProof
\end{tabular}
\end{flushright}

\begin{flushleft}
\AxiomC{$\R, u \leq z, z \leq z', u \leq z', \unda \in D_{v}, \Gamma, w : A(\unda/x), w : \forall x A \Rightarrow \Delta$}
\RightLabel{$\allndl$}
\UnaryInfC{$\R, u \leq z, z \leq z', u \leq z', \unda \in D_{v}, \Gamma, w : \forall x A \Rightarrow \Delta$}
\RightLabel{$\trans$}
\UnaryInfC{$\R, u \leq z, z \leq z', \unda \in D_{v}, \Gamma, w : \forall x A \Rightarrow \Delta$}
\DisplayProof
\end{flushleft}
\begin{flushright}
\AxiomC{}
\RightLabel{IH}
\dashedLine
\UnaryInfC{$\R, u \leq z, z \leq z', \unda \in D_{v}, w : A(\unda/x), w : \forall x A, \Gamma \Rightarrow \Delta$}
\RightLabel{$\allndl$}
\UnaryInfC{$\R, u \leq z, z \leq z', \unda \in D_{v}, w : \forall x A, \Gamma \Rightarrow \Delta$}
\DisplayProof
\end{flushright}

\begin{flushleft}
\AxiomC{$\R, u \leq z, z \leq z', u \leq z', \unda \in D_{v}, \Gamma, w : A(\unda/x), w : \forall x A \Rightarrow \Delta$}
\RightLabel{$\allnedl$}
\UnaryInfC{$\R, u \leq z, z \leq z', u \leq z', \Gamma, w : \forall x A \Rightarrow \Delta$}
\RightLabel{$\trans$}
\UnaryInfC{$\R, u \leq z, z \leq z', \Gamma, w : \forall x A \Rightarrow \Delta$}
\DisplayProof
\end{flushleft}
\begin{flushright}
\AxiomC{}
\RightLabel{IH}
\dashedLine
\UnaryInfC{$\R, u \leq z, z \leq z', \unda \in D_{v}, w : A(\unda/x), w : \forall x A, \Gamma \Rightarrow \Delta$}
\RightLabel{$\allnedl$}
\UnaryInfC{$\R, u \leq z, z \leq z', w : \forall x A, \Gamma \Rightarrow \Delta$}
\DisplayProof
\end{flushright}

\begin{flushleft}
\AxiomC{$\R, u \leq z, z \leq z', u \leq z', \unda \in D_{v}, \Gamma \Rightarrow w : A(\unda/x), w : \exists x A, \Delta$}
\RightLabel{$\existsndr$}
\UnaryInfC{$\R, u \leq z, z \leq z', u \leq z', \unda \in D_{v}, \Gamma \Rightarrow w : \exists x A, \Delta$}
\RightLabel{$\trans$}
\UnaryInfC{$\R, u \leq z, z \leq z', \unda \in D_{v}, \Gamma \Rightarrow w : \exists x A, \Delta$}
\DisplayProof
\end{flushleft}
\begin{flushright}
\AxiomC{}
\RightLabel{IH}
\dashedLine
\UnaryInfC{$\R, u \leq z, z \leq z', \unda \in D_{v}, \Gamma \Rightarrow w : A(\unda/x), w : \exists x A, \Delta$}
\RightLabel{$\existsndr$}
\UnaryInfC{$\R, u \leq z, z \leq z', \unda \in D_{v}, \Gamma \Rightarrow  w : \exists x A, \Delta$}
\DisplayProof
\end{flushright}

\begin{flushleft}
\AxiomC{$\R, u \leq z, z \leq z', u \leq z', \Gamma, \unda \in D_{v} \Rightarrow w : A(\unda/x), w : \exists x A, \Delta$}
\RightLabel{$\existsnedr$}
\UnaryInfC{$\R, u \leq z, z \leq z', u \leq z', \Gamma \Rightarrow w : \exists x A, \Delta$}
\RightLabel{$\trans$}
\UnaryInfC{$\R, u \leq z, z \leq z', \Gamma \Rightarrow w : \exists x A, \Delta$}
\DisplayProof
\end{flushleft}
\begin{flushright}
\AxiomC{}
\RightLabel{IH}
\dashedLine
\UnaryInfC{$\R, u \leq z, z \leq z', \unda \in D_{v}, \Gamma \Rightarrow w : A(\unda/x), w : \exists x A, \Delta$}
\RightLabel{$\existsnedr$}
\UnaryInfC{$\R, u \leq z, z \leq z', \Gamma \Rightarrow  w : \exists x A, \Delta$}
\DisplayProof
\end{flushright}

\begin{flushleft}
\AxiomC{$\R, w \leq u, u \leq v, w \leq v, \unda \in D_{w}, \unda \in D_{v}, \Gamma \Rightarrow \Delta$}
\RightLabel{$\nd$}
\UnaryInfC{$\R, w \leq u, u \leq v, w \leq v, \unda \in D_{w}, \Gamma \Rightarrow \Delta$}
\RightLabel{$\trans$}
\UnaryInfC{$\R, w \leq u, u \leq v, \unda \in D_{w}, \Gamma \Rightarrow \Delta$}
\DisplayProof
\end{flushleft}
\begin{flushright}
\AxiomC{$\R, w \leq u, u \leq v, w \leq v, \unda \in D_{w}, \unda \in D_{v}, \Gamma \Rightarrow \Delta$}
\RightLabel{$\wk$}
\dashedLine
\UnaryInfC{$\R, w \leq u, u \leq v, w \leq v, \unda \in D_{w}, \unda \in D_{u}, \unda \in D_{v}, \Gamma \Rightarrow \Delta$}
\RightLabel{IH}
\dashedLine
\UnaryInfC{$\R, w \leq u, u \leq v, \unda \in D_{w}, \unda \in D_{u}, \unda \in D_{v}, \Gamma \Rightarrow \Delta$}
\RightLabel{$\nd$}
\UnaryInfC{$\R, w \leq u, u \leq v, \unda \in D_{w}, \unda \in D_{u}, \Gamma \Rightarrow \Delta$}
\RightLabel{$\nd$}
\UnaryInfC{$\R, w \leq u, u \leq v, \unda \in D_{w}, \Gamma \Rightarrow \Delta$}
\DisplayProof
\end{flushright}

As with the admissibility proof for $\refl$, the side condition of the $\idnd$ rule still holds after $u \leq z'$ is deleted, and the side condition of the $\existsndr$, $\existsnedr$, $\allndl$, and $\allnedl$ cases continues to hold after the invocation of IH. If a directed path traverses the relational atom $u \leq z'$ in the $\idnd$ case, a directed path still exists by replacing $u \leq z'$ with $u \leq z, z \leq z'$ in the given path. In the other cases, if the directed path traverses the $u \leq z'$ relational atom, then the relational atoms $u \leq z, z \leq z'$ are still present in the sequent after invoking IH, ensuring that a path between $v$ and $w$ continues to exist.
\qed
\end{proof}

\begin{customlem}{\ref{lem:trans-refl-focd-elim}}
The $\refl$ and $\trans$ rules are admissible in the calculus $\lintfocd^{*}$.
\end{customlem}

\begin{proof} We prove the result by induction on the height of the given derivation. We assume w.l.o.g. that the last inference in the given derivation is an instance of $\refl$ or $\trans$, and that this is the only instance of the rule in the given derivation. The general result follows by successively eliminating topmost occurrences of $\refl$ or $\trans$ from a given derivation until it is free of such inferences. By Lem.~\ref{lm:structural-rules-permutation} and Lem.~\ref{lem:old-instance-of-new-cd}, we need only consider the non-trivial $\idcd$, $\existscdr$, $\existscedr$, $\allcdl$, $\allcedl$, $\nd$, and $\cd$ cases.

\textit{Base case.} We let $\R' := \R, \unda_{1} \in D_{v_{1}}, \ldots, \unda_{n} \in D_{v_{n}}$ in the $\idnd$ cases below. In both the $\refl$ and $\trans$ cases, the side condition that there is an undirected path from $v_{i}$ to $w$, for each $i \in \{1, \ldots, n\}$, holds in the end sequent. If none of the undirected paths from $v_{i}$ to $w$ contain $u \leq u$ ($u \leq z'$), then the paths are present in $\R'$ ($\R',u \leq z, z \leq z'$, resp.), and if an undirected path from $v_{i}$ to $w$ contains $u \leq u$ ($u \leq z'$, resp.), then deleting each occurrence of $u \leq u$ from the undirected path (replacing each occurrence of $u \leq z'$ with $u \leq z, z \leq z'$, resp.) gives a new path from $v_{i}$ to $w$.
\begin{center}
\resizebox{\columnwidth}{!}{
\begin{tabular}{c c}
\AxiomC{}
\RightLabel{$\idnd$}
\UnaryInfC{$\R', u \leq u, \Gamma, w : p(\vv{\unda}) \Rightarrow w : p(\vv{\unda}), \Delta$}
\RightLabel{$\refl$}
\UnaryInfC{$\R', \Gamma, w : p(\vv{\unda}) \Rightarrow w : p(\vv{\unda}), \Delta$}
\DisplayProof
&
\AxiomC{}
\RightLabel{$\idcd$}
\UnaryInfC{$\R', \Gamma, w : p(\vv{\unda}) \Rightarrow w : p(\vv{\unda}), \Delta$}
\DisplayProof
\end{tabular}
}
\end{center}

\begin{center}
\begin{tabular}{c}
\AxiomC{}
\RightLabel{$\idfocd$}
\UnaryInfC{$\R', u \leq z, z \leq z', u \leq z', \Gamma, w : p(\vv{\unda}) \Rightarrow w : p(\vv{\unda}), \Delta$}
\RightLabel{$\trans$}
\UnaryInfC{$\R', u \leq z, z \leq z', \Gamma, w : p(\vv{\unda}) \Rightarrow w : p(\vv{\unda}), \Delta$}
\DisplayProof
\end{tabular}
\end{center}

\begin{center}
\begin{tabular}{c}
\AxiomC{}
\RightLabel{$\idfocd$}
\UnaryInfC{$\R', u \leq z, z \leq z', \Gamma, w : p(\vv{\unda}) \Rightarrow w : p(\vv{\unda}), \Delta$}
\DisplayProof
\end{tabular}
\end{center}

\textit{Inductive step.} For the inductive step, we consider the non-trivial $\existscdr$, $\existscedr$, $\allcdl$, $\allcedl$, $\nd$, and $\cd$ cases. We first show how to permute $\refl$ above each of these rules, and then focus on the permutation of $\trans$ above each considered rule.

\begin{flushleft}
\AxiomC{$\R, u \leq u, \unda \in D_{v}, \Gamma \Rightarrow w : A(\unda/x), w : \exists x A, \Delta$}
\RightLabel{$\existscdr$}
\UnaryInfC{$\R, u \leq u, \unda \in D_{v}, \Gamma \Rightarrow w : \exists x A, \Delta$}
\RightLabel{$\refl$}
\UnaryInfC{$\R, \unda \in D_{v}, \Gamma \Rightarrow w : \exists x A, \Delta$}
\DisplayProof
\end{flushleft}
\begin{flushright}
\AxiomC{}
\RightLabel{IH}
\dashedLine
\UnaryInfC{$\R, \unda \in D_{v}, \Gamma \Rightarrow w : A(\unda/x), w : \exists x A, \Delta$}
\RightLabel{$\existscdr$}
\UnaryInfC{$\R, \unda \in D_{v}, \Gamma \Rightarrow w : \exists x A, \Delta$}
\DisplayProof
\end{flushright}

\begin{flushleft}
\AxiomC{$\R, u \leq u, \unda \in D_{v}, \Gamma \Rightarrow w : A(\unda/x), w : \exists x A, \Delta$}
\RightLabel{$\existscedr$}
\UnaryInfC{$\R, u \leq u, \Gamma \Rightarrow w : \exists x A, \Delta$}
\RightLabel{$\refl$}
\UnaryInfC{$\R, \Gamma \Rightarrow w : \exists x A, \Delta$}
\DisplayProof
\end{flushleft}
\begin{flushright}
\AxiomC{}
\RightLabel{IH}
\dashedLine
\UnaryInfC{$\R, \unda \in D_{v}, \Gamma \Rightarrow w : A(\unda/x), w : \exists x A, \Delta$}
\RightLabel{$\existscedr$}
\UnaryInfC{$\R, \Gamma \Rightarrow w : \exists x A, \Delta$}
\DisplayProof
\end{flushright}

\begin{flushleft}
\AxiomC{$\R, u \leq u, \unda \in D_{v}, w : A(\unda/x), w : \forall x A, \Gamma \Rightarrow \Delta$}
\RightLabel{$\allcdl$}
\UnaryInfC{$\R, u \leq u, \unda \in D_{v}, w : \forall x A, \Gamma \Rightarrow \Delta$}
\RightLabel{$\refl$}
\UnaryInfC{$\R, \unda \in D_{v}, w : \forall x A, \Gamma \Rightarrow \Delta$}
\DisplayProof
\end{flushleft}
\begin{flushright}
\AxiomC{}
\RightLabel{IH}
\dashedLine
\UnaryInfC{$\R, \unda \in D_{v}, w : A(\unda/x), w : \forall x A, \Gamma \Rightarrow \Delta$}
\RightLabel{$\allcdl$}
\UnaryInfC{$\R, \unda \in D_{v}, w : \forall x A, \Gamma \Rightarrow \Delta$}
\DisplayProof
\end{flushright}

\begin{flushleft}
\AxiomC{$\R, u \leq u, \unda \in D_{v}, \Gamma, w : A(\unda/x), w : \forall x A \Rightarrow \Delta$}
\RightLabel{$\allcedl$}
\UnaryInfC{$\R, u \leq u, \Gamma, w : \forall x A \Rightarrow \Delta$}
\RightLabel{$\refl$}
\UnaryInfC{$\R, w : \forall x A, \Gamma \Rightarrow \Delta$}
\DisplayProof
\end{flushleft}
\begin{flushright}
\AxiomC{}
\RightLabel{IH}
\dashedLine
\UnaryInfC{$\R, \unda \in D_{v}, \Gamma, w : A(\unda/x), w : \forall x A, \Gamma \Rightarrow \Delta$}
\RightLabel{$\allcedl$}
\UnaryInfC{$\R, w : \forall x A, \Gamma \Rightarrow \Delta$}
\DisplayProof
\end{flushright}

\begin{flushleft}
\AxiomC{$\R, w \leq w, \unda \in D_{w}, \unda \in D_{w}, \Gamma \Rightarrow \Delta$}
\RightLabel{$\nd$}
\UnaryInfC{$\R, w \leq w, \unda \in D_{w}, \Gamma \Rightarrow \Delta$}
\RightLabel{$\refl$}
\UnaryInfC{$\R, \unda \in D_{w}, \Gamma \Rightarrow \Delta$}
\DisplayProof
\end{flushleft}
\begin{flushright}
\AxiomC{$\R, w \leq w, \unda \in D_{w}, \unda \in D_{w}, \Gamma \Rightarrow \Delta$}
\RightLabel{$\ctrrel$}
\dashedLine
\UnaryInfC{$\R, w \leq w, \unda \in D_{w}, \Gamma \Rightarrow \Delta$}
\RightLabel{IH}
\dashedLine
\UnaryInfC{$\R, \unda \in D_{w}, \Gamma \Rightarrow \Delta$}
\DisplayProof
\end{flushright}

\begin{flushleft}
\AxiomC{$\R, w \leq w, \unda \in D_{w}, \unda \in D_{w}, \Gamma \Rightarrow \Delta$}
\RightLabel{$\cd$}
\UnaryInfC{$\R, w \leq w, \unda \in D_{w}, \Gamma \Rightarrow \Delta$}
\RightLabel{$\refl$}
\UnaryInfC{$\R, \unda \in D_{w}, \Gamma \Rightarrow \Delta$}
\DisplayProof
\end{flushleft}
\begin{flushright}
\AxiomC{$\R, w \leq w, \unda \in D_{w}, \unda \in D_{w}, \Gamma \Rightarrow \Delta$}
\RightLabel{$\ctrrel$}
\dashedLine
\UnaryInfC{$\R, w \leq w, \unda \in D_{w}, \Gamma \Rightarrow \Delta$}
\RightLabel{IH}
\dashedLine
\UnaryInfC{$\R, \unda \in D_{w}, \Gamma \Rightarrow \Delta$}
\DisplayProof
\end{flushright}

In the $\existscdr$, $\existscedr$, $\allcdl$, and $\allcedl$ cases above, observe that each rule can be applied after invoking IH for the following reason: if the undirected path from $v$ to $w$ does not go through $u$, then the path is still present after applying IH, and if the undirected path contains $u \leq u$, then deleting all occurrences of $u \leq u$ from the undirected path yields another undirected path between $v$ and $w$. The $\trans$ cases are considered below:

\begin{flushleft}
\AxiomC{$\R, u \leq z, z \leq z', u \leq z', \unda \in D_{v}, \Gamma, w : A(\unda/x), w : \forall x A \Rightarrow \Delta$}
\RightLabel{$\allcdl$}
\UnaryInfC{$\R, u \leq z, z \leq z', u \leq z', \unda \in D_{v}, \Gamma, w : \forall x A \Rightarrow \Delta$}
\RightLabel{$\trans$}
\UnaryInfC{$\R, u \leq z, z \leq z', \unda \in D_{v}, \Gamma, w : \forall x A \Rightarrow \Delta$}
\DisplayProof
\end{flushleft}
\begin{flushright}
\AxiomC{}
\RightLabel{IH}
\dashedLine
\UnaryInfC{$\R, u \leq z, z \leq z', \unda \in D_{v}, w : A(\unda/x), w : \forall x A, \Gamma \Rightarrow \Delta$}
\RightLabel{$\allcdl$}
\UnaryInfC{$\R, u \leq z, z \leq z', \unda \in D_{v}, w : \forall x A, \Gamma \Rightarrow \Delta$}
\DisplayProof
\end{flushright}

\begin{flushleft}
\AxiomC{$\R, u \leq z, z \leq z', u \leq z', \unda \in D_{v}, \Gamma, w : A(\unda/x), w : \forall x A \Rightarrow \Delta$}
\RightLabel{$\allcedl$}
\UnaryInfC{$\R, u \leq z, z \leq z', u \leq z', \Gamma, w : \forall x A \Rightarrow \Delta$}
\RightLabel{$\trans$}
\UnaryInfC{$\R, u \leq z, z \leq z', \Gamma, w : \forall x A \Rightarrow \Delta$}
\DisplayProof
\end{flushleft}
\begin{flushright}
\AxiomC{}
\RightLabel{IH}
\dashedLine
\UnaryInfC{$\R, u \leq z, z \leq z', \unda \in D_{v}, w : A(\unda/x), w : \forall x A, \Gamma \Rightarrow \Delta$}
\RightLabel{$\allcedl$}
\UnaryInfC{$\R, u \leq z, z \leq z', w : \forall x A, \Gamma \Rightarrow \Delta$}
\DisplayProof
\end{flushright}

\begin{flushleft}
\AxiomC{$\R, u \leq z, z \leq z', u \leq z', \unda \in D_{v}, \Gamma \Rightarrow w : A(\unda/x), w : \exists x A, \Delta$}
\RightLabel{$\existscdr$}
\UnaryInfC{$\R, u \leq z, z \leq z', u \leq z', \unda \in D_{v}, \Gamma \Rightarrow w : \exists x A, \Delta$}
\RightLabel{$\trans$}
\UnaryInfC{$\R, u \leq z, z \leq z', \unda \in D_{v}, \Gamma \Rightarrow w : \exists x A, \Delta$}
\DisplayProof
\end{flushleft}
\begin{flushright}
\AxiomC{}
\RightLabel{IH}
\dashedLine
\UnaryInfC{$\R, u \leq z, z \leq z', \unda \in D_{v}, \Gamma \Rightarrow w : A(\unda/x), w : \exists x A, \Delta$}
\RightLabel{$\existscdr$}
\UnaryInfC{$\R, u \leq z, z \leq z', \unda \in D_{v}, \Gamma \Rightarrow  w : \exists x A, \Delta$}
\DisplayProof
\end{flushright}

\begin{flushleft}
\AxiomC{$\R, u \leq z, z \leq z', u \leq z', \Gamma, \unda \in D_{v} \Rightarrow w : A(\unda/x), w : \exists x A, \Delta$}
\RightLabel{$\existscedr$}
\UnaryInfC{$\R, u \leq z, z \leq z', u \leq z', \Gamma \Rightarrow w : \exists x A, \Delta$}
\RightLabel{$\trans$}
\UnaryInfC{$\R, u \leq z, z \leq z', \Gamma \Rightarrow w : \exists x A, \Delta$}
\DisplayProof
\end{flushleft}
\begin{flushright}
\AxiomC{}
\RightLabel{IH}
\dashedLine
\UnaryInfC{$\R, u \leq z, z \leq z', \unda \in D_{v}, \Gamma \Rightarrow w : A(\unda/x), w : \exists x A, \Delta$}
\RightLabel{$\existscedr$}
\UnaryInfC{$\R, u \leq z, z \leq z', \Gamma \Rightarrow  w : \exists x A, \Delta$}
\DisplayProof
\end{flushright}

\begin{flushleft}
\AxiomC{$\R, w \leq u, u \leq v, w \leq v, \unda \in D_{w}, \unda \in D_{v}, \Gamma \Rightarrow \Delta$}
\RightLabel{$\nd$}
\UnaryInfC{$\R, w \leq u, u \leq v, w \leq v, \unda \in D_{w}, \Gamma \Rightarrow \Delta$}
\RightLabel{$\trans$}
\UnaryInfC{$\R, w \leq u, u \leq v, \unda \in D_{w}, \Gamma \Rightarrow \Delta$}
\DisplayProof
\end{flushleft}
\begin{flushright}
\AxiomC{$\R, w \leq u, u \leq v, w \leq v, \unda \in D_{w}, \unda \in D_{v}, \Gamma \Rightarrow \Delta$}
\RightLabel{$\wk$}
\dashedLine
\UnaryInfC{$\R, w \leq u, u \leq v, w \leq v, \unda \in D_{w}, \unda \in D_{u}, \unda \in D_{v}, \Gamma \Rightarrow \Delta$}
\RightLabel{IH}
\dashedLine
\UnaryInfC{$\R, w \leq u, u \leq v, \unda \in D_{w}, \unda \in D_{u}, \unda \in D_{v}, \Gamma \Rightarrow \Delta$}
\RightLabel{$\nd$}
\UnaryInfC{$\R, w \leq u, u \leq v, \unda \in D_{w}, \unda \in D_{u}, \Gamma \Rightarrow \Delta$}
\RightLabel{$\nd$}
\UnaryInfC{$\R, w \leq u, u \leq v, \unda \in D_{w}, \Gamma \Rightarrow \Delta$}
\DisplayProof
\end{flushright}

\begin{flushleft}
\AxiomC{$\R, w \leq u, u \leq v, w \leq v, a \in D_{w}, a \in D_{v}, \Gamma \Rightarrow \Delta$}
\RightLabel{$\cd$}
\UnaryInfC{$\R, w \leq u, u \leq v, w \leq v, a \in D_{v}, \Gamma \Rightarrow \Delta$}
\RightLabel{$\trans$}
\UnaryInfC{$\R, w \leq u, u \leq v, a \in D_{v}, \Gamma \Rightarrow \Delta$}
\DisplayProof
\end{flushleft}
\begin{flushright}
\AxiomC{$\R, w \leq u, u \leq v, w \leq v, a \in D_{w}, a \in D_{v}, \Gamma \Rightarrow \Delta$}
\RightLabel{$\wk$}
\dashedLine
\UnaryInfC{$\R, w \leq u, u \leq v, w \leq v, a \in D_{w}, a \in D_{u}, a \in D_{v}, \Gamma \Rightarrow \Delta$}
\RightLabel{IH}
\dashedLine
\UnaryInfC{$\R, w \leq u, u \leq v, a \in D_{w}, a \in D_{u}, a \in D_{v}, \Gamma \Rightarrow \Delta$}
\RightLabel{$\cd$}
\UnaryInfC{$\R, w \leq u, u \leq v, a \in D_{u}, a \in D_{v}, \Gamma \Rightarrow \Delta$}
\RightLabel{$\cd$}
\UnaryInfC{$\R, w \leq u, u \leq v, a \in D_{v}, \Gamma \Rightarrow \Delta$}
\DisplayProof
\end{flushright}

We now argue that the side condition $v \sim^{\R} w$ of $\existscdr$, $\existscedr$, $\allcdl$, and $\allcedl$ continues to hold after applying IH. If the undirected path from $v$ to $w$ does not contain $u \leq z'$, then the side condition trivially holds. If, on the other hand, the undirected path from $v$ to $w$ contains $u \leq z'$, then another undirected path from $v$ to $w$ is found by replacing all occurrences of $u \leq z'$ with $u \leq z, z \leq z'$, which continues to be present after the invocation of IH.
\qed
\end{proof}

\begin{customcor}{\ref{cor:nested-admissibility}}
All rules in $\nstrucset$ are admissible in $\nint$, $\nintfond$ and $\nintfocd$.
\end{customcor}

\begin{proof} We prove the admissibility of each rule of $\nstrucset$ in turn. The first three cases we consider are the $(wk_{l})$, $(wk_{r})$, and $(psub)$ cases, though we omit the $(wk_{r})$ case since it is similar to the $(wk_{l})$ case.
\begin{center}
\begin{tabular}{c c}
\AxiomC{$\Sigma\{X \far Y\}$}
\RightLabel{Thm.~\ref{thm:nested-to-labelled}}
\dashedLine
\UnaryInfC{$\switchtwo(\Sigma\{X \far Y\})$}
\RightLabel{=}
\dottedLine
\UnaryInfC{$\R, \Gamma, w : X \sar w : Y, \Delta$}
\RightLabel{$\wk$}
\dashedLine
\UnaryInfC{$\R', \Gamma, w : X, w : Z \sar w : Y, \Delta$}
\RightLabel{Thm.~\ref{thm:nested-to-labelled}}
\dashedLine
\UnaryInfC{$\switch(\R, \Gamma, w : X, w : Z \sar w : Y, \Delta)$}
\RightLabel{=}
\dottedLine
\UnaryInfC{$\Sigma\{X, Z \far Y\}$}
\DisplayProof

&

\AxiomC{$\Sigma$}
\RightLabel{Thm.~\ref{thm:nested-to-labelled}}
\dashedLine
\UnaryInfC{$\switchtwo(\Sigma)$}
\RightLabel{=}
\dottedLine
\UnaryInfC{$\Lambda$}
\RightLabel{$(psub)$}
\dashedLine
\UnaryInfC{$\Lambda(\unda / \undb)$}
\RightLabel{Thm.~\ref{thm:nested-to-labelled}}
\dashedLine
\UnaryInfC{$\switch(\Lambda(\unda / \undb))$}
\RightLabel{=}
\dottedLine
\UnaryInfC{$\Sigma(\unda/\undb)$}
\DisplayProof
\end{tabular}
\end{center}
We assume that $\R' := \R, \vv{\unda} \in D_{w}$ where $\vv{\unda}$ are all parameters occurring in the labelled formulae of $w : Z$. This ensures that the labelled sequent $\R', \Gamma, w : X, w : Z \sar w : Y, \Delta$ is nestedlike (cf.~Def.~\ref{def:switchtwo} and~\ref{def:nestedlike}). Also, the labelled sequent $\Lambda(\unda/\undb)$ in the $(psub)$ derivation is trivially nestedlike since $\switchtwo(\Sigma(\unda / \undb)) = \Lambda(\unda / \undb)$ (up to a change of labels). This implies that the second invocation of Thm.~\ref{thm:nested-to-labelled} in both cases is valid.

We next consider the $\ctrl$, $\ctrr$, $(nr)$, and $(mrg_{1})$ cases, but omit the proof of $\ctrr$ since it is similar to $\ctrl$. In the top right (i.e. $(nr)$) derivation, we assume that $v$ is the root of $G(\switchtwo(\Sigma))$ and that $w$ is fresh, meaning we effectively change the root from $v$ to $w$ after applying $\wk$. 
\begin{center}
\resizebox{\columnwidth}{!}{
\begin{tabular}{c c}
\AxiomC{$\Sigma\{X,Z,Z \far Y\}$}
\RightLabel{Thm.~\ref{thm:nested-to-labelled}}
\dashedLine
\UnaryInfC{$\switchtwo(\Sigma\{X,Z,Z \far Y\})$}
\RightLabel{=}
\dottedLine
\UnaryInfC{$\R, \Gamma, w : X, w : Z, w : Z \sar w : Y, \Delta$}
\RightLabel{$\ctrl + \ctrrel$}
\dashedLine
\UnaryInfC{$\R, \Gamma, w : X, w : Z \sar w : Y, \Delta$}
\RightLabel{Thm.~\ref{thm:nested-to-labelled}}
\dashedLine
\UnaryInfC{$\switch(\R, \Gamma, w : X, w : Z \sar w : Y, \Delta)$}
\RightLabel{=}
\dottedLine
\UnaryInfC{$\Sigma\{X, Z \far Y\}$}
\DisplayProof

&

\AxiomC{$\Sigma$}
\RightLabel{Thm.~\ref{thm:nested-to-labelled}}
\dashedLine
\UnaryInfC{$\switchtwo(\Sigma)$}
\RightLabel{=}
\dottedLine
\UnaryInfC{$\R,\Gamma \sar \Delta$}
\RightLabel{$\wk$}
\dashedLine
\UnaryInfC{$\R, w \leq v, \Gamma \sar \Delta$}
\RightLabel{Thm.~\ref{thm:nested-to-labelled}}
\dashedLine
\UnaryInfC{$\switch(\R, w \leq v, \Gamma \sar \Delta)$}
\RightLabel{=}
\dottedLine
\UnaryInfC{$\far [\Sigma]$}
\DisplayProof
\end{tabular}
}
\end{center}
\begin{center}
\AxiomC{$\Sigma\{X \far Y, [\Sigma'], [\Sigma']\}$}
\RightLabel{Thm.~\ref{thm:nested-to-labelled}}
\dashedLine
\UnaryInfC{$\switchtwo(\Sigma\{X \far Y, [\Sigma'], [\Sigma']\})$}
\RightLabel{=}
\dottedLine
\UnaryInfC{$\R, \R', \R'', \Gamma, \Gamma', \Gamma'' \sar \Delta', \Delta'', \Delta$}
\RightLabel{$(lsub) \times n$}
\dashedLine
\UnaryInfC{$\R, \R', \R', \Gamma, \Gamma', \Gamma' \sar \Delta', \Delta', \Delta$}
\RightLabel{$\ctrrel + \ctrl + \ctrr$}
\dashedLine
\UnaryInfC{$\R, \R',\Gamma, \Gamma' \sar \Delta', \Delta$}
\RightLabel{Thm.~\ref{thm:nested-to-labelled}}
\dashedLine
\UnaryInfC{$\switch(\R, \R',\Gamma, \Gamma' \sar \Delta', \Delta)$}
\RightLabel{=}
\dottedLine
\UnaryInfC{$\Sigma\{X \far Y, [\Sigma']\}$}
\DisplayProof
\end{center}
The labelled sequents $\R, \Gamma, w : X, w : Z \sar w : Y, \Delta$, $\R, w \leq v, \Gamma \sar \Delta$, and $\R, \R',\Gamma, \Gamma' \sar \Delta', \Delta$ are nestedlike, as the conclusion of each derivation \emph{is} the nested sequent that when input into $\switchtwo$ outputs each such labelled sequent (up to a change of labels). In the bottom (i.e. $(mrg_{1})$) derivation, we assume that the first copy of $\Sigma'$ is translated to $\R',\Gamma' \sar \Delta'$ under $\switchtwo$ and that the second copy of $\Sigma'$ is translated to $\R'',\Gamma'' \sar \Delta''$ under $\switchtwo$. Therefore, the graphs $G(\R', \Gamma' \sar \Delta')$ and $G(\R'', \Gamma'' \sar \Delta'')$ are isomorphic, and so, the $n$ applications of $(lsub)$ correspond to the replacement of labels in $\R'', \Gamma'' \sar \Delta''$ with labels in $\R', \Gamma' \sar \Delta'$ to make these two \emph{identical}. Then, using $\ctrrel$, $\ctrl$, and $\ctrr$, we contract all duplicate copies of relational atoms and labelled formulae to obtain $\R, \R',\Gamma, \Gamma' \sar \Delta', \Delta$.

We prove the admissibility of the $(mrg_{2})$ rule below. We assume that the label $w$ is associated with the component $X_{1} \far Y_{1}$ and the label $v$ is associated with the component $X_{2} \far Y_{2}$. As before, the conclusion of the derivation serves as the nested sequent that when input into $\switchtwo$ yields $\R(w/v), \Gamma(w/v), w : X_{1}, w : X_{2} \sar w : Y_{1}, w : Y_{2}, \Delta(w/v)$ (up to a change of labels), thus showing the validity of our second use of Thm.~\ref{thm:nested-to-labelled}.
\begin{center}
\AxiomC{$\Sigma\{X_{1} \far Y_{1}, [X_{2} \far Y_{2}, [\Sigma_{1}], \ldots, [\Sigma_{n}]]\}$}
\RightLabel{Thm.~\ref{thm:nested-to-labelled}}
\dashedLine
\UnaryInfC{$\switchtwo(\Sigma\{X_{1} \far Y_{1}, [X_{2} \far Y_{2}, [\Sigma_{1}], \ldots, [\Sigma_{n}]]\})$}
\RightLabel{=}
\dottedLine
\UnaryInfC{$\R, w \leq v, \Gamma, w : X_{1}, v : X_{2} \sar w : Y_{1}, v : Y_{2}, \Delta$}
\RightLabel{$(lsub)$}
\dashedLine
\UnaryInfC{$\R(w/v), w \leq w, \Gamma(w/v), w : X_{1}, w : X_{2} \sar w : Y_{1}, w : Y_{2}, \Delta(w/v)$}
\RightLabel{$\refl$}
\UnaryInfC{$\R(w/v), \Gamma(w/v), w : X_{1}, w : X_{2} \sar w : Y_{1}, w : Y_{2}, \Delta(w/v)$}
\RightLabel{Thm.~\ref{thm:nested-to-labelled}}
\dashedLine
\UnaryInfC{$\switch(\R(w/v), \Gamma(w/v), w : X_{1}, w : X_{2} \sar w : Y_{1}, w : Y_{2}, \Delta(w/v))$}
\RightLabel{=}
\dottedLine
\UnaryInfC{$\Sigma\{X_{1}, X_{2} \far Y_{1}, Y_{2}, [\Sigma_{1}], \ldots, [\Sigma_{n}]\}$}
\DisplayProof
\end{center}
Next we consider the  $(ew_{1})$ case. We assume that $\switchtwo(\Sigma') = \R', \Gamma' \sar \Delta'$, which shows that the conclusion of the derivation serves as the nested sequent, which when input into $\switchtwo$ outputs $\R, \R',\Gamma, \Gamma' \sar \Delta', \Delta$ (up to a change of labels). Therefore, the second use of Thm.~\ref{thm:nested-to-labelled} is valid.
\begin{center}
\AxiomC{$\Sigma\{X \far Y\}$}
\RightLabel{Thm.~\ref{thm:nested-to-labelled}}
\dashedLine
\UnaryInfC{$\switchtwo(\Sigma\{X \far Y\})$}
\RightLabel{=}
\dottedLine
\UnaryInfC{$\R, \Gamma \sar \Delta$}
\RightLabel{$\wk$}
\dashedLine
\UnaryInfC{$\R, \R',\Gamma, \Gamma' \sar \Delta', \Delta$}
\RightLabel{Thm.~\ref{thm:nested-to-labelled}}
\dashedLine
\UnaryInfC{$\switch(\R, \R',\Gamma, \Gamma' \sar \Delta', \Delta)$}
\RightLabel{=}
\dottedLine
\UnaryInfC{$\Sigma\{X \far Y, [\Sigma']\}$}
\DisplayProof
\end{center}
In the proof of the admissibility of $(ew_{2})$ below,  we assume that the label $w$ is associated with the component $X_{1} \far Y_{1}$, the label $v$ is associated with the component $X_{2} \far Y_{2}$, and that the label $u$ is fresh. Up to a change of labels, the conclusion of the derivation is the nested sequent that when input into $\switchtwo$ outputs the labelled sequent $\R, w \leq u, u \leq v, \Gamma, w : X_{1}, v : X_{2} \sar w : Y_{1}, v : Y_{2}, \Delta$, justifying our second use of Thm.~\ref{thm:nested-to-labelled}.
\begin{center}
\AxiomC{$\Sigma\{X_{1} \far Y_{1}, [X_{2} \far Y_{2}, [\Sigma_{1}], \ldots, [\Sigma_{n}]]\}$}
\RightLabel{Thm.~\ref{thm:nested-to-labelled}}
\dashedLine
\UnaryInfC{$\switchtwo(\Sigma\{X_{1} \far Y_{1}, [X_{2} \far Y_{2}, [\Sigma_{1}], \ldots, [\Sigma_{n}]]\})$}
\RightLabel{=}
\dottedLine
\UnaryInfC{$\R, w \leq v, \Gamma, w : X_{1}, v : X_{2} \sar w : Y_{1}, v : Y_{2}, \Delta$}
\RightLabel{$\wk$}
\dashedLine
\UnaryInfC{$\R, w \leq u, u \leq v, w \leq v, \Gamma, w : X_{1}, v : X_{2} \sar w : Y_{1}, v : Y_{2}, \Delta$}
\RightLabel{$\trans$}
\UnaryInfC{$\R, w \leq u, u \leq v, \Gamma, w : X_{1}, v : X_{2} \sar w : Y_{1}, v : Y_{2}, \Delta$}
\RightLabel{Thm.~\ref{thm:nested-to-labelled}}
\dashedLine
\UnaryInfC{$\switch(\R, w \leq u, u \leq v, \Gamma, w : X_{1}, v : X_{2} \sar w : Y_{1}, v : Y_{2}, \Delta)$}
\RightLabel{=}
\dottedLine
\UnaryInfC{$\Sigma\{X_{1} \far Y_{1}, [ \ \far [X_{2} \far Y_{2}, [\Sigma_{1}], \ldots, [\Sigma_{n}]]]\}$}
\DisplayProof
\end{center}
Last, we show the admissibility of $(lwr)$ and $\cut$. In the proof of the admissibility of $(lwr)$ below, we assume that the label $w$ is associated with the component $X_{1} \far A, Y_{1}$ and that the label $v$ is associated with the component $X_{2} \far A, Y_{2}$. In the proof of the admissibility of $\cut$, we assume that the label $w$ is associated with $X \far A, Y$ and $X, A \far Y$. As in all previous cases, the end sequent of each derivation serves as the nested sequent (up to a change of labels) that when input into $\switchtwo$ yields the labelled sequent before the (second) invocation of Thm.~\ref{thm:nested-to-labelled}, thus showing that our use of the theorem is valid.
\begin{center}
\resizebox{\columnwidth}{!}{
\AxiomC{$\Sigma\{X_{1} \far A, Y_{1}, [X_{2} \far A, Y_{2}]\}$}
\RightLabel{Thm.~\ref{thm:nested-to-labelled}}
\dashedLine
\UnaryInfC{$\switchtwo(\Sigma\{X_{1} \far A, Y_{1}, [X_{2} \far A, Y_{2}]\})$}
\RightLabel{=}
\dottedLine
\UnaryInfC{$\R, w \leq v, \Gamma \sar w : A, v : A, \Delta$}

\AxiomC{}
\RightLabel{Thm.~\ref{thm:lint-properties}-(i)-(c)}
\dashedLine
\UnaryInfC{$\R, w \leq v, \Gamma, w : A \sar v : A, \Delta$}

\RightLabel{$\cut$}
\dashedLine
\BinaryInfC{$\R, w \leq v, \Gamma \sar v : A \sar \Delta$}
\RightLabel{Thm.~\ref{thm:nested-to-labelled}}
\dashedLine
\UnaryInfC{$\switch(\R, w \leq v, \Gamma \sar v : A, \Delta)$}
\RightLabel{=}
\dottedLine
\UnaryInfC{$\Sigma\{X_{1} \far Y_{1}, [X_{2} \far A, Y_{2}]\}$}
\DisplayProof
}
\end{center}

\begin{center}
\AxiomC{$\Sigma\{X \far A, Y\}$}
\RightLabel{Thm.~\ref{thm:nested-to-labelled}}
\dashedLine
\UnaryInfC{$\switchtwo(\Sigma\{X \far A, Y\})$}
\RightLabel{=}
\dottedLine
\UnaryInfC{$\R, \Gamma \sar w : A,\Delta$}

\AxiomC{$\Sigma\{X, A \far Y\}$}
\RightLabel{Thm.~\ref{thm:nested-to-labelled}}
\dashedLine
\UnaryInfC{$\switchtwo(\Sigma\{X, A \far Y\})$}
\RightLabel{=}
\dottedLine
\UnaryInfC{$\R, \Gamma, w : A \sar \Delta$}

\RightLabel{$\cut$}
\dashedLine
\BinaryInfC{$\R, \Gamma \sar \Delta$}
\RightLabel{Thm.~\ref{thm:nested-to-labelled}}
\dashedLine
\UnaryInfC{$\switch(\R, \Gamma \sar \Delta)$}
\RightLabel{=}
\dottedLine
\UnaryInfC{$\Sigma\{X \far Y\}$}
\DisplayProof
\end{center}
\qed
\end{proof}

\end{document}